\documentclass[a4paper,twoside,11pt,reqno]{amsart}
\newcommand{\Ueberschrift}{Categories of abelian varieties over finite fields II:  \\[1ex] 
Abelian varieties over \texorpdfstring{$\bF_q$}{finite fields}  and Morita equivalence}
\newcommand{\Kurztitel}{Abelian varieties over \texorpdfstring{$\bF_q$}{finite fields}  and Morita equivalence}
\usepackage{amsmath} \usepackage{amssymb} \usepackage{amsopn}
\usepackage{amsthm} \usepackage{amsfonts} \usepackage{mathrsfs}
\usepackage{latexsym}
\usepackage[headings]{fullpage}
\usepackage{mathtools}
\usepackage[all]{xy}
\usepackage[usenames]{color}
\usepackage[OT2, T1]{fontenc}
\usepackage[utf8]{inputenc}
\usepackage{mathabx}
\usepackage{enumitem}
\usepackage{stmaryrd}
\usepackage[pdfpagelabels, pdftex]{hyperref}
\hypersetup{
  pdftitle={},
  pdfauthor={},
  pdfsubject={},
  pdfkeywords={},
  colorlinks=true,    
  linkcolor=red,     
  citecolor=blue,     
  filecolor=blue,      
  urlcolor=blue,       
  breaklinks=true,
  bookmarksopen=true,
  bookmarksnumbered=true,
  pdfpagemode=UseOutlines,
  plainpages=false}
\DeclareMathOperator{\rH}{H}

\DeclareMathOperator{\rM}{M}

\newcommand{\bF}{{\mathbb F}}

\newcommand{\bN}{{\mathbb N}}

\newcommand{\bQ}{{\mathbb Q}}

\newcommand{\bZ}{{\mathbb Z}}

\newcommand{\cD}{{\mathscr D}}

\newcommand{\cF}{{\mathscr F}}

\newcommand{\cS}{{\mathscr S}}

\newcommand{\cV}{{\mathscr V}}

\newcommand{\dA}{{\mathcal A}}

\newcommand{\dO}{{\mathcal O}}

\newcommand{\dR}{{\mathcal R}}
\newcommand{\dS}{{\mathcal S}}

\newcommand{\fO}{{\mathfrak O}}

\newcommand{\fo}{{\mathfrak o}}
\newcommand{\fp}{{\mathfrak p}}

\DeclareSymbolFont{cyrletters}{OT2}{wncyr}{m}{n}
\DeclareMathSymbol{\Sha}{\mathalpha}{cyrletters}{"58}

\newcommand{\surj}{\twoheadrightarrow} 
\newcommand{\inj}{\hookrightarrow}
\DeclareMathOperator{\id}{id}
\DeclareMathOperator{\pr}{pr}

\DeclareMathOperator{\Hom}{Hom}

\DeclareMathOperator{\End}{End}

\DeclareMathOperator{\im}{im}

\DeclareMathOperator{\AV}{{\sf AV}}

\DeclareMathOperator{\Mod}{Mod}

\DeclareMathOperator{\Spec}{Spec}

\DeclareMathOperator{\Br}{Br}
\newcommand{\inv}{{\rm inv}}
\DeclareMathOperator{\Frob}{Frob}

\DeclareMathOperator{\ord}{ord}

\DeclareMathOperator{\Ext}{Ext}

\DeclareMathOperator{\Gal}{Gal}

\newcommand{\expo}{r}
\newcommand{\otherw}{{w'}}
\newcommand{\ellplace}{v}
\newcommand{\pplace}{v}

\newcommand{\ep}{\varepsilon}
\newcommand{\ph}{\varphi}

\newcommand{\com}{{\rm com}}

\newcommand{\op}{{\rm op}}

\DeclarePairedDelimiter\abs{\lvert}{\rvert}

\DeclareMathOperator{\lcm}{lcm}
\DeclareMathOperator{\rank}{rk}

\newcommand{\Ztf}{{\bZ\text{\rm -tf}}}
\newcommand{\Zltf}{{\bZ_\ell\text{\rm -tf}}}
\newcommand{\Zptf}{{\bZ_p\text{\rm -tf}}}

\newtheorem{thm}{Theorem}[section]

\newtheorem{prop}[thm]{Proposition}
\newtheorem{lem}[thm]{Lemma}

\newtheorem{cor}[thm]{Corollary}

\theoremstyle{definition}
\newtheorem{defi}[thm]{Definition}

\theoremstyle{remark}
\newtheorem{rmk}[thm]{Remark}

\newtheorem{ex}[thm]{Example}

\newenvironment{pro*}[1][Proof]{{\it{#1:}} }{}
\newenvironment{pro**}[1][]{{\it{#1}} }{\hfill $\square$}

\numberwithin{equation}{section}
\begin{document}

\hrule width\hsize

\vskip 0.5cm

\title[\Kurztitel]{\Ueberschrift} 
\author{Tommaso Giorgio Centeleghe}
\address{Tommaso Giorgio Centeleghe, 69120 Heidelberg, Germany}
\email{tommaso.centeleghe@gmail.com}

\author{Jakob Stix}
\address{Jakob Stix, Institut f\"ur Mathematik, 
Goethe--Universit\"at Frankfurt, Robert-Mayer-Stra\ss e {6--8},
60325~Frankfurt am Main, Germany}
\email{stix@math.uni-frankfurt.de}

\thanks{The second author acknowledges support by Deutsche  Forschungsgemeinschaft  (DFG) through the Collaborative Research Centre TRR 326 "Geometry and Arithmetic of Uniformized Structures", project number 444845124.}

\date{May 9, 2022} 
\dedicatory{dedicated to Moshe Jarden on the occasion of his 80th birthday} 

\maketitle

\begin{quotation} 
\noindent \small {\bf Abstract} --- The category of abelian varieties over 
$\bF_q$ is shown to be anti-equivalent to a category of $\bZ$-lattices that are 
modules for a non-commutative pro-ring of endo\-morphisms of a suitably chosen direct system of  abelian varieties over $\bF_q$.  On full subcategories cut out by a finite set $w$ of conjugacy classes of Weil $q$-numbers, the anti-equivalence is represented by what we call \emph{$w$-locally projective} abelian varieties.
\end{quotation}

\setcounter{tocdepth}{1} {\scriptsize \tableofcontents}

\section{Introduction}
\label{sec:intro}

\subsection{The scene} 
\label{sec:introduction_categories_of_reflexive_modules}
Let $q=p^\expo$ be a power of a prime number $p$, and let $\bF_q$ be a finite field with $q$ elements.
In this paper we generalize \cite{CS:part1} to the category $\AV_{\bF_q}$ of abelian varieties over $\bF_q$. Our main result says that there is a non-commutative pro-ring $\dS_q$ and an anti-equivalence
\begin{equation}\label{eq:mainresult}
T:{\AV_{\bF_q}} \longrightarrow \Mod_{\Ztf}(\dS_q)
\end{equation}
between $\AV_{\bF_q}$ and the category of left $\dS_q$-modules that are free and of finite rank over $\bZ$.
The ring $\dS_q$ arises in the construction of \eqref{eq:mainresult} and admits a description as a pro-ring
\[
\dS_q = \varprojlim_w S_w
\]
for $\bZ$-orders $S_w$ of certain finite $\bQ$-algebras. The construction of
$T$ involves several choices. As a consequence $\dS_q$ is not unique, however, its center $\dR_q$ is unique and can be described explicitly in terms of the set of Weil $q$-numbers.  Besides having its own  interest, our result provides a non-commutative algebra into which potentially every question on $\AV_{\bF_q}$ can be translated. 

It will be clear to the reader that our work is permeated by Honda--Tate theory \cite{Honda, Tate:bourbakiHondaTate} and by Tate's result on the local structure of the $\Hom$-groups \cite{Tate:endomorphisms}. 
Since these milestones for the subject were placed in the 1960s,
variants of the anti-equivalence \eqref{eq:mainresult} for several subcategories of $\AV_{\bF_q}$ have been studied by many mathematicians, with reviving interest in recent years.
We recall  results on this classical topic in a similar spirit as ours, and apologize in advance for possible omissions.
Waterhouse \cite{Wa} studied isomorphism classes and endomorphism rings of certain simple objects of $\AV_{\bF_q}$. More recently,
Yu \cite[Theorem 3.1]{Yu} classified isomorphism classes of objects of $\AV_{\bF_p}$ whose characteristic polynomial of Frobenius is relatively prime to $x^2-p$. 
Giraud \cite[\S1]{Giraud:Shimura} and Waterhouse \cite[Appendix]{Wa} (see also \cite[\S4]{JKPRSBT}) exploited a functor due to Serre and Tate in the converse direction as \eqref{eq:mainresult} to study elliptic curves.
Deligne \cite{De} used the canonical lifting of Serre and Tate to give a complete description of the full subcategory of $\AV_{\bF_q}$ consisting of ordinary objects.
This construction was further developed by Howe \cite{Howe2} to solve questions concerning polarizations,
a topic also addressed in a similar fashion by Bergstr\"om, Karemaker and Marseglia \cite{BKM}.
The canonical lifting technique was extended by Oswal and Shankar \cite{OS} and used to classify objects in isogeny classes of simple, almost ordinary abelian varieties. 
Classification of abelian varieties isogenous to powers of a given simple abelian variety (often an elliptic curve)
have been the focus of attention in \cite{Kani}, \cite{Yu}, \cite{JKPRSBT}. 

Our emphasis is on a uniform categorical description of all of $\AV_{\bF_q}$ in terms of modules.

\subsection{Frobenius and Weil numbers}\label{subsection:notation}
Before giving more details on our method and stating a precise theorem we recall some notation and terminology.
Two Weil $q$-numbers $\pi$ and $\pi'$ are \emph{conjugate} if there is an isomorphism
$\bQ(\pi) \simeq \bQ(\pi')$ sending $\pi$ to $\pi'$. The set of conjugacy classes of Weil $q$-numbers is denoted by $W_q$.
When no confusion is likely to arise, we may suppress the distinction between a Weil number and its conjugacy class. For any $\pi \in W_q$, we fix once and for all an $\bF_q$-simple abelian variety $B_\pi$ belonging to the isogeny class of $\AV_{\bF_q}$ defined by $\pi$ according to Honda--Tate theory \cite{Honda, Tate:bourbakiHondaTate}.

The \emph{Weil support} $w(A)$ of an object $A$ of $\AV_{\bF_q}$ is the finite subset of $W_q$ consisting of
Weil numbers associated to the simple factors of $A$. This is to say that there is an $\bF_q$-isogeny
\[
A\longrightarrow\prod_{\pi\in w(A)}B_\pi^{n_\pi},
\]
with positive multiplicities $n_\pi$.
For a subset $w\subseteq W_q$, finite or infinite, we denote by $\AV_w$ the full subcategory of $\AV_{\bF_q}$ whose objects are those
abelian varieties $A$ such that $w(A)\subseteq w$.

For $w\subseteq W_q$ finite, the \emph{minimal central order $R_w$} is the largest quotient through which the ring homomorphism
\begin{equation}\label{FV-q}
\bZ[F,V]/(FV-q)\longrightarrow \bQ(w) := \prod_{\pi\in w}\bQ(\pi),
\end{equation}
sending $F$ and $V$ to the diagonal images of $\pi$ and $q/\pi$ respectively, factors. 
As $w$ ranges through the finite subsets of $W_q$, the $R_w$ form a pro-ring
$\dR_q=(R_w,r_{w,\otherw})$,
where the transition maps $r_{w,\otherw}:R_\otherw \to R_w$ are the natural surjections, defined when $w\subseteq \otherw$.
The category $\AV_w$ has a natural $R_w$-linear structure which varies compatibly as $w$ increases. This datum forms what we refer to as the
$\dR_q$-linear structure on $\AV_{\bF_q}$. 

\subsection{The main result}

The strategy for constructing the anti-equivalence \eqref{eq:mainresult} is similar to that followed in \cite{CS:part1}. We first construct, for any finite subset $w\subseteq W_q$, what we call a \emph{$w$-balanced} abelian variety $A_w$, see  Definition~\ref{defi:balanced}. We then show that $A_w$ represents  an anti-equivalence
\begin{equation}
\label{eq:functorTwIntro}
T_w \colon \AV_w \longrightarrow \Mod_{\Ztf}(S_w), \qquad T_w(X) =  \Hom_{\bF_q}(X,A_w),
\end{equation}
where $S_w = \End_{\bF_q}(A_w)$ and $\Mod_{\Ztf}(S_w)$ is the category of left $S_w$-modules that are 
free and of finite rank over $\bZ$. 
The $w$-balanced object $A_w$ and hence $S_w$ is not unique.
However the center of $S_w$ is isomorphic to $R_w$ under the map induced by the $R_w$-linear structure of $\AV_w$.

The second step is to show that the objects $A_w$ can be chosen compatibly
as $w$ increases. More precisely, we prove that there is a family of $w$-balanced abelian varieties $A_w$ and an ind-object $\mathcal{A}=(A_w, \ph_{w,\otherw})$ such that the corresponding ind-representable functor
\[
\Hom_{\bF_q}(-,\mathcal{A})=\varinjlim_w\Hom_{\bF_q}(-,A_w)
\]
interpolates the anti-equivalences \eqref{eq:functorTwIntro} for each $w$. An important by-product of the construction of $\mathcal{A}$
is that the endomorphism rings $S_w$ form a projective system
\[
\dS_q = \End_{\bF_q}(\mathcal{A}) = (S_w, s_{w,\otherw})
\]
with $r_{w,\otherw}$-linear and surjective transition maps $s_{w,\otherw}:S_\otherw \to S_w$, for $w\subseteq \otherw$.
Denote by 
\[
\Mod_{\Ztf}(\dS_q)
\]
the category of $\dS_q$-modules for which the structural action factors through some $S_w$, and which as $\bZ$-modules
are free of finite rank. This category is $\dR_q$-linear and for each object there is a clear notion of support, parallel to that in $\AV_q$.
Our main result can be formulated as follows. 

\begin{thm}
\label{MainThm} 
Let $q=p^\expo$ be a power of a prime number $p$. There exists an ind-abelian variety $\dA = (A_w,\ph_{w,\otherw})$ such that $A_w$ is $w$-balanced for all finite $w \subseteq W_q$ and the transition maps are inclusions. With $\dS_q = \End_{\bF_q}(\dA)$, the contravariant and  $\dR_q$-linear functor
\[
T \colon \AV_{\bF_q} \longrightarrow  \Mod_{\Ztf}(\dS_q), \qquad  T(X) = \Hom_{\bF_q}(X,\mathcal{A})
\]
is an anti-equivalence of categories which preserves the support.
Moreover, the $\bZ$-rank of $T(X)$ 
is equal to $4\expo\dim(X)$. 
\end{thm}

\begin{rmk} 
\begin{enumerate}[label=(\arabic*),align=left,labelindent=0pt,leftmargin=*,widest = (8)]
\item
The main difference with \cite{CS:part1} is that the varieties $A_w$ can no longer be chosen to be multiplicity free
in general (see Theorem~\ref{thm:connectedSpec classification of locallyprojective}), and so  their endomorphism rings $S_w$ are non--commutative, hence harder to describe explicitly, see Theorem~\ref{thm:Sw}. 
We consider the use of multiplicities as a major insight which is stimulated  by and used here in the context of Morita equivalence. This explains how we decided the title of this note.

\item 
The rank of the associated $\bZ$-lattice $T(X)$ exceeds the dimension of $\rH^1(X)$ in any Weil cohomology by a factor of $2\expo$.
This cannot be avoided in general, see Proposition~\ref{prop:constantC}. However the ratio $\rank_\bZ(T(X))/2\dim(X)$ can be lowered
to $\expo$ if $\expo$ is even or by excluding real Weil numbers from the scope of the anti-equivalence, see Remark~\ref{rmk:equiv_reducedbalanced}.

\item
The rational version of Theorem~\ref{MainThm} follows easily from Honda--Tate theory. 
Here rational means $-\otimes \bQ$ for modules and working up to isogeny for abelian varieties. In fact, Theorem~\ref{MainThm} can be seen as a categorification of an integral version of Honda--Tate theory.
\end{enumerate}
\end{rmk}

\subsection{Geometry of the category of abelian varieties}
\label{subsection:geometry} 
Let $w$ be a finite set of Weil $q$-numbers. Restricted to $\AV_w$, our main result has a geometric interpretation on the $1$-dimensional scheme $X_w = \Spec(R_w)$. The $R_w$-algebra $S_w$ corresponds to a coherent sheaf of (in general) non-commutative $\dO_{X_w}$-algebras $\cS_w$. Finitely generated modules for $S_w$ are coherent sheaves on $X_w$ endowed with an $\dO_{X_w}$-linear  $\cS_w$-module structure.  We call these coherent $\cS_w$-modules. The anti-equivalence~\eqref{eq:functorTwIntro} describes abelian varieties from $\AV_w$ as coherents sheaves of $\cS_w$-modules that are flat over $\Spec(\bZ)$.

With this point of view, it makes sense to talk about \textit{categorically local properties} of an abelian variety $X$ in $\AV_w$ at a prime $\ell \not=p$ (or at $p$) as properties of the $\ell$-adic (or $p$-adic) completion 
of the module $\Hom_{\bF_q}(X,A_w)$. With hindsight, one should be able to detect local structure of $X$ in a more direct way.  Indeed, Tate's theorems 
identify local structure of $X$ in terms of the Galois module $T_\ell(X)$ (resp.\ the Dieudonn\'e module $T_p(X)$). The proof of our main result compares the datum of the Tate module and the completion of  $\Hom_{\bF_q}(X,A_w)$ based on the special properties of $w$-balanced objects. We therefore adopt the terminology of \emph{local} properties of abelian varieties if these are defined in terms of Tate modules only (see Definition~\ref{defi:localTate modules} 
for a refinement which indicates that the present point of view is only semi-local).

The following definition is an important example of a local property. We use the notation $\cD_w$  to denote a certain quotient of the Dieudonn\'e ring to be explained in \eqref{eq:definitionDw} of Section~\ref{sec:mincenord}.

\begin{defi} 
\label{defi:locallyprojective} 
Let $q=p^\expo$ be a power of a prime number $p$, and let $w \subseteq W_q$ be a finite subset of Weil $q$-numbers. 
An abelian variety $A \in \AV_w$ is \emph{$w$-locally projective} 
if 
\begin{enumerate}[label=(\roman*),align=left,labelindent=0pt,leftmargin=*,widest = (iii)]
\item
for all $\ell \not= p$ the Tate module $T_\ell(A)$ is a projective $R_w \otimes \bZ_\ell$-module, and 
\item
$T_p(A)$ is a projective $\cD_{w}$-module.
\end{enumerate}
\end{defi}

The $w$-balanced abelian varieties, constructed in Theorem~\ref{thm:integralstructure}, are $w$-locally projective. The  $w$-locally projective objects of $\AV_w$ can be characterized as follows.

\begin{thm} 
\label{thm:truncatedfullyfaithful Intro}
Let $w\subseteq W_q$ be a finite subset, and let $A$ in $\AV_w$ be an abelian variety with $S = \End_{\bF_q}(A)$. Then the following are equivalent.
\begin{enumerate}[label=(\alph*),align=left,labelindent=0pt,leftmargin=*,widest = (m)]
\item 
\label{thmitem:wlocallyprojective intro}
$A$ is $w$-locally projective with support $w(A) = w$. 
\item
\label{thmitem:equivalence intro}
The functor $\Hom_{\bF_q}(-,A) \colon  \AV_w  \longrightarrow \Mod_\Ztf(S)$ is an anti-equivalence of categories. 
\end{enumerate}
\end{thm}

This result will follow from Theorem~\ref{thm:truncatedfullyfaithful}. More on the classification of $w$-locally projective abelian varieties is explained by Theorem~\ref{thm:connectedSpec classification of locallyprojective} and  Theorem~\ref{thm:classification wlocallyprojectiveordinary}.

\begin{rmk}
While \emph{$w$-locally projective} is the decisive property for the representable anti-equivalence of $\AV_w$ with a suitable category of modules (see Theorem~\ref{thm:truncatedfullyfaithful}), we have formulated our main result Theorem~\ref{MainThm} in the more restrictive setting using \emph{$w$-balanced} abelian varieties. 
The \emph{$w$-balanced} abelian varieties are easy-to-construct (see Theorem~\ref{thm:integralstructure}) examples of $w$-locally projective
objects which fit well together to form an ind-object suitable to prove our main result, see Section~\S\ref{sec:proof}. Furthermore, (reduced) $w$-balanced abelian varieties enjoy a minimality property
among the larger set of \emph{$w$-locally projective} ones, as long as $w$ is large enough and $\Spec(R_w)$ is connected (see Theorem~\ref{thm:connectedSpec classification of locallyprojective}).
\end{rmk}

\subsection{The commutative case} 
\label{subsection:subcategories} 
Our method can be applied also to the study of any full subcategory of $\AV_{\bF_q}$ of the form
$\AV_W$, where $W\subseteq W_q$ is any subset. 
The outcome is the existence of a pro-ring $\dS_W$ and an ind-representable anti-equivalence
\begin{equation}\label{eq:submainresult}
T_W:{\AV_W} \longrightarrow \Mod_{\Ztf}(\dS_W)
\end{equation}
between $\AV_W$ and the category $\Mod_{\Ztf}(\dS_W)$ of modules over $\dS_W$ that are finite and torsion free over $\bZ$.
As in the case where $W=W_q$, the functor $T_W$ and the pro-ring $\dS_W$ appearing in \eqref{eq:submainresult} are far from being unique.
The center of any such $\dS_W$ is the quotient $\dR_W$ of $\dR_q$ defined by $W$.

We say that a subset $W \subseteq W_q$ is a \emph{commutative} set of Weil numbers if there is a functor $T_W(-)$ as in \eqref{eq:submainresult} for which $\dS_W$ is commutative. There are two examples:
\begin{itemize}
\item
the subset $W_q^{\ord}$ of ordinary Weil $q$-numbers and, 
\item 
for $q=p$, the set $W_p^{\com}=W_p\setminus \{\pm\sqrt p\}$ given by
the complement of the real conjugacy class of Weil $p$-numbers, see \cite{CS:part1} 
and Section~\S\ref{sec:proof}. 
\end{itemize}
The existence of \eqref{eq:submainresult} for $W=W_q^{\ord}$ was first shown by Deligne in \cite{De}
(up to switching the variance of $T_W$), who exploited the existence of the Serre-Tate canonical lifting for ordinary objects. In \S\ref{sec:ordinary}, see Theorem~\ref{thm:reprovedeligne}, we reprove his result without involving lifts to characteristic zero. Instead, our argument relies on a calculation of an endomorphism ring. 
Then a Morita-equivalence trick deduces the claimed anti-equivalence of categories. 

These two cases $W_q^{\ord}$ and $W_p^\com$ are the only maximal commutative $W$, 
see Proposition~\ref{prop:commutative cases}, and in this case 
the anti-equivalence \eqref{eq:submainresult} reads as
\begin{equation}\label{eq:submainresultcom}
T_W:{\AV_W} \longrightarrow \Mod_{\Ztf}(\dR_W)
\end{equation}
with  the rank of $T_W(X)$ being equal to $2\dim(X)$.

\subsection{Outline}
This article is structured as follows. In Section~\S\ref{sec:rationaltheory} we recall Honda--Tate theory which can be considered as the rational version 
of our anti-equivalence of categories. This leads directly to the existence of $w$-balanced abelian varieties in Theorem~\ref{thm:integralstructure}, a choice of integral structure within the rational theory. The study of the local integral theory starts in Section~\S\ref{sec:truely local} with the definition of the local Tate modules $T_\lambda(A)$ (resp.\ $T_\fp(A)$) for maximal ideals $\lambda$ (resp.\ $\fp$) of $R_w$.
We moreover compute in Section~\S\ref{sec:balanced} in Theorem~\ref{thm:Sw} a description of the endomorphism ring of balanced abelian varieties and thus make the main theorem more concrete and potentially accessible to computations.

In Section~\S\ref{sec:representablefunctors} we provide a formal criterion 
(Morita equivalence, Theorem~\ref{thm:formal}) for a representable functor (in our setting)
to be an anti-equivalence. We subsequently apply the criterion in Section~\S\ref{sec:truncequiv} to $w$-locally projective 
abelian varieties, in particular to $w$-balanced abelian varieties: these are injective cogenerators for the respective $w$-truncated category of abelian varieties $\AV_w$. It should be emphasized that this step relies on Tate's theorems \eqref{TateThmell} and \eqref{TateThmp}, and, for the prime-to-$p$ part, on the minimal central order $R_w$ being Gorenstein. 
We moreover determine the center of $\dS_q$ in Section~\ref{subsec:center}. Section~\S\ref{sec:compatibleantiequivalences} glues the truncated anti-equivalences by a limit argument and so completes the proof of Theorem~\ref{MainThm}.

Section~\S\ref{sec:lowermult} adresses special subcategories for which the multiplicity of the representing object can be lowered, and we show in Proposition~\ref{prop:constantC}  that multiplicities are unavoidable in general. 

We conclude in Section~\S\ref{sec:multiplicity} with an instructive example that shows why multiplicity for the injective cogenerator is important (there are two non-isomorphic simple objects) and that the cogenerator is not just a product of simple objects but involves a congruence (the product has a cyclic isogeny onto the injective cogenerator). Moreover, we classify isogeny classes of injective cogenerators for Weil numbers of supersingular elliptic curves.

\subsection*{Acknowledgements}
We are grateful to the anonymous referee for numerous suggestions that helped improve the presentation of our results.

\section{Rational theory and choice of a lattice} 
\label{sec:rationaltheory}

We recall Honda--Tate theory in order to fix notation and to prepare for Theorem~\ref{thm:integralstructure}.

\subsection{Minimal central orders}
\label{sec:mincenord}
For a Weil $q$-number $\pi$, denote by $P_\pi(x)$ its monic minimal polynomial over $\bQ$, and by
$h_\pi(F,V)$ the associated \emph{symmetric polynomial}, which  is uniquely determined by the equation
\[
h_\pi(x,q/x)  =  x^{-\deg(P_\pi)/2} \cdot P_\pi(x)
\]
and being a linear combination of powers $F^i$ and $V^j$ with $i,j  \in \frac{1}{2}\bZ$ and $i,j \geq 0$. The symmetric polynomial $h_{\pi}(F,V)$ lies in the polynomial ring $\bZ[F,V]$ if $\pi$ is not rational, as then $\deg(P_\pi)$ is even. If $\pi$ is rational, then $\expo$ is even and $\pi = \ep \sqrt{q}$ with $\ep \in \{1, -1\}$ and the positive rational square root $\sqrt{q}$. The symmetric polynomial then is $h_\pi(F,V) = F^{1/2} - \ep V^{1/2}$ and lies in $\bZ[F^{1/2}, V^{1/2}]$. 

For a finite subset $w\subseteq W_q$, we define 
\[
P_w(x)=\prod_{\pi\in w}P_\pi(x) \quad \text{ and } \quad h_w(F,V)=\prod_{\pi\in w}h_\pi(F,V).
\]
We moreover define the degree of $w$ as 
\[
\deg(w) = \deg(P_w) = [\bQ(w):\bQ]  = \sum_{\pi \in w} \deg(\pi).
\]
The degree $\deg(w)$ is even unless $\expo = [\bF_q:\bF_p]$ is even and $w$ contains exactly one of the two rational Weil $q$-numbers $\sqrt{q}$ and $-\sqrt{q}$. 

Let us recall the structure of the minimal central order $R_w$ from \cite[\S2.5]{CS:part1}.  If $w$ either contains both or no rational Weil $q$-numbers, i.e.\ if $\deg(w) = 2d$ is even,  
then the natural map induces an isomorphism
\[
\bZ[F,V]/(FV-q, h_w(F,V)) \stackrel{\sim}{\longrightarrow} R_w,
\]
and shows that we have a $\bZ$-baisis of $R_w$ by the elements represented by 
\begin{equation}
\label{eq:Z basis Rw even}
F^d, \ldots, F,1,V, \ldots, V^{d-1} \quad \text{  or alternatively by } \quad F^{d-1}, \ldots, F,1,V, \ldots, V^{d}.
\end{equation}

If, on the other hand,  we have $w$ equal to $w_0 \cup \{\ep \cdot p^m\}$, where  $\expo=2m$ is even, $\ep \in\{1,-1\}$, and $w_0$ contains no rational Weil $q$-number, then the degree $\deg(w) = 2d_0 + 1$ is odd.  Moreover, we have an isomorphism
\[
\bZ[F, V]/(FV-q, \ h_{w_0}(F,V)(F-\ep p^m), \ h_{w_0}(F,V)(V-\ep p^m)) \stackrel{\sim}{\longrightarrow}  R_w,
\]
and a $\bZ$-basis of $R_w$ is given by the elements represented by 
\begin{equation}
\label{eq:Z basis Rw odd}
F^{d_0}, \ldots, F, 1 , V, \ldots, V^{d_0}.
\end{equation}

\subsection{Localizations of \texorpdfstring{$\AV_{\bF_q}$}{abelian varieties over the field with q elements}}\label{subsec:localizations}

Let $\ell$ be a prime number different from $p$. As usual, for an object $A$ of $\AV_{\bF_q}$ we denote the $\ell$-adic Tate module of $A$ by 
\[
T_\ell(A),
\]
and by $V_\ell(A)$ the $\bQ_\ell$-vector space $T_\ell(A)\otimes_{\bZ_\ell}\bQ_\ell$. 
For any finite subset $w\subseteq W_q$, and any
$A$ in $\AV_w$, the $R_w$-linear structure on $\AV_w$ induces on $T_\ell(A)$ the structure of an $R_w\otimes\bZ_\ell$-module. The ring $R_w\otimes\bZ_\ell$
is generated over $\bZ_\ell$ by image of $F$, since $\ell\neq p$, and we have
\begin{equation}
\label{eq:RwtensorZell}
R_w\otimes\bZ_\ell\simeq\bZ_\ell[x]/\big(P_w(x)\big).
\end{equation}
Let $\overline\bF_p$ be a fixed algebraic closure of the prime field $\bF_p$, and consider $\bF_q$ as a subfield of $\overline\bF_p$. The Galois action on $T_\ell(A)$  can be recovered from the $R_w\otimes\bZ_\ell$-structure, in that $F$ acts on $T_\ell(A)$
as the arithmetic Frobenius $\Frob_q\in\Gal(\overline\bF_p/\bF_q)$. 
 In particular, maps as Galois representations between $\ell$-adic Tate modules for objects in $\AV_w$ are the same as $R_w \otimes \bZ_\ell$-module homomorphisms.
The theorem of Tate \cite[Main Theorem]{Tate:endomorphisms} 
thus says that the functor $T_\ell(-)$ induces an isomorphism
\begin{equation}\label{TateThmell}
\Hom_{\bF_q}(A,B)\otimes\bZ_\ell\stackrel{\sim}{\longrightarrow}\Hom_{R_w\otimes\bZ_\ell}(T_\ell(A), T_\ell(B)),
\end{equation}
for all $A, B$ in $\AV_w$.

\smallskip

To study the localization of $\AV_{\bF_q}$ at $p$, consider the Dieudonn\'e ring
\[
\cD_q = W(\bF_q)\{\cF,\cV\}/(\cF \cV - p).
\]
Here $W(\bF_q)$ is the ring of Witt vectors of $\bF_q$, the variable $\cF$ is $\sigma$-linear and $\cV$ is $\sigma^{-1}$-linear
where $\sigma$ is the arithmetic Frobenius of $W(\bF_q)/\bZ_p$. Moreover, $\cF$ and $\cV$ commute with each other. We introduce the notation 
\[
\cF_i \coloneqq \begin{cases}
\cF^i & i > 0, \\
1 & i = 0, \\
\cV^{-i} & i < 0. 
\end{cases}
\]
The ring $\cD_q$ is a free $W(\bF_q)$-module with basis $(\cF_i)_{i \in \bZ}$, and
\[
\bZ_p[F,V]/(FV-q) \longrightarrow \cD_q, \qquad F \mapsto \cF^\expo \text{ and } V \mapsto \cV^\expo
\]
is an isomorphism onto the center of $\cD_q$. The contravariant Dieudonn\'e module 
\[
T_p(A)
\]
of an object $A$ of $\AV_{\bF_q}$ is a left $\cD_q$-module which is free of rank $2\dim(A)$ as a $W(\bF_q)$-module.
The $\bQ_p$-vector space  $T_p(A)\otimes_{\bZ_p} \bQ_p$ will be denoted by $V_p(A)$.
For any finite subset $w\subseteq W_q$ and any $A$ in $\AV_w$, just like in the $\ell\neq p$ case, we deduce an $R_w\otimes\bZ_p$-structure on $T_p(A)$
such that $F$ acts on $T_p(A)$ as the central element $\cF^\expo\in\cD_q$ and $V$ as $\cV^\expo$. 

Let $\deg(w) = 2d$ be even, i.e. $w$ contains either both rational Weil $q$-numbers or none. Then $\cD_q$ acts on $T_p(A)$ via its quotient
\begin{equation}
\label{eq:definitionDw}
\cD_w \coloneqq W(\bF_q)\{\cF,\cV\}/(\cF \cV - p, \ h_w(\cF^\expo,\cV^\expo)),
\end{equation}
which is a free $W(\bF_q)$-module with basis by the elements represented by 
\[
\cF_{-d\expo}, \ldots, \cF_{d\expo - 1} \quad \text{ or alternatively } \quad \cF_{-d\expo + 1}, \ldots, \cF_{d\expo}.
\]
Let now $\deg(w) = 2d_0  +1$ be odd, i.e., $w$ is equal to $w_0 \cup\{\ep p^m\}$, where  $\expo=2m$ is even, $\ep \in\{1,-1\}$, and $w_0$ contains no rational Weil $q$-number. Then $\cD_q$ acts on $T_p(A)$ via its quotient
\[
\cD_w \coloneqq W(\bF_q)\{\cF,\cV\}/\left(
\cF \cV - p, \ h_{w_0}(\cF^\expo,\cV^\expo)(\cF^{m} - \ep \cV^m)
\right)
\]
which is a free $W(\bF_q)$-module with basis by the elements represented by 
\begin{equation}
\label{eq:W basis Dw}
\cF_{-(2d_0+1)m}, \ldots, \cF_{(2d_0+1)m - 1} \quad \text{ or alternatively } \quad \cF_{-(2d_0+1)m + 1}, \ldots, \cF_{(2d_0+1)m}.
\end{equation}
We remark that $(\cF^m - \ep \cV^m) \cdot \cF^m = F - \sqrt{q}$ and similarly $(\cF^m - \ep \cV^m) \cdot \cV^m = -\ep (V - \sqrt{q})$.

In both cases, even and odd degree $\deg(w)$, we find a natural map 
\[
R_w \otimes \bZ_p \longrightarrow \cD_w, \qquad F \mapsto \cF^\expo \text{ and } V \mapsto \cV^\expo. 
\]
\begin{lem}
\label{lem:centerDw}
The natural map  $R_w \otimes \bZ_p \to \cD_w$ identifies $R_w \otimes \bZ_p$ with the center of $\cD_w$. 
\end{lem}
\begin{proof}
Injectivity follows by comparing the $\bZ_p$-basis of $R_w \otimes \bZ_p$ (arising from the $\bZ$-basis of $R_w$ described in \eqref{eq:Z basis Rw even} and \eqref{eq:Z basis Rw odd}) with the $W(\bF_q)$-basis of $\cD_w$  described in \eqref{eq:W basis Dw}. The image is also clearly contained in the center since the $\expo$-th power of the Frobeniuis $\sigma$ of $W(\bF_q)$ equals the identity.

The number $N= \deg(w)\expo/2$ is an integer and $\cF_i$ for $-N \leq i < N$ form a $W(\bF_q)$-basis of $\cD_w$. Let $x = \sum_{i = -N}^{N-1} x_i \cF_i$ be an element of the center of $\cD_w$. For an arbitrary $a \in W(\bF_q)$ we have
\[
\sum_{i = -N}^{N-1} a x_i \cF_i = ax = xa = \sum_{i = -N}^{N-1} x_i  \sigma^i(a) \cF_i. 
\]
It follows that $x_i = 0$ for all $i$ not divisible by $\expo$. Next, let $m_i = 1$ for $i< 0$ and $m_i = 0$ for $i \geq 0$. Then the equation $\cF x = x \cF$ yields
\[
\sum_{i = -N}^{N-1} \sigma(x_i) p^{m_i} \cF_{i+1} = \sum_{i = -N}^{N-1} x_i p^{m_i} \cF_{i+1}.
\]
Since the  $\cF_i$ for $-N <  i \leq  N$ also form a $W(\bF_q)$-basis of $\cD_w$, we find that all $x_i$ are in 
$\bZ_p$. 
Hence $x$ lies in the image of $R_w \otimes \bZ_p \to \cD_w$. 
\end{proof}

Using the language of Dieudonn\'e modules, Tate's theorem \cite[Part II Theorem 1]{WM} says that the natural map gives an isomorphism
\begin{equation}\label{TateThmp}
\Hom_{\bF_q}(A, B)\otimes\bZ_p\stackrel{\sim}{\longrightarrow}\Hom_{\cD_w}(T_p(B), T_p(A))
\end{equation}
for all $A, B$ in $\AV_w$.

\subsection{Rational local module structure}
For any $\pi\in W_q$, recall that $B_\pi$ is an $\bF_q$-simple abelian variety associated to $\pi$ via Honda--Tate theory. The ring
\[
E_\pi=\End_{\bF_q}(B_\pi)\otimes\bQ
\]
of endomorphisms of $B_\pi$ in the isogeny category of abelian varieties over $\bF_q$ is a central division ring over the subfield $\bQ(\pi)$. The index of $E_\pi$ will be denoted by
$s_\pi$, so that $[E_\pi:\bQ(\pi)]=s_\pi^2$. While $E_\pi$ is determined by $\pi$ up to isomorphism, it is well known that its order $\End_{\bF_q}(B_\pi)$ is not
an isogeny invariant in general. The structure of $E_\pi$ as a central simple algebra over $\bQ(\pi)$ is determined by its local invariants. 
\begin{thm}[Tate {\cite[Theorem 1]{Tate:bourbakiHondaTate}}]
\label{thm:TateEndostructure}
Let $\pi$ be a Weil $q$-number. 
\begin{enumerate}[label=(\arabic*),align=left,labelindent=0pt,leftmargin=*,widest = (8)]
\item
\label{thmitem:dimfoemula} 
The dimension of the simple object $B_\pi$ is computed by the dimension formula
\[
2\dim(B_\pi)=s_\pi[\bQ(\pi):\bQ].
\]
\item
\label{thmitem:invariantatp} 
The local invariant of $E_\pi = \End_{\bF_q}(B_\pi)\otimes\bQ$ at a place $v$ of its center $\bQ(\pi)$  is the following element of $\bQ/\bZ$:
\[
\inv_v(E_\pi) = 
\begin{cases}
0 & \text{ if } \ellplace \nmid p\infty \text{ or if  $\ellplace$ is complex,} \\
\frac{1}{2} & \text{ if $v$ is real}, \\
\dfrac{v(\pi)[K_\pplace:\bQ_p]}{v(q)} = \dfrac{v(\pi)f_\pplace}{\expo} & \text{ if } \pplace \mid p,
\end{cases} 
\]
where $f_\pplace = [K_{\pplace}:\bQ_p]$ is the inertial degree of $\bQ(\pi)$ at $\pplace \mid p$. 
\end{enumerate}
\end{thm}

\smallskip

Let now $\ell$ be a prime number including for the moment possibly $\ell = p$. 
The completions $K_v$ of $\bQ(\pi)$ at the places $v \mid \ell$ are factors of the product decomposition
\begin{equation}
\label{decompositioncenter}
\bQ(\pi) \otimes \bQ_\ell = \prod_{\ellplace | \ell} K_\ellplace.
\end{equation}
The algebra $E_\pi\otimes\bQ_\ell$ has center $\bQ(\pi)\otimes\bQ_\ell$ and \eqref{decompositioncenter} similarly induces a decomposition
\[
E_\pi \otimes_{\bQ} \bQ_\ell = \prod_{\ellplace|\ell} E_\ellplace,
\]
where $E_\ellplace$ is a central simple algebra over $K_\ellplace$. Since the idempotents cutting out $K_\ellplace$ from $\bQ(\pi)\otimes\bQ_\ell$ lie in the center of $E_\pi\otimes\bQ_\ell$, we
have an analogous decomposition
\begin{equation}
\label{eq:decomposeV}
V_\ell(B_\pi) = \bigoplus_{\ellplace| \ell} V_\ellplace(B_\pi)
\end{equation}
as a module under $E_\pi\otimes\bQ_\ell$ with the summand $V_\ellplace(B_\pi)$ being an $E_v$-module. 

\smallskip

For a prime $\ell\neq p$, the rational Tate module $V_\ell(B_\pi)$ has both an action by $E_\pi \otimes \bQ_\ell$  and a commuting action by the absolute Galois group of $\bF_q$. The latter was identified in Section~\ref{subsec:localizations} as the action of the center $R_\pi \otimes \bQ_\ell$ of $E_\pi \otimes \bQ_\ell$.
Hence the decomposition \eqref{eq:decomposeV} is also a decomposition as Galois representations, with $K_v$ acting on $V_v(B_\pi)$ capturing the Galois action. 
The summands have the following structure.

\begin{lem} \label{lem:rationalAtell}
For a prime $\ell\neq p$ and a place $\ellplace \mid \ell$ of $\bQ(\pi)$,
there is an isomorphim
\begin{equation}\label{rationalAtell1} 
V_\ellplace(B_\pi) \simeq {K_\ellplace}^{\oplus s_\pi}
\end{equation}
as Galois modules, i.e.\ as modules for $K_\ellplace$.  In particular, there is an isomorphism 
\begin{equation}\label{rationalAtell}
V_\ell(B_\pi) \simeq (\bQ(\pi) \otimes \bQ_\ell)^{\oplus s_\pi}
\end{equation}
as Galois modules, i.e.\ as modules for $\bQ(\pi) \otimes \bQ_\ell$. 
\end{lem}
\begin{proof}
From Tate's Theorem \eqref{TateThmell} for $A = B = B_\pi$ we deduce, by applying base change along 
$R_\pi \otimes \bZ_\ell \to K_v$,
a natural $K_v$-algebra isomorphism
\begin{equation*}
E_\ellplace \simeq \End_{K_\ellplace}(V_\ellplace(B_\pi)).
\end{equation*}
The algebra $E_v$ has $K_v$-dimension $s_\pi^2$ as a base change of $E_\pi$.
Thus $s_\pi = \dim_{K_\ellplace}(V_\ellplace(B_\pi))$ and \eqref{rationalAtell1} follows  because $K_\ellplace$ is a field.
The existence of the isomorphism \eqref{rationalAtell} follows from \eqref{rationalAtell1} together with \eqref{decompositioncenter}, since the exponents $s_\pi$ are independent
of $\ellplace$.
\end{proof}

\smallskip

We now turn our attention to the local structure at $p$. The rational Dieudonn\'e ring
\[
\cD_q^0 = W(\bF_q)\{\cF,\cV\}/(\cF\cV-p)\otimes_{\bZ_p}  \bQ_p
\]
is central over the algebra $\bQ_p[F,F^{-1}]$ 
with $F =\cF^\expo$. We set $W(\bF_q)^0  = W(\bF_q)[\frac{1}{p}] $, and observe that 
\[
W(\bF_q)^0 \otimes_{\bQ_p} \bQ_p[F,F^{-1}]
\]
is a cyclic \'etale $\bQ_p[F,F^{-1}]$-algebra of degree $\expo$ and canonical generator $\sigma \otimes \id =: \sigma$ of its Galois group. This imposes on $\cD_q^0$ the structure of a cyclic Azumaya algebra
\[
\cD_q^0 = \big(W(\bF_q)^0 \otimes_{\bQ_p} \bQ_p[F,F^{-1}],\sigma, F\big)
\]
over  $\bQ_p[F,F^{-1}]$. 
The action of $\cD_q^0$ on $V_p(B_\pi)$ factors through the base change with 
\[
\bQ_p[F,F^{-1}] \surj \bQ(\pi) \otimes \bQ_p, \qquad F \mapsto \pi,
\]
namely the quotient
\[
\cD_\pi^0 = \cD_q^0 \otimes_{\bQ_p[F,F^{-1}]} (\bQ(\pi) \otimes \bQ_p) = W(\bF_q)^0\{\cF\}/(P_\pi(\cF^\expo)).
\]
The $\bQ_p$-linear map sending $\pi$ to $F=\cF^\expo$ gives an isomorphism between $\bQ(\pi)\otimes\bQ_p$ and the center of $\cD_\pi^0$.
As a base change of a cyclic algebra, the ring $\cD_\pi^0$ is itself a cyclic algebra over $\bQ(\pi) \otimes \bQ_p$ as follows. Consider
\[
L_\pi = W(\bF_q)^0 \otimes_{\bQ_p} \big(\bQ(\pi) \otimes \bQ_p\big),
\]
as a cyclic \'etale $\bQ(\pi)\otimes\bQ_p$-algebra of degree $\expo$ and $\sigma\otimes \id=\sigma$ as a distinguished generator of the Galois group. Then 
\begin{equation} \label{eq:rationalcyclicalgebra}
\cD_\pi^0 = (L_\pi,\sigma, \pi) = L_\pi\{\cF\}/(\cF \text{ is } \sigma\text{-semilinear, and } \cF^\expo = \pi)
\end{equation}
as a cyclic Azumaya algebra of index $\expo$ over $\bQ(\pi) \otimes \bQ_p$. After decomposing $\cD_\pi^0$ according to the components $K_\pplace$ of its center,
we obtain a decomposition
\[
\cD_\pi^0 = \prod_{\pplace | p} \cD_{\pi,\pplace}^0
\]
into the product over the $p$-adic places of $\bQ_p$ of certain central simple algebras $\cD_{\pi,\pplace}^0$ over $K_\pplace$, each of index $\expo$.

The rational Dieudonn\'e module $V_p(B_\pi)$ is simultaneously a left module for the Dieudonn\'e ring $\cD_\pi^0$ and for the opposite endomorphisms actions by $E_\pi^\op\otimes \bQ_p$.
Since the idempotents cutting out $K_\pplace$ commute with both these actions, we can decompose
\[
V_p(B_\pi) = \bigoplus_{\pplace | p} V_{\pplace}(B_\pi)
\]
as a module for both actions. The summands have the following structure.

\begin{lem} \label{lem:rationalAtp}
For a  place $\pplace \mid p$ of $\bQ(\pi)$, 
there is an isomorphism
\begin{equation} \label{rationalAtp1}
V_\pplace(B_\pi)^{\oplus \expo} \simeq {(\cD_{\pi,\pplace}^0)}^{\oplus s_\pi}
\end{equation}
as Dieudonn\'e modules, i.e.\ as modules for $\cD_{\pi,\pplace}^0$.  In particular, there is an isomorphism 
\begin{equation} \label{rationalAtp}
V_p(B_\pi)^{\oplus \expo} \simeq (\cD_\pi^0)^{\oplus s_\pi}
\end{equation}
as Dieudonn\'e modules, i.e.\ as modules for $\cD_{\pi}^0$.
\end{lem}

\begin{proof}
From Tate's Theorem \eqref{TateThmp}  for $A = B = B_\pi$ we deduce, by applying base change along 
$R_\pi \otimes \bZ_p \to K_\pplace$, a natural $K_v$-algebra isomorphism
\[
E_\pplace^\op \simeq \End_{\cD_{\pi,\pplace}^0}\big(V_{\pplace}(B_\pi)\big).
\]
This means that $E_\pplace^\op$ is the centralizer of $\cD_{\pi,\pplace}^0$ in $\End_{K_v}\big(V_{\pplace}(B_\pi)\big)$. A corollary to the double centralizer theorem, \cite[Corollary 7.14]{reiner:maxorder}, then states that 
\begin{equation*} 
E_\pplace^\op \otimes_{K_\pplace} \cD_{\pi,\pplace}^0 = \End_{K_\pplace}(V_\pplace(B_\pi)).
\end{equation*}
In particular, comparing the degrees of these simple central $K_\pplace$-algebras, we find 
\[
s_\pi \cdot \expo = \dim_{K_\pplace}(V_\pplace(B_\pi)).
\]
Since modules over $\cD_{\pi,\pplace}^0$ are determined, up to isomorphism, by their $K_\pplace$-dimension, assertion \eqref{rationalAtp1} follows by comparing dimensions
\[
\dim_{K_v}\big(V_\pplace(B_\pi)^{\oplus \expo} \big) = r \cdot \dim_{K_\pplace}(V_\pplace(B_\pi)) = s_\pi \cdot r^2 = \dim_{K_v}\big(  {(\cD_{\pi,\pplace}^0)}^{\oplus s_\pi}\big) .
\]
The existence of the isomorphism \eqref{rationalAtp} follows from \eqref{rationalAtp1} together with \eqref{decompositioncenter}, since the exponents $s_\pi$ and $r$ are independent
of $\pplace$.
\end{proof}

\subsection{Choosing good multiplicities and choice of a lattice}\label{sec:multiplicities} 

The index $s_\pi$ of $E_\pi$ agrees with the period of $E_\pi$, defined to be the order of $E_\pi$ in the Brauer group $\Br(\bQ(\pi))$ of $\bQ(\pi)$ (essentially due to \cite{BHN}, see \cite{reiner:maxorder} Theorem 32.19 for a textbook reference).
By the Hasse--Brauer--Noether Theorem the period of $E_\pi$ is the
least common multiple over all places of $\bQ(\pi)$ of all local orders of $E_\pi$ in the appropriate Brauer groups. Since the localization of $E_\pi$
at any $\ell$-adic place is trivial for $\ell\neq p$, only $p$-adic and real places (if any) of $\bQ(\pi)$ must be taken into account to compute the global order
$s_\pi$ from the local ones.

The denominator of the local invariant of $E_\pi$ at a $p$-adic place $\fp$ divides $\expo$ by Theorem~\ref{thm:TateEndostructure}~\ref{thmitem:invariantatp}. 
Because the local invariant of $E_\pi$ at a real place of $\bQ(\pi)$ equals $1/2$, we conclude that
\[
s_\pi\mid\lcm(\expo,2),
\]
and in particular $s_\pi$ divides $2\expo$ in all cases.
Notice that unless $\expo$ is odd and $\pi$ is the real conjugacy class $\{\pm\sqrt{q}\}$ of Weil $q$-numbers, then $s_\pi$ divides $\expo$.

\begin{defi} 
The \emph{balanced multiplicity of $\pi$} is the integer
\[
m_\pi=\frac{2 \expo}{s_\pi}.
\]
Unless $\expo$ is odd and $\pi$ is real we define the \emph{reduced balanced multiplicity} of $\pi$ to be the integer
\[
\widebar{m}_\pi = m_\pi/2 = \frac{\expo}{s_\pi}.
\]
\end{defi}

This choice of $m_\pi$ (resp.\ $\widebar{m}_\pi$) ensures that the ranks are independent of 
$\pi \in W_q$ in the following freeness statement for the rational Tate and Dieudonn\'e modules of $B_\pi^{m_\pi}$ (resp.\ $B_\pi^{\widebar{m}_\pi}$). 

\begin{prop}
\label{prop:rationalfreeness} 
Let $\pi$ be a Weil $q$-number. There is an isomorphism 
\begin{enumerate}[label=(\arabic*),align=left,labelindent=0pt,leftmargin=*,widest = (8)]
\item  
$V_\ell(B_\pi^{m_\pi})\simeq {(\bQ(\pi) \otimes \bQ_\ell)}^{\oplus 2 \expo}$ of $\bQ(\pi) \otimes \bQ_\ell$-modules, for any prime $\ell\neq p$, and
\item 
$V_p(B_\pi^{m_\pi})\simeq{(\cD_{\pi}^0)}^{\oplus  2}$ of $\cD_\pi^0$-modules.
\end{enumerate}
Unless $\expo$ is odd and $\pi=\pm \sqrt{q}$  there is an isomorphism 
\begin{enumerate}[resume]
\item 
$V_\ell(B_\pi^{\widebar{m}_\pi})\simeq{(\bQ(\pi) \otimes \bQ_\ell)}^{\oplus \expo}$ of $\bQ(\pi) \otimes \bQ_\ell$-modules, for any prime $\ell\neq p$, and
\item 
$V_p(B_\pi^{\widebar{m}_\pi})\simeq \cD_{\pi}^0$ of $\cD_\pi^0$-modules. 
\end{enumerate}
\end{prop}

\begin{proof}
The proposition follows from the fact that $V_\ell(B_\pi)$ is a product of modules over simple algebras, and from Lemmas~\ref{lem:rationalAtell} and \ref{lem:rationalAtp}.
\end{proof}

Proposition~\ref{prop:rationalfreeness} admits the following integral refinement. 

\begin{thm} 
\label{thm:integralstructure}
Let $w\subseteq W_q$ be a finite subset. There is an abelian variety $A_w$, isogenous to $\prod_{\pi\in w}B_\pi^{m_\pi}$, such that 
\begin{enumerate}[label=(\arabic*),align=left,labelindent=0pt,leftmargin=*,widest = (8)]
\item $T_\ell(A_w) \simeq (R_w\otimes \bZ_\ell)^{\oplus 2\expo}$ as Galois modules, for all primes $\ell\neq p$.
\item $T_p(A_w) \simeq (\cD_{w})^{\oplus 2}$ as Dieudonn\'e module.
\end{enumerate}

If $\expo$ is even or $w$ avoids $\pi=\pm \sqrt{q}$ there is an abelian variety $\widebar{A}_w$, isogenous to $\prod_{\pi\in w}B_\pi^{\widebar{m}_\pi}$, such that 
\begin{enumerate}[resume]
\item $T_\ell(\widebar{A}_w) \simeq (R_w\otimes \bZ_\ell)^{\oplus \expo}$ as Galois modules, for all primes $\ell\neq p$.
\item $T_p(\widebar{A}_w) \simeq \cD_{w}$ as Dieudonn\'e module.
\end{enumerate}
\end{thm}
\begin{proof}
We first work rationally and set $A'_w = \prod_{\pi \in w} B_\pi^{m_\pi}$.
Proposition~\ref{prop:rationalfreeness} thus gives isomorphisms
\begin{enumerate}[label=(\arabic*),align=left,labelindent=0pt,leftmargin=*,widest = (8)]
\item $V_\ell(A'_w) \simeq (\bQ(w) \otimes \bQ_\ell)^{\oplus 2\expo}$,
for any $\ell\neq p$,
\item $V_p(A'_w) \simeq (\cD_w^0)^{\oplus 2}$.
\end{enumerate}

Both sides of these isomorphisms contain natural lattices: $T_\ell(A'_w)$ versus $(R_w\otimes \bZ_\ell)^{\oplus 2\expo}$, and $T_p(A'_w)$ versus ${(\cD_w)}^{\oplus 2}$. For abstract algebraic reasons we may modify all but finitely many of the isomorphisms forcing them to respect these lattices. The point is that for suitable $N \in \bZ$ the localized ring $R_w[\frac{1}{N}]$ is a product of Dedekind rings, hence after completing at $\ell$, for $\ell \nmid Np$, all torsion free modules are locally free. Since the rank of $T_\ell(A'_w)$ is constant and equal to $2\expo$ we may even deduce that it is necessarily free. A suitable modification of the isomorphism then maps one free lattice into the other.

It remains to deal with finitely many primes $\ell$, and the prime $p$. But here the two lattices on both sides are commensurable, hence without loss of generality we can assume that
one is contained in the other, after rescaling the isomorphisms. Then a suitable isogeny allows us to replace $A_w'$ with the required $A_w$. 

The construction of $\widebar{A}_w$ with reduced balanced multiplicities follows with the same proof.
\end{proof}

\begin{defi}
\label{defi:balanced}
Let $A$ be an abelian variety in $\AV_w$.
\begin{enumerate}[label=(\arabic*),align=left,labelindent=0pt,leftmargin=*,widest = (8)]
\item
We say that $A$ is \emph{$w$-balanced} if $A$ satisfies the conditions of Theorem~\ref{thm:integralstructure} with respect to the balanced multiplicities $m_\pi$. An abelian variety is \emph{balanced} if it is $w$-balanced for some finite set $w$ of Weil numbers (necessarily with $w$ equal to the Weil support of the abelian variety).

\item
We say that $A$ is \emph{reduced $w$-balanced} if $A$ satisfies the conditions of Theorem~\ref{thm:integralstructure} with respect to the reduced balanced multiplicities $\widebar{m}_\pi$. 
\end{enumerate}
\end{defi}

A (reduced) $w$-balanced abelian variety is 
$w$-locally projective in the sense of Definition~\ref{defi:locallyprojective}.

\section{Integral local theory and locally projective abelian varieties}
\label{sec:truely local}

Let $w$ be a finite set of Weil $q$-numbers. Let $\ell \not=p$  be a prime number. The ring $R_w \otimes \bZ_\ell$ is a semi-local $\bZ_\ell$-algebra, all maximal ideals $\lambda$ lie over $(\ell)$ and there is a canonical product decomposition
\[
R_w \otimes \bZ_\ell = \prod_\lambda R_{w,\lambda}
\] 
where $R_{w,\lambda} = (R_w \otimes \bZ_\ell)_\lambda$ is the localization (isomorphic to the completion of $R_w$ at the maximal ideal corresponding to $\lambda$). 
Similarly, we can decompose $R_w \otimes \bZ_p = \prod_{\fp} R_{w,\fp}$ with respect to the maximal ideals $\fp$ of $R_w \otimes \bZ_p$ and the localization $R_{w,\fp} = (R_w \otimes \bZ_p)_\fp$.

\begin{defi}
\label{defi:localTate modules}
Let $A$ be an abelian variety in $\AV_w$. Let $\lambda$ (resp. $\fp$) be a maximal ideal of $R_w \otimes \bZ_\ell$ (resp.\ of $R_w \otimes \bZ_p)$ for some $\ell \not= p$.
\begin{enumerate}[label=(\arabic*),align=left,labelindent=0pt,leftmargin=*,widest = (8)]
\item 
We define the \emph{$\lambda$-adic Tate module} of $A$ as the $R_{w,\lambda}$-module
\[
T_\lambda(A) = T_\ell(A) \otimes_{R_w \otimes \bZ_\ell} R_{w,\lambda}.
\]
\item
We define the \emph{$\fp$-component} of the Dieudonn\'e ring as
\[
\cD_{w,\fp} = \cD_w \otimes_{R_w \otimes \bZ_p} R_{w,\fp},
\]
and  the \emph{$\fp$-adic Tate module} of $A$ as the $\cD_{w,\fp}$-module
\[
T_\fp(A) = T_p(A) \otimes_{R_w \otimes \bZ_p} R_{w,p}.
\]
\end{enumerate}
\end{defi}

The analogs of Tate's theorems \eqref{TateThmell} and \eqref{TateThmp} continue to hold.

\begin{prop}
\label{prop:localTateTheorem}
Let $A$ and $B$ be abelian varieties in $\AV_w$. 
\begin{enumerate}[label=(\arabic*),align=left,labelindent=0pt,leftmargin=*,widest = (8)]
\item \label{propitem:localTateTheorem lambda}
For all maximal ideals $\lambda$ of $R_w \otimes \bZ_\ell$ with $\ell \not= p$, the functor $T_{\lambda}(-)$ induces an isomorphism
\[
\Hom_{\bF_q}(A,B) \otimes_{R_w} R_{w,\lambda}  \stackrel{\sim}{\longrightarrow} \Hom_{R_{w,\lambda}}(T_\lambda(A),T_\lambda(B))
\]
of $R_{w,\lambda}$-modules.
\item \label{propitem:localTateTheorem p}
For all maximal ideals $\fp$ of $R_w \otimes \bZ_p$, the functor $T_{\fp}(-)$ induces an isomorphism
\[
\Hom_{\bF_q}(A,B) \otimes_{R_w} R_{w,\fp}  \stackrel{\sim}{\longrightarrow} \Hom_{\cD_{w,\fp}}(T_\fp(B),T_\fp(A))
\]
of $R_{w,\fp}$-modules.
\end{enumerate}
\end{prop}
\begin{proof}
This follows at once from Tate's theorems \eqref{TateThmell} and \eqref{TateThmp} in view of base change by the flat map $R_w \otimes \bZ_\ell \to R_{w,\lambda}$ (resp. the flat map $R_w \otimes \bZ_p \to R_{w,\fp}$).
\end{proof}

\begin{defi}
\label{defi:locallyfree}
Let $q=p^\expo$ be a power of a prime number $p$, and let $w \subseteq W_q$ be a finite subset of Weil $q$-numbers. 
An abelian variety $A \in \AV_w$ is \emph{$w$-locally free} if 
\begin{enumerate}[label=(\roman*),align=left,labelindent=0pt,leftmargin=*,widest = (iii)]
\item
for all maximal ideals $\lambda$  of $R_w \otimes \bZ_\ell$ the Tate module $T_\lambda(A)$ is a free $R_{w,\lambda}$-module, and
\item
for all maximal ideals $\fp$  of $R_w \otimes \bZ_p$ the Tate module $T_\fp(A)$ is a free $\cD_{w,\fp}$-module.
\end{enumerate}
\end{defi}

\begin{rmk}
The notion of an abelian variety $A$ in $\AV_w$ being $w$-locally projective, see Definition~\ref{defi:locallyprojective}, is compatible with the point of view towards local properties taken in Definition~\ref{defi:locallyfree}, because $T_\ell(A)$ is projective if and only if $T_\lambda(A)$ is projective for all $\lambda \mid \ell$ and similarly $T_p(A)$ is projective if and only if $T_\fp(A)$ is projective for all $\fp \mid p$.
\end{rmk}

\begin{lem}
\label{lem:topologicalNakayama}
Let $w \subseteq W_q$ be a finite set of Weil $q$-numbers. Let $\fp$ be a maximal ideal of $R_w \otimes \bZ_p$ such that $F$ and $V$ are in $\fp$. 

Let $M$ be a finitely generated $\cD_{w,\fp}$-module, and set $\bar M \coloneqq M/(\cF,\cV)M$. Then the following holds.
\begin{enumerate}[label=(\arabic*),align=left,labelindent=0pt,leftmargin=*,widest = (8)]
\item 
\label{lemitem:Dieudonnemodsupersingular}
The ideal $\fp$ equals $(p,F,V)$, and the ring $\cD_{w,\fp}/(\cF,\cV)$ equals $\bF_q$.
\item 
\label{lemitem:TopologicalNakaGenerators}
Elements $x_1, \ldots, x_m \in M$ generate $M$ as a $\cD_{w,\fp}$-module if and only if their images generate $\bar M$ as an $\bF_q$-vector space.
\item
\label{lemitem:TopologicalNakaVanishing}
$M = 0$ if and only if $\bar M = 0$.
\end{enumerate}
\end{lem}
\begin{proof}
Assertion~\ref{lemitem:Dieudonnemodsupersingular} follows from the definition of $R_w \otimes \bZ_p$ and of $\cD_{w,\fp}$. The assumption implies that the constant term in $h_w(F,V)$ is divisible by $p$ and thus $(\cF,\cV)$ is actually a nontrivial two-sided ideal.

Assertion~\ref{lemitem:TopologicalNakaGenerators} follows from assertion~\ref{lemitem:TopologicalNakaVanishing} applied to the cokernel of the map $\cD_{w,\fp}^{\oplus m} \to M$ induced by the tupel of elements $x_1, \ldots, x_m$. It therefore remains to show that $\bar M = 0$ implies $M= 0$. 

Since $R_{w,\fp}$ is a finite $\bZ_p$-module, the $\fp$-adic topology agrees with the $p$-adic topology. Therefore $\fp^n \subseteq pR_{w,\fp}$ for sufficiently large $n$. 
By assumption then 
\[
(\cF,\cV)^{n\expo} \subseteq (\cF^\expo, p, \cV^\expo)^n = \fp^n\cD_{w,\fp}  \subseteq p \cD_{w,\fp} \subseteq (\cF,\cV). 
\]
It follows that the $p$-adic topology on $M$ agrees with the $(\cF,\cV)$-adic topology. Since $M$ is finitely generated as a $\cD_{w,\fp}$-module, it is also finitely generated as a $\bZ_p$-module and thus $p$-adically hausdoff. Therefore 
\[
0 = \bigcap_{t \geq 0} p^t M = \bigcap_{t \geq 0} (\cF,\cV)^t M.
\]
If $\bar M = 0$, then  $M$ equals $(\cF,\cV)M$, and  the latter intersection equals $M$. 
\end{proof}

\begin{prop}
\label{prop:localprojectiveDieudonnestructure}
Let $\fp$ be a maximal ideal of $R_w \otimes \bZ_p$. Up to isomorphism there is a unique non-zero finitely generated indecomposable projective $\cD_{w,\fp}$-module $P_{w,\fp}$. All finitely generated projective $\cD_{w,\fp}$-modules are isomorphic to multiples $P_{w,\fp}^{\oplus m}$  
for some $m \geq 0$.

If $\fp$ equals the supersingular maximal ideal $\fp_o = (F,V,p)$, then the indecomposable projective module $P_{w,\fp_o}$ is the free $\cD_{w,\fp_o}$-module of rank $1$.
\end{prop}
\begin{proof}
There are three cases. If $F \notin \fp$, then $F$ is a unit in $R_{w,\fp}$. So $V = q F^{-1}$ holds in $R_{w,\fp}$ and 
\[
\cD_{w,\fp} = \big(W(F_q) \otimes R_{w,\fp} \big)\{\cF\}/(\cF^\expo - F)
\]
is a cyclic algebra over the complete local ring $R_{w,\fp}$ with finite residue field. It follows from  
\cite[IV Proposition 1.4]{milne:etcoh} and the fact that finite fields have trivial Brauer group that this cyclic algebra is trivial, i.e.\ there is an $R_{w,\fp}$-algebra isomorphism
\begin{equation}
\label{eq:structureDwfp Finvertible}
\cD_{w,\fp} \simeq \rM_\expo(R_{w,\fp}).
\end{equation}
By Morita equivalence, projective $\cD_{w,\fp}$-modules are translated into projective $R_{w,\fp}$-modules, which are free. So all projective $R_{w,\fp}$-modules are multiples of a unique indecomposable one. This translates back by Morita equivalence to $\cD_{w,\fp}$-modules. The module $P_{w,\fp}$ corresponds to the $\rM_\expo(R_{w,\fp})$-module of column vectors $(R_{w,\fp})^\expo$.

If $V \notin \fp$, then we have $F = q V^{-1}$ in $R_{w,\fp}$ and the above holds with $F$ and $V$ interchanged. So again we have 
\begin{equation}
\label{eq:structureDwfp Vinvertible}
\cD_{w,\fp} \simeq \rM_\expo(R_{w,\fp})
\end{equation}
with the same conclusion.

In the third case both $F$ and $V$ are contained in $\fp$ and we are in the situation of Lemma~\ref{lem:topologicalNakayama}. Let $P$ be a finitely generated projective $\cD_{w,\fp}$-module, and let $m = \dim_{\bF_q} P/(\cF, \cV) P$. Then, by choosing an $\bF_q$-basis of $P/(\cF, \cV) P$, Lemma~\ref{lem:topologicalNakayama}~\ref{lemitem:TopologicalNakaGenerators} shows there is a surjection
\[
f: \cD_{w,\fp}^m \surj P.
\]
This surjection splits since $P$ is projective. Therefore $Q=\ker(f)$ is a direct summand, and consequently also finitely generated as a $\cD_{w,\fp}$-module. Since $f$ modulo $(\cF,\cV)$ is an isomorphism, we deduce that $Q/(\cF,\cV)Q = 0$. Lemma~\ref{lem:topologicalNakayama}~\ref{lemitem:TopologicalNakaVanishing} implies $Q = 0$ and hence $P \simeq \cD_{w,\fp}^{\oplus m}$ is free. So in this case $P_{w,\fp}$ equals $\cD_{w,\fp}$.
\end{proof}

The connection between the notions `$w$-locally projective' and `$w$-locally free' is summarized by the following proposition.

\begin{prop}
\label{prop:locallyfreeandprojective}
Let $q=p^\expo$ be a power of a prime number $p$, and let $w \subseteq W_q$ be a finite subset of Weil $q$-numbers. Let $A$ be an abelian variety in $\AV_w$. 
\begin{enumerate}[label=(\arabic*),align=left,labelindent=0pt,leftmargin=*,widest = (8)]
\item 
\label{propitem:locallyfreeandprojective1}
If $A$ is $w$-locally free, then $A$ is $w$-locally projective. 
\item 
\label{propitem:locallyfreeandprojective2}
If $A$ is $w$-locally projective, then there is an $n \geq 1$ such that $A^n$ is $w$-locally free. In fact $n = r$ always works. 
\end{enumerate}
\end{prop}
\begin{proof}
Assertion \ref{propitem:locallyfreeandprojective1} is obvious because free modules are projective. We now prove \ref{propitem:locallyfreeandprojective2} and assume that $A$ is $w$-locally projective. Tate modules are finitely generated modules, and finitely generated projective modules over a commutative local ring are free. Therefore, since for $\ell \not= p$ and all maximal ideals $\lambda$ of $R_w \otimes \bZ_\ell$ the ring $R_{w,\lambda}$ is a commutative local ring, we find that $T_\lambda(A)$ is a free $R_{w,\lambda}$-module and the same holds for multiples $A^n$. 

It remains to discuss the local structure at all maximal ideals $\fp$ of $R_w \otimes \bZ_p$. By Proposition~\ref{prop:localprojectiveDieudonnestructure}, the $\fp$-adic Tate module $T_\fp(A)$ is a multiple of $P_{w,\fp}$. As in all cases considered in Proposition~\ref{prop:localprojectiveDieudonnestructure} the  
$r$-th multiple $P_{w,\fp}^{\oplus r}$ is a free $\cD_{w,\fp}$-module, the claim follows.
\end{proof}

\begin{ex}
Let $w$ be a finite set of non-real Weil $q$-numbers, and set $P_w(x)$ to be the product of the  minimal polynomials $P_\pi(x)$ for $\pi \in w$. Let $2d$ be the degree of $P_w(x)$. Since $x \mapsto q/x$ permutes the roots of $P_w(x)$, there is a Polynomial $Q_\beta(x) \in \bZ[x]$, the polynomial with roots $\beta = \pi + q/\pi$ for $\pi \in w$, with $P_w(x) = x^d Q_\beta(x + q/x)$. Now $Q_\beta(x)$ has totally real roots of absolute value $\abs{\beta} < 2 \sqrt{q}$. Let us assume that $Q_\beta(x) + 1$ still is separable with totally real roots, all of which have absolute value bounded by $2 \sqrt{q}$. This certainly can happen when $d=1$, i.e.\ when $w = \{\pi\}$ and $P_\pi(x)$ is a quadratic polynomial. Then $P_w(x) + x^d = x^d(Q_\beta(x+ q/x) + 1)$ is still a separable Polynomial with Weil $q$-numbers as roots. Let $w'$ be the set of conjugacy classes of roots of $P_w(x) + x^d$, so that $P_{w'}(x) = P_w(x) + x^d$. It follows that 
\[
h_{w'}(F,V) = h_w(F,V) + 1.
\]
Therefore the the natural map 
\[
R_{w \cup w'} \longrightarrow R_{w} \times R_{w'}
\]
is an isomorphism in this case. There are no `congruences' between $\pi \in w$ and $\pi' \in w'$. 

Let us consider balanced abelian varieties $A_w$  in $\AV_{w}$ and $A_{w'}$ in $\AV_{w'}$. Then, for $n,n' \in \bN$, the abelian variety $A =A_w^n \times A_{w'}^{n'}$ is $w\cup w'$-locally projective. But $T_\ell(A)$  is free as $R_{w \cup w'} \otimes \bZ_\ell$-module only if $n=n'$, while $T_\lambda(A)$ is a free $R_{w \cup w',\lambda}$-module for all $n, n'$. This phenomenon occurs because $\Spec(R_{w \cup w',\lambda})$ is local while $\Spec(R_{w \cup w'} \otimes \bZ_\ell)$ is only semi-local and not connected in the example. This illustrates the conceptual advantage of $T_\lambda(-)$ over $T_\ell(-)$, since $T_\lambda(-)$ is local in the sense of Section~\S\ref{subsection:geometry}. 
\end{ex}

The above example exploits a case where $\Spec(R_w)$ is not connected. Now we analyse the connected case.

\begin{prop}
\label{prop:VB}
Let $w$ be a finite set of Weil $q$-numbers. Let $A$ and $B$  be $w$-locally projective abelian varieties. 
\begin{enumerate}[label=(\arabic*),align=left,labelindent=0pt,leftmargin=*,widest = (8)]
\item 
\label{propitem:projectiveHom}
If $w$ consists only of ordinary Weil $q$-numbers, then $\fp_o = (F,V,p)$ is not a maximal ideal of $R_w$ and  $\Hom_{\bF_q}(A,B)$ is a projective $R_w$-module.
\item
\label{propitem:projectiveHom punctured}
If $w$ contains a non-ordinary Weil $q$-number, then $\fp_o = (F,V,p)$ is a point in $X_w =  \Spec(R_w)$, the coherent sheaf described by the $R_w$-module $\Hom_{\bF_q}(A,B)$ is locally free on $X_w - \{\fp_o\}$, and there is a Zariski open neighborhood $U_w$ of $\fp_o$ such that the rank is constant on $U_w - \{\fp_o\}$.
\end{enumerate}
\end{prop}
\begin{proof}
Since $X_w = \Spec(R_w)$ is the union of the $X_\pi = \Spec(R_\pi)$ in $\Spec \bZ[F,V]/(FV-q)$, the point $\fp_o$ lies on $X_w$ if and only if there is a $\pi \in w$ with $\fp_o$ contained in $X_\pi$. The latter is equivalent to the existence of a quotient $R_\pi \surj \bF_p$  that sends $F, V$ to $0$. In other words, there is a prime ideal of the order $R_\pi$ at which $\pi$ and $q/\pi$ lie in the maximal ideal. This happens if and only if $\pi$ is not ordinary.

The (completed) local structure as an $R_w$-module at a maximal ideal $\lambda$ above a prime $\ell \not= p$ (resp.\ at $\fp$ above $p$) is described by Proposition~\ref{prop:localTateTheorem}. Since $A$ and $B$ are $w$-locally projective, at $\lambda$ the modules $T_\lambda(A)$ and $T_\lambda(B)$ are projective and 
\[
\Hom_{\bF_q}(A,B) \otimes_{R_w} R_{w,\lambda} = \Hom_{R_{w,\lambda}}(T_\lambda(A),T_\lambda(B))
\]
is a projective $R_{w,\lambda}$-module. To analyse the situatioin at $\fp$ we may replace $A$ and $B$ by multiples according to Proposition~\ref{prop:locallyfreeandprojective} and thus assume that $A$ and $B$ are $w$-locally free. It then follows that $T_\fp(A) \simeq \cD_{w,\fp}^{\oplus n}$ and $T_\fp(B) \simeq \cD_{w,\fp}^{\oplus m}$ for some $n,m \in \bN_0$ and 
\begin{equation}
\label{eq:HomAtPo}
\Hom_{\bF_q}(A,B) \otimes_{R_w} R_{w,\fp}  = \Hom_{\cD_{w,\fp}}(\cD_{w,\fp}^{\oplus m}, \cD_{w,\fp}^{\oplus n})
\end{equation}
is isomorphic to a space of matrices with entries in $\End_{\cD_{w,\fp}} (\cD_{w,\fp}) \simeq \cD_{w,\fp}^\op$. 

The first claim follows if $\cD_{w,\fp}$ for $\fp \not= \fp_o$ is a free $R_{w,\fp}$-module. This follows if  $F \not\in \fp$  from \eqref{eq:structureDwfp Finvertible} and if $V \not\in \fp$ from \eqref{eq:structureDwfp Vinvertible}. This completes the proof of~\ref{propitem:projectiveHom}.

We now prove~\ref{propitem:projectiveHom punctured}. From the proof of~\ref{propitem:projectiveHom} we learn that still the coherent sheaf on $X_w$ associated to $\Hom_{\bF_q}(A,B)$ is a locally free sheaf of finite rank on $X_w - \{\fp_o\}$. So it is locally away from $\fp_o$ a vector bundle, but the rank a priori may depend on the connected component of $X_w - \{\fp_o\}$. Assertion~\ref{propitem:projectiveHom punctured} claims that this rank is actually the same on all components of $X_w - \{\fp_o\}$ whose Zariski closure passes through $\fp_o$. Therefore it remains to work complete locally at $\fp = \fp_o$, i.e.\ on $\Spec(R_{w,\fp_o})$, where our coherent sheaf is described by \eqref{eq:HomAtPo} as isomorphic to the coherent sheaf associated to the $R_{w,\fp_o}$-module $(\cD_{w,\fp_o})^{\oplus mn}$. Note that the $R_{w,\fp_o}$-module structure is insensitive to passing to the opposite ring. 

The generic points of $\Spec(R_{w,\fp_o})$ correspond to the branches of $X_w$ at $\fp_o$. We must therefore  show that $\cD_{w,\fp_o} \otimes \bQ_p$ has the same rank at each generic point of $\Spec(R_{w,\fp_0})$. These generic points correspond to $\pi \in w$ and places $v$ of $\bQ(\pi)$ above $p$ with $0 < v(\pi) < v(q)$. But after inverting $p$, the $\bQ_p$-algebra $\cD_{w} \otimes \bQ_p$ is an Azumaya algebra of dimension $\expo^2$ over $\bQ(w) \otimes \bQ_p$, see \eqref{eq:rationalcyclicalgebra},  and so the same holds for the base change to $R_{w,\fp_0} \otimes \bQ_p$. It follows that indeed the rank of $\cD_{w,\fp_o} \otimes \bQ_p$ on the generic points of $\Spec(R_{w,\fp_0})$ is constant equal to $\expo^2$.
\end{proof}

\begin{thm}
\label{thm:connectedSpec classification of locallyprojective}
Let $w$ be a finite set of Weil $q$-numbers. We assume that 
\begin{enumerate}[label=(\roman*),align=left,labelindent=0pt,leftmargin=*,widest = (iii)]
\item
$X_w = \Spec(R_w)$ is connected, and
\item
there exists $\pi \in w$ which is non-ordinary.
\end{enumerate}
Let $A_w$ be a $w$-balanced abelian variety. Then the following holds.
\begin{enumerate}[label=(\arabic*),align=left,labelindent=0pt,leftmargin=*,widest = (8)]
\item 
If $\expo$ is even or $w$ only consists of non-real Weil $q$-numbers, in which case a reduced $w$-balanced abelian variety $\widebar{A}_w$ exists, then any $w$-locally projective abelian variety $A$ is isogenous to a power of $\widebar A_w$.
\item
If $\expo$ is odd and the real conjugacy class of Weil $q$-numbers $\{\pm \sqrt{q}\}$ is contained in $w$, then any $w$-locally projective abelian variety $A$ is isogenous to a power of $A_w$.
\end{enumerate} 
\end{thm}
\begin{proof}
Since $w$ contains a non-ordinary Weil $q$-number by assumption, the ideal $\fp_o = (F,V,p)$ is a maximal ideal of $R_w$, see Proposition~\ref{prop:VB}. By Proposition~\ref{prop:localprojectiveDieudonnestructure}
there is an $n \in \bN$ such that
\begin{equation}
\label{eq:compareatsusiprime}
T_{\fp_o}(A) \simeq \cD_{w,\fp_o}^{\oplus n} .
\end{equation} 
The $R_w$-module $\Hom_{\bF_q}(A_w,A)$ is locally free on $X_w - \{\fp_o\}$ by Proposition~\ref{prop:VB}. 
Moreover, by assertion \ref{propitem:projectiveHom punctured} of Proposition~\ref{prop:VB} its rank is locally constant in a neighborhood of $\fp_o$. 
Since $X_w$ is connected, such a neighborhood meets every connected component of $X_w - \{\fp_o\}$. 
It follows that $\Hom_{\bF_q}(A_w,A)$ actually corresponds to a vector bundle on $X_w - \{\fp_o\}$ of some fixed rank. 
This rank can be computed over $R_{w,\fp_o} \otimes \bQ_p$ from 
the proof of Proposition~\ref{prop:VB} and \eqref{eq:compareatsusiprime} to be
\begin{align*}
& \rank _{R_{w,\fp_o} \otimes \bQ_p}\big(\Hom_{\bF_q}(A_w,A) \otimes_{R_w} R_{w,\fp_o} \otimes \bQ_p\big) \\
& \quad = \rank _{R_{w,\fp_o} \otimes \bQ_p}\big(\Hom_{\cD_{w,\fp_o}^0}(V_{\fp_o}(A), V_{\fp_o}(A_w)) \big) \\
& \quad =  \rank _{\cD_{w,\fp_o}^0}\big(V_{\fp_o}(A)\big)  \cdot \rank _{\cD_{w,\fp_o}^0}\big(V_{\fp_o}(A_w)\big) \cdot \rank _{R_{w,\fp_o} \otimes \bQ_p}\big(\cD_{w,\fp_o}^0 \big) 
=  2n  \expo^2 .
\end{align*}
For all maximal ideals $\lambda$  of $R_w$ over a prime $\ell \not= p$ we deduce that 
\begin{align*}
2 \expo \cdot  \rank_{R_{w,\lambda}} \big(T_\lambda(A)\big)   & = 
\rank_{R_{w,\lambda}} \big(T_\lambda(A_w)\big) \cdot \rank_{R_{w,\lambda}}\big(T_\lambda(A)\big)\\
& = \rank_{R_{w,\lambda}} \big(\Hom_{R_{w,\lambda}}(T_\lambda(A_w),T_\lambda(A))\big) \\
& =  \rank_{R_{w,\lambda}} \big(\Hom_{\bF_q}(A_w,A) \otimes_{R_w} R_{w,\lambda} \big) \\
& =  \rank _{R_{w,\fp_o} \otimes \bQ_p}\big(\Hom_{\bF_q}(A_w,A) \otimes_{R_w} R_{w,\fp_o} \otimes \bQ_p\big) =  2n  \expo^2 ,
\end{align*}
where we used that the rank of the vector bundle $\Hom_{\bF_q}(A_w,A)$ on $X_w - \{\fp_o\}$ is the same at $\lambda$ and over a formal punctured neighborhood of $\fp_o$.
In particular, the $R_{w,\lambda}$-rank of $T_\lambda(A)$ equals $nr$ and is independent of $\lambda$, so that
as $R_w \otimes \bZ_\ell$-modules
\[
T_\ell(A) \simeq \prod_{\lambda \mid \ell} T_\lambda(A) \simeq (R_w \otimes \bZ_\ell)^{\oplus nr}.
\]

For a moment we assume that $\expo$ is odd and the real Weil number $\{\pm \sqrt{q}\}$ is contained in $w$. We set $A(\pi)$ to be the maximal abelian subvariety of $A$ with Weil support in $\pi$, and we let $n_\pi$ denote the  multiplicity up to isogeny of the simple $B_\pi$ in $A(\pi)$. Let $\ell$ be a prime different from $p$. By performing base change along $R_w \to \bQ(w) \to \bQ(\sqrt{q})$ we obtain
the rational $\ell$-adic Tate module 
\begin{equation} \label{eq:sqrtqisotypicalcomponent}
V_\ell\big(A(\pm \sqrt{q})\big) = T_\ell(A) \otimes_{R_w \otimes \bZ_\ell}  \bQ(\sqrt{q}) \otimes \bQ_\ell  \simeq  (\bQ(\sqrt{q}) \otimes \bQ_\ell )^{\oplus nr}.
\end{equation}
Comparing $\bQ_\ell$-dimensions of \eqref{eq:sqrtqisotypicalcomponent} with the help of Theorem~\ref{thm:TateEndostructure}~\ref{thmitem:dimfoemula} yields
\begin{align*}
n_{\pm \sqrt{q}} \cdot s_{\pm \sqrt{q}} \cdot  [ \bQ(\sqrt{q}): \bQ] & = n_{\pm \sqrt{q}} \cdot 2 \dim(B_{\pm \sqrt{q}}) 
= 2 \dim\big(A(\pm \sqrt{q})\big) \\
& = \dim_{\bQ_\ell}\big( V_\ell\big(A(\pm \sqrt{q})\big) \big) 
= \dim_{\bQ_\ell}\big( (\bQ(\sqrt{q}) \otimes \bQ_\ell )^{\oplus nr} \big) \\
& = nr \cdot [ \bQ(\sqrt{q}): \bQ]   .
\end{align*}
Since $\expo$ is odd and $s_{\pm \sqrt{q}} = 2$, it follows that $n$ must be even in this case.

Now we come back to the general case. We set (note that we just proved that $n/2$ is an integer in the second case)
\[
B = \begin{cases}
\widebar{A}_w^n & \text{ if $\expo$ is even or $w$ only consists of non-real Weil $q$-numbers}, \\
A_w^{n/2} &  \text{ if $\expo$ is odd and $w$ contains the real Weil $q$-number}.
\end{cases}
\]
It follows that as $R_w \otimes \bZ_\ell$-modules and thus also as Galois modules
\[
T_\ell(A) \simeq (R_w \otimes \bZ_\ell)^{\oplus nr} \simeq T_\ell(B) ,
\]
Tate's theorem \eqref{TateThmell} and the usual arguments show that $A$ and $B$ are isogenous.
\end{proof}

\begin{rmk}
\label{rmk:survey of injective cogenerators}
The connected components of $X_w = \Spec(R_w)$ correspond to a partition $w = w_1 \amalg \ldots \amalg w_s$ such that the natural map 
\[
R_w  \stackrel{\sim}{\longrightarrow}  R_{w_1} \times \ldots \times R_{w_s}
\]
is an isomorphism. Since all connected components $X_i = \Spec (R_{w_i})$ with a non-ordinary Weil $q$-number in $w_i$ contain the point $\fp_o = (F,V,p)$, there is at most one such connected component: we call it the \emph{non-ordinary} component. All the other connected components will be called \emph{ordinary} connected components. 
Any $w$-locally projective abelian variety is the product of $w_i$-locally projective abelian varieties. 
Theorem~\ref{thm:connectedSpec classification of locallyprojective} describes the factor arising from the non-ordinary component if that is present. The factor from ordinary components will be explained in  Theorem~\ref{thm:classification wlocallyprojectiveordinary}.
\end{rmk}

\section{Endomorphism rings of balanced abelian varieties} \label{sec:balanced}

\subsection{Modifying endomorphism rings by local conjugation}

In Section~\S\ref{sec:ordinary} we will benefit from a certain flexibility to modify the endomorphism ring of a balanced abelian variety. This technique already occured in \cite{Wa} and we explain it here in a form suitable for our application.

\begin{defi}
\label{defi:Tatelocallyisomorphic}
Let $A$ and $B$ are abelian varieties in $\AV_{\bF_q}$. We say that $A$ and $B$ are \emph{Tate-locally isomorphic} if 
\begin{enumerate}[label=(\roman*),align=left,labelindent=0pt,leftmargin=*,widest = (iii)]
\item for all primes $\ell \not= p$ the Tate modules $T_\ell(A)$ and $T_\ell(B)$ are isomorphic as Galois modules,
\item and $T_p(A)$ is isomorphic to $T_p(B)$ as Dieudonn\'e modules.
\end{enumerate}
Versions of this notion have appeared with various terminology, e.g., Zarhin uses `\emph{almost isomorphic}' in the context of abelian varieties over finitely generated fields over $\bQ$ in \cite{zarhin:almostisomorphic}.
\end{defi}

\begin{rmk}
Since we have canonical product decompositions
\[
T_\ell(A) = \prod_{\lambda \mid \ell} T_\lambda(A), \quad \Big(\text{resp.\ } T_p(A) = \prod_{\fp \mid p} T_\fp(A) \Big)
\]
as $R_w \otimes \bZ_\ell$-modules (resp.\ as $\cD_w$-modules),  we do not get something different in Definition~\ref{defi:Tatelocallyisomorphic} if we replace $T_\ell(A)$ and $T_p(A)$ by their more local versions $T_\lambda(A)$ and $T_\fp(A)$.
\end{rmk}

\begin{prop}
\label{prop:2outof3locallyisom}
Let $w$ be a finite set of Weil $q$-numbers. Let $A$ and $B$ be isogenous abelian varieties in $\AV_w$.  Then, if two of the properties
\begin{enumerate}[label=(\alph*),align=left,labelindent=0pt,leftmargin=*,widest = (m)]
\item \label{enumitem:2outof3a}
$A$ and $B$ are Tate-locally isomorphic,
\item \label{enumitem:2outof3b}
$A$ is $w$-locally projective, 
\item \label{enumitem:2outof3c}
$B$ is $w$-locally projective. 
\end{enumerate}
hold, then also the third holds. The same holds with `$w$-locally projective' replaced by  `$w$-locally free', or (even without assuming $A$ and $B$ being isogenous) `$w$-balanced' or `reduced $w$-balanced'.
\end{prop}
\begin{proof}
Projective $R_{w,\lambda}$-modules (resp.\ projective $\cD_{w,\fp}$-modules) are isomorphic to a multiple of a single indecomposable projective module, because $R_{w,\lambda}$ is a local ring (resp.\ by Proposition~\ref{prop:localprojectiveDieudonnestructure}). Therefore it suffices to compare ranks to deduce \ref{enumitem:2outof3a} from \ref{enumitem:2outof3b} and \ref{enumitem:2outof3c}. This is taken care of by the assumption that $A$ and $B$ be isogenous. The other implications are easy.
\end{proof}

\begin{lem}[compare with {\cite[Lemma~2.3]{zarhin:almostisomorphic}}]
\label{lem:TateONlocallyisomorphic}
Let $A$ and $B$ are abelian varieties in $\AV_{\bF_q}$. Then $A$ and $B$ are Tate-locally isomorphic if and only if for every prime number $\ell$ (including $\ell = p$) there exists an isogeny $A \to B$ of degree prime to $\ell$.

In particular, if $f: A \to B$ and $g: B \to A$ are isogenies of coprime degree, then $A$ and $B$ are Tate-locally isomorphic. 
\end{lem}
\begin{proof}
This follows from Tate's theorems recalled in \eqref{TateThmell} and \eqref{TateThmp} as follows. Given an isomorphism $\ph: T_\ell(A) \to T_\ell(B)$ as Galois or Dieudonn\'e modules, there are isogenies $f: A \to B$ and $g: B \to A$ such that $T_\ell(f)$ and $T_\ell(g)$ agree with $\ph$ respectively $\ph^{-1}$ modulo $\ell$. It follows that $f$ and $g$ are bijective on $\ell$-torsion, hence their degree is coprime to $\ell$. 

For the converse direction let $f: A \to B$ be an isogeny of degree prime to $\ell$. Then there is an isogeny $g: B \to A$ such that $g \circ f$ is multiplication by an integer prime to $\ell$, and consequently both $T_\ell(f)$ and $T_\ell(g)$ are isomorphisms.
\end{proof}

\begin{defi}
Let $w \subseteq W_q$ be a finite set of Weil $q$-numbers. An \emph{$R_w$-order} in a finite dimensional $\bQ(w)$-algebra $E$ is an $R_w$-subalgebra $\fO \subseteq E$ which is finitely generated as an $R_w$-module and contains a $\bQ(w)$-basis of $E$. 

\begin{enumerate}[label=(\arabic*),align=left,labelindent=0pt,leftmargin=*,widest = (8)]
\item
Two $R_w$-orders $\fO \subseteq E$ and $\fO' \subseteq E'$ are called \emph{Tate-locally isomorphic} if for every prime number $\ell$ (including $\ell = p$) the rings $\fO \otimes \bZ_\ell$ and $\fO' \otimes \bZ_\ell$ are isomorphic as $R_w \otimes \bZ_\ell$-algebras.
\item
Two $R_w$-orders $\fO, \fO' \subseteq E$ are called \emph{Tate-locally conjugate} if for every prime number $\ell$ (including $\ell = p$) the rings $\fO \otimes \bZ_\ell$ and $\fO' \otimes \bZ_\ell$ are conjugate inside $E \otimes \bQ_\ell$.
\end{enumerate}
\end{defi}

\begin{rmk}
\label{rmk:SkolemNoether}
We will consider $R_w$-orders in $E = \End_{\bF_q}(A) \otimes \bQ$ for abelian varieties $A \in \AV_w$. Since these $E$ are Azumaya algebras over $\bQ(w)$, it follows from the 
Skolem--Noether Theorem \cite[IV, Proposition 1.4]{milne:etcoh} that after identifying $E \simeq E'$ all isomorphisms to be considered are restrictions of inner automorphisms of $E$. 
In this case the term \emph{Tate-locally isomorphic} coincides with the notion of being \emph{Tate-locally conjugate}.
\end{rmk}

\begin{lem}
\label{lem:pushEndprimetodegree}
Let $f: A \to B$ be an isogeny in $\AV_{w}$, and let $\ell$ be a prime number 
that is coprime to the degree of $f$. Then the isomorphism
\[
f(-)f^{-1} : \End_{\bF_q}(A) \otimes \bQ \xrightarrow{\sim} \End_{\bF_q}(B) \otimes \bQ
\]
maps the $R_w$-order $\End_{\bF_q}(A)$ to an $R_w$-order that agrees with $\End_{\bF_q}(B)$ after completion at $\ell$, i.e.\ considered as coherent algebras on $\Spec(R_w)$, these orders agree Zariski locally at $(\ell)$.
\end{lem}
\begin{proof}
We may complete $\ell$-adically and use Tate's theorem. We will only discuss the case $\ell \not= p$. 
The case of $\ell=p$ works \textit{mutatis mutandis}. 
We are led to consider the isomorphism
\[
f(-)f^{-1} : \End_{R_w \otimes \bQ_\ell}(V_\ell(A)) \xrightarrow{\sim}  \End_{R_w \otimes \bQ_\ell}(V_\ell(B))
\]
But by assumption and the proof of Lemma~\ref{lem:TateONlocallyisomorphic} the map 
\[
T_\ell(f)  : T_\ell(A) \to T_\ell(B)
\]
is an isomorphism. The claim follows at once from Tate's theorems \eqref{TateThmell} (resp.\  \eqref{TateThmp}).
\end{proof}

\begin{prop}
\label{prop:modifyEnd}
Let $A$ be an abelian variety in $\AV_{w}$.  Let $\fO$ be an $R_w$-order in $\End_{\bF_q}(A) \otimes \bQ$ that is Tate-locally
isomorphic to $\End_{\bF_q}(A)$. Then there is an abelian variety $B$ that is Tate-locally isomorphic to $A$ and an
$\bF_q$-isogeny $f:A\to B$ such that the induced isomorphism of
$R_w\otimes\bQ$-algebras 
\[
\End_{\bF_q}(A)\otimes\bQ \xrightarrow{f(-)f^{-1}} \End_{\bF_q}(B)\otimes\bQ
\]
restricts to an isomorphism $\fO\simeq \End_{\bF_q}(B)$  of $R_w$-orders. 
\end{prop}

\begin{proof}
Since $\fO$ and $\End_{\bF_q}(A)$ are both $R_w$-orders of $\End_{\bF_q}(A) \otimes \bQ$ there is only a finite set $\Sigma$ of bad prime numbers $\ell$ (including potentially $\ell=p$) for which $\End_{\bF_q}(A)$ and $\fO$ do not agree after completion at $\ell$. We argue by induction on the size $\#\Sigma$. If $\Sigma$ is empty, then $\End_{\bF_q}(A) = \fO$, and we are done. 

We now assume that $\ell \in \Sigma$ and construct an isogeny $f: A \to B$ of degree a power of $\ell$ such that $A$ and $B$ are Tate-locally isomorphic and the set of bad primes for $B$ with respect to $\fO \simeq f \fO f^{-1} \subseteq \End_{\bF_q}(B) \otimes \bQ$ is contained in $\Sigma \setminus \{\ell\}$. 

By Remark~\ref{rmk:SkolemNoether} there is an element $\hat\ph \in \End_{\bF_q}(A) \otimes \bQ_\ell$ such that 
\[
\hat\ph (\fO \otimes \bZ_\ell)  \hat\ph^{-1} = \End_{\bF_q}(A) \otimes \bZ_\ell
\]
in $\End_{\bF_q}(A) \otimes \bQ_\ell$. Since the stabilizer of $\fO \otimes \bZ_\ell$ is open, we may approximate $\hat \ph$ by an element $\ph \in \End_{\bF_q}(A)$ which still conjugates $\fO \otimes \bZ_\ell$ onto $\End_{\bF_q}(A) \otimes \bZ_\ell$. By primary decomposition of $\ker(\ph)$ we may factor $\ph : A \to A$ as
\[
\ph \colon A \xrightarrow{f} B \xrightarrow{g} A
\]
such that $f$ has degree a power of $\ell$ and $g$ has degree prime to $\ell$. It follows from 
Lemma~\ref{lem:TateONlocallyisomorphic} that $B$ is Tate-locally isomorphic to $A$. 

Moreover, Lemma~\ref{lem:pushEndprimetodegree} applied to $f$ tells us that primes different from $\ell$ at which $\fO$ and $\End_{\bF_q}(A)$ agree locally, remain primes at which $f\fO f^{-1}$ agrees locally with $\End_{\bF_q}(B)$.
In particular the bad primes for $B$ with respect to $f \fO f^{-1}$ are contained in $\Sigma$. 

It remains to show that the situation has improved at $\ell$. For that we compare via $g$ and again by Lemma~\ref{lem:pushEndprimetodegree}, now applicable to $\ell$ since $g$ is of degree coprime to $\ell$. It follows that 
\begin{align*}
\End_{\bF_q}(B) \otimes \bZ_\ell = g^{-1}(\End_{\bF_q}(A) \otimes \bZ_\ell) g & = f( \ph^{-1} (\End_{\bF_q}(A) \otimes \bZ_\ell) \ph)f^{-1}  \\
& = f(\fO \otimes \bZ_\ell) f^{-1} = f \fO f^{-1} \otimes \bZ_\ell.
\end{align*}
This concludes the inductive step and thus proves the proposition.
\end{proof}

\subsection{A structure theorem for endomorphism rings of balanced objects}
\label{sec:endomorphisms}

We now describe  the endomorphism ring $S_w = \End_{\bF_q}(A_w)$ for a particular choice of a $w$-balanced abelian variety $A_w$  in $\AV_{\bF_q}$. A similar description also holds for a reduced $w$-balanced abelian variety by cancelling a factor of $2$ at various places. 
But first we need to choose some data and fix notation.

Let $w$ be a finite set of Weil $q$-numbers. Recall that $\bQ(w) = R_w \otimes \bQ = \prod_{\pi \in w} \bQ(\pi)$. Then 
\[
\cS^0(w) = \prod_{\pi \in w} \rM_{m_\pi}(E_\pi)
\]
is an Azumaya algebra over $\bQ(w)$ of degree $2\expo$ (so locally a form of $\rM_{2\expo}(-)$) with the local invariants as specified by Tate's formulas recalled in Theorem~\ref{thm:TateEndostructure}~\ref{thmitem:invariantatp}. It follows that the desired $S_w$ is an $R_w$-order in $\cS^0(w)$. 

Waterhouse proves in \cite[Theorem 3.13]{Wa} that there is a simple abelian variety $B_\pi \in \AV_\pi$ such that 
$\End_{\bF_q}(B_\pi)$ is a maximal order  $\fO_\pi$ in $E_\pi$. We set  $B_w = \prod_{\pi} B_\pi^{m_\pi}$ and define
\[
\tilde{\cS}(w) :=  \End_{\bF_q}(B_w) = \prod_{\pi \in w} \rM_{m_\pi}(\fO_\pi)  \subseteq  \prod_{\pi \in w} \rM_{m_\pi}(E_\pi) 
= \cS^0(w).
\]
Then $\tilde\cS(w)$ is a maximal $R_w$-order of $\cS^0(w)$. We next choose a splitting for all $\ell \not= p$ and for $p$ as
\begin{align*}
\psi_\ell: \tilde{\cS}(w)  \inj \cS^0(w) \otimes \bQ_\ell & \xrightarrow[\sim]{\psi_\ell^0} \rM_{2\expo}(\bQ(w) \otimes \bQ_\ell),  \\
\psi_p: \tilde{\cS}(w)  \inj  \cS^0(w) \otimes \bQ_p &\xrightarrow[\sim]{\psi_p^0} \rM_{2}(\cD^0_w) .
\end{align*}

Now, for every prime number $\ell \not= p$, the proof of the Skolem-Noether theorem (see also Proposition~\ref{prop:rationalfreeness}) shows that there is an isomorphism of Galois modules
\[
h_\ell \colon  (\bQ(w) \otimes \bQ_\ell)^{\oplus 2\expo} \xrightarrow{\sim}  V_\ell(B_w)
\]
such that conjugation $h_\ell(-)h_\ell^{-1}$ agrees with 
\[
\rM_{2\expo}(\bQ(w) \otimes \bQ_\ell) \xleftarrow[\sim]{\psi_\ell^0} \cS^0(w) \otimes \bQ_\ell \xrightarrow[\sim]{\rm Tate} 
\End_{\bQ(w) \otimes \bQ_\ell}(V_\ell(B_w)) 
\]
And, similarly for $p$, we have an isomorphism of Dieudonn\'e modules
\[
h_{p} \colon   V_p(B_w)  \xrightarrow{\sim}  (\cD_w^0)^{\oplus 2}
\]
such that conjugation $h_{p}(-)h_p^{-1}$ agrees with  (note that we have passed to the opposite rings)
\[
 \End_{\cD_w^0}( (\cD_w^0)^{\oplus 2})^\op = \rM_2(\cD_w^0) \xleftarrow[\sim]{\psi^0_p}  \cS^0(w) \otimes  \bQ_p \xrightarrow[\sim]{\rm Tate} \End_{\cD_w^0}(V_p(B_w))^\op
\]
Now we perform the construction of a $w$-balanced abelian variety from the proof of Theorem~\ref{thm:integralstructure}. We choose $N$ coprime to $p$ and large enough such that $R_w \otimes \bZ[\frac{1}{Np}]$ is a product of Dedekind rings. For all prime numbers $\ell \nmid Np$ we may modify $h_\ell$ and perforce $\psi^0_\ell$ and $\psi_\ell$ so that $h_\ell$ induces  an isomorphism of integral structures 
\[
h_\ell \colon (R_w \otimes \bZ_\ell)^{\oplus 2\expo} \xrightarrow{\sim} T_\ell(B_w).
\]
Moreover, we scale suitably $h_\ell$ for $\ell \mid N$ and $h_p$ such that these maps restrict to maps
\[
h_\ell \colon  (R_w \otimes \bZ_\ell)^{\oplus 2\expo} \inj T_\ell(B_w),
\]
and 
\[
h_{p} \colon   T_p(B_w)  \inj  (\cD_w)^{\oplus 2}.
\]
Using the standard dictionary translating lattices in Tate and Dieudonn\'e modules into isogenies, it follows that there is a corresponding isogeny 
\[
h \colon A_w \to B_w
\]
of degree a product of primes dividing $Np$, such that $h_\ell = T_\ell(h)$ for all $\ell \mid N$ and $h_p = T_p(h)$. It follows by construction and the proof of Lemma~\ref{lem:TateONlocallyisomorphic}
that $A_w$ is indeed $w$-balanced. Moreover, we have the following description of its endomorphism ring. 

\begin{thm}
\label{thm:Sw}
With the above notation the endomorphism ring $S_w = \End_{\bF_q}(A_w)$ of the $w$-balanced abelian variety $A_w$  constructed above sits in the following cartesian square:
\[
\xymatrix@M+1ex{
S_w \ar[r] \ar[d] & \tilde{\cS}(w) \otimes \bZ[\frac{1}{Np}] \ar[d]^{ (\psi_p, \psi_\ell)} \\
{\displaystyle \rM_2(\cD_w) \times \prod_{\ell \mid N} \rM_{2\expo}(R_w \otimes \bZ_\ell)} \ar[r] & 
{\displaystyle \rM_2(\cD^0_w) \times \prod_{\ell \mid N} \rM_{2\expo}(\bQ(w) \otimes \bQ_\ell)}.
}
\]
\end{thm}
\begin{proof}
The cartesian square is essentially nothing but the obvious cartesian square   (fpqc descent)
\[
\xymatrix@M+1ex{
\End_{\bF_q}(A_w) \ar[r] \ar[d] & \End_{\bF_q}(A_w) \otimes \bZ[\frac{1}{Np}] \ar[d]  \\
{\displaystyle \End_{\bF_q}(A_w) \otimes \bZ_p \times \prod_{\ell \mid N} \End_{\bF_q}(A_w) \otimes \bZ_\ell} \ar[r] & 
{\displaystyle \End_{\bF_q}(A_w) \otimes \bQ_p \times \prod_{\ell \mid N} \End_{\bF_q}(A_w) \otimes \bQ_\ell} . }
\]
suitably translated by isomorphisms as follows. In view of $T_\ell(A_w) = (R_w \otimes \bZ_\ell)^{\oplus 2\expo}$ for $\ell \not= p$ (used here only for $\ell \mid N$) and $T_p(A_w) = \cD_w^{\oplus 2}$ as by construction, Tate's theorem translates the bottom row into the  inclusion
\[
\rM_2(\cD_w) \times \prod_{\ell \mid N} \rM_{2\expo}(R_w \otimes \bZ_\ell)
\to
\rM_2(\cD^0_w) \times \prod_{\ell \mid N} \rM_{2\expo}(\bQ(w) \otimes \bQ_\ell).
\]
It remains to identify the right vertical arrow. We may compare $A_w$ with $B_w$ along $h: A_w \to B_w$ as in the following diagram, in which the top arrow is an isomorphism because $\deg(h)$ is invertible in $\bZ[\frac{1}{Np}]$ in combination with  Lemma~\ref{lem:pushEndprimetodegree}. 
\[
\xymatrix@M+0.4ex{
\End_{\bF_q}(A_w) \otimes \bZ[\frac{1}{Np}]  \ar[rr]^{h(-)h^{-1}}_\sim \ar[d] && \End_{\bF_q}(B_w) \otimes \bZ[\frac{1}{Np}] = \tilde{\cS}(w)\otimes \bZ[\frac{1}{Np}]  \ar[d]  \\
{\displaystyle \End_{\bF_q}(A_w) \otimes \bQ_p \times \prod_{\ell \mid N} \End_{\bF_q}(A_w) \otimes \bQ_\ell} 
\ar[rr]^{h(-)h^{-1}}_\sim
 && 
{\displaystyle \End_{\bF_q}(B_w) \otimes \bQ_p \times \prod_{\ell \mid N} \End_{\bF_q}(B_w) \otimes \bQ_\ell} }
\]
Now we translate the bottom row of the previous diagram into the top row of the following diagram due to Tate's theorem. 
\[
\xymatrix@M+0.4ex{
{\displaystyle \End(V_p(A_w))^\op \times \prod_{\ell \mid N} \End(V_\ell(A_w)) } 
\ar[rrr]^{(h_p(-)h_p^{-1} , h_\ell(-)h_\ell^{-1})}_\sim
\ar@{=}[d] &&& 
{\displaystyle \End(V_p(B_w))^\op \times \prod_{\ell \mid N} \End(V_\ell(B_w)) } \\
{\displaystyle \rM_2(\cD^0_w) \times \prod_{\ell \mid N} \rM_{2\expo}(\bQ(w) \otimes \bQ_\ell)}  &&& 
{\displaystyle \cS^0(w) \otimes \bQ_p \times \prod_{\ell \mid N} \cS^0(w) \otimes \bQ_\ell} 
\ar[lll]_{(\psi^0_p, \psi^0_\ell)}^\sim \ar[u]_(0.4){\rm Tate}^(0.4)\simeq.
}
\]
As the bottom square commutes by the definition of $h_p$ and $h_\ell$, we have identified the right vertical map in the asserted cartesian square as the map claimed in the theorem.
\end{proof}

\begin{rmk}
The ring $S_w$ constructed in Theorem~\ref{thm:Sw} considered as a coherent algebra over $\Spec(R_w)$  differs from a maximal order in the appropriate Azumaya algebra over $\bQ(w)$ at most above prime numbers $\ell$ at which $R_w$ is singular and those above $p$. The construction can be performed more carefully so that $S_w$ only differs from a maximal order at most in the singularities of $\Spec(R_w)$ and in the supersingular locus, the vanishing locus of  the ideal $(F,V,p)$.
\end{rmk}

\section{Representable functors on \texorpdfstring{$\AV_w$}{abelian varieties with Weil support in w}}
\label{sec:representablefunctors}

In this section we spell out abstract  Morita equivalence, originally formulated with abelian categories as for example in \cite[Chapter 4]{reiner:maxorder}, in the context of the additive category $\AV_{\bF_q}$.

\smallskip

Let $w\subseteq W_q$ be a finite set of conjugacy classes of Weil $q$-numbers, and let $A$ be any object of $\AV_w$. Denote by
\[
S(A) = \End_{\bF_q}(A)
\]
the $R_w$-algebra of $\bF_q$-endomorphisms of $A$, and by
\begin{equation}\label{equation:functorM_X}
T_A \colon \AV_w \longrightarrow \Mod_{\Ztf}(S(A)), \qquad  T_A(X) = \Hom_{\bF_q}(X,A)
\end{equation}
the contravariant functor represented by $A$, viewed as valued in the category of left $S(A)$-modules that are torsion free as $\bZ$-modules.
Just like $\AV_w$, the category $\Mod_{\Ztf}(S(A))$ has a natural $R_w$-linear structure, since the center of $S(A)$ is an $R_w$-algebra in a natural way.
Moreover, it is clear that the functor $T_A(-)$ is $R_w$-linear.

The main result of this section is Theorem~\ref{thm:formal}, a twofold criterion for deciding when \eqref{equation:functorM_X} induces
an anti-equivalence.

\subsection{The maps \texorpdfstring{$I_A$}{IA} and \texorpdfstring{$G_A$}{GA}}
The following constructions go back to Waterhouse \cite{Wa}.
Let $X$ be any object of $\AV_w$, and let $H\subseteq X$ be any closed subgroup scheme of $X$.
Define the kernel of restriction to $H$:
\[
I_A(H)=\{f\in T_A(X): H\subseteq \ker(f) \} = \ker\big(\Hom_{\bF_q}(X,A) \to \Hom_{\bF_q}(H,A)\big),
\]
a left $S(A)$-submodule of $T_A(X)=\Hom_{\bF_q}(X, A)$. Notice that the quotient $X/H$ as group scheme is in fact an abelian variety in $\AV_w$ and the pull-back
via the quotient map $\psi_H:X\to X/H$ gives an identification of left $S(A)$-modules
\begin{equation}\label{esssurj}
T_A(X/H)\stackrel{\sim}{\longrightarrow}I_A(H).
\end{equation}

Conversely, if $M\subseteq T_A(X)$ is a left $S(A)$-submodule, define
\[
G_A(M)=\bigcap_{r\in M}\ker(r)
\]
as the closed subgroup scheme of $X$ given by the scheme-theoretic intersection of all kernels of elements of $M$. Notice that the above intersection equals an intersection of finitely many kernels, since $M$ is a finitely generated abelian group. We collect in the next proposition the basic properties of the maps $I_A$ and $G_A$.

\begin{prop}\label{prop:IGproperties} 
Let $X$ be any object of $\AV_w$. For any closed subgroup schemes $H, H'$ of $X$ and any left $S(A)$-submodules
$M, M'$ of $T_A(X)$ the following properties hold:
\begin{enumerate}[label=(\arabic*),align=left,labelindent=0pt,leftmargin=*,widest = (8)]
\setlength\itemsep{0.2em}
\item\label{IGprop1} $H\subseteq H'\implies I_A(H)\supseteq I_A(H')$, 
\item\label{IGprop2} $M\subseteq M'\implies G_A(M)\supseteq G_A(M')$,
\item\label{IGprop3} $H\subseteq G_A(I_A(H))$,
\item\label{IGprop4} $M\subseteq I_A(G_A(M))$,
\item\label{IGprop5} $I_A(H)=I_A(G_A(I_A(H)))$,
\item\label{IGprop6} $G_A(M)=G_A(I_A(G_A(M)))$,
\item\label{IGprop7} $I_A(X[n]) = n T_A(X)$, for all integers $n\geq 1$,
\item\label{IGprop8} $G_A(n T_A(X))= X[n]$, for all integers $n\geq 1$.
\end{enumerate}
\end{prop}
\begin{proof} The properties \ref{IGprop1} to \ref{IGprop4}, \ref{IGprop7} and \ref{IGprop8} easily follow from the definition of $I_A$ and $G_A$.
To see \ref{IGprop5} and \ref{IGprop6}, notice that the inclusions
\[
I_A(H)\subseteq I_A(G_A(I_A(H)))\:\:\text{ and }\:\: G_A(M)\subseteq G_A(I_A(G_A(M)))
\]
are the special cases \ref{IGprop4} for $M = I_A(H)$ and \ref{IGprop3} for $H=G_A(M)$. 
The corresponding opposite inclusions can be deduced after applying
properties \ref{IGprop1} and \ref{IGprop2} to the inclusions expressed by \ref{IGprop3} and \ref{IGprop4} respectively.
\end{proof}

\begin{rmk} 
\label{rmk:finite-cofinite}
A consequence of Proposition~\ref{prop:IGproperties} is that $I_A(H)\subseteq T_A(X)$ has finite index if and only if $H\subseteq X$ is
finite, and that $G_A(M)\subseteq X$ is a finite subgroup if and only if $M\subseteq T_A(X)$ has finite index.
\end{rmk}

\begin{rmk} 
We point out that the map $G_A$ has already been considered by Waterhouse for $X=A$, see \cite[\S3.2]{Wa}. 
In this circumstance, a finite index submodule $M\subseteq T_A(A) = \End_{\bF_q}(A)$ is an ideal containing an isogeny, and $G_A(M)$, denoted $H(M)$ in \emph{loc.~cit.}, is studied via the isogeny $A \to A/G_A(M)$. The converse construction $I_A(H)$ yields a kernel ideal in the terminology of \cite[p.~533]{Wa}. Waterhouse proves
that if $M$ is a kernel ideal then the structure of $A/H(M)$ depends only on $M$ as an $\End_{\bF_q}(A)$-module \cite[Theorem~3.11]{Wa}, He moreover proves in \cite[Theorem 3.15]{Wa} that for an abelian variety $A$ whose endomorphism ring is a maximal order in $\End(A) \otimes \bQ$ any finite index submodule $M\subseteq T_A(A)$ arises as  a kernel ideal.
\end{rmk}

\subsection{A criterion for a functor to be an anti-equivalence of categories}
We now state and prove the main theorem of the section.
Condition \ref{thmitem:formal c} of Theorem~\ref{thm:formal} is inspired by the use of injective cogenerators in embedding theorems for abelian categories, see \cite{freyd}. Note that the category $\AV_w$ is additive as a full subcategory of the abelian category of all finite type group schemes over the field $\bF_q$, see \cite[Exp. VI, \S5.4 Th\'eor\`eme]{sga3}. But $\AV_w$ is not an abelian category, because an isogeny of abelian varieties is both a monomorphism and an epimorphism in $\AV_w$ without necessarily being an isomorphism. 

\begin{thm} 
\label{thm:formal}
Let $w \subseteq W_q$ be a finite subset and let $A \in \AV_w$ be an abelian variety. Let 
\[
T_A \colon  \AV_w \longrightarrow  \Mod_\Ztf(S(A)), \qquad  T_A(X) = \Hom_{\bF_q}(X,A),
\]
be the functor represented by $A$. We consider the following statements.
\begin{enumerate}[label=(\alph*),align=left,labelindent=0pt,leftmargin=*,widest = (m)]
\item
\label{thmitem:formal a}
$T_A$ is an anti-equivalence of categories: 
\begin{enumerate}[label = (a$_\text{\arabic*}$) \ , , ref = (a$_\text{\arabic*}$)] 
\item 
\label{thmitem:formal a1}
$T_A$ is fully faithful, 
\item 
\label{thmitem:formal a2}
$T_A$ is essentially surjective.
\end{enumerate}

\item
\label{thmitem:formal b}
For all $X \in \AV_w$, the assignments $H\mapsto I_A(H)$ and $M\mapsto G_A(M)$ are mutually inverse maps that describe a bijection
\[
\left\{H \subseteq X \ ; \ 
\begin{array}{c}
\text{finite} \\
\text{subgroup scheme}
\end{array}
\right\}
\longleftrightarrow 
\left\{M \subseteq T_A(X) \ ; \ 
\begin{array}{c}
\text{finite index} \\
\text{$S(A)$-submodule}
\end{array}
\right\}.
\]
More precisely:
\begin{enumerate}[label = (b$_\text{\arabic*}$) \ , ref = (b$_\text{\arabic*}$)] 
\item 
\label{thmitem:formal b1}
for all $X \in \AV_w$ and all finite subgroup schemes $H \subseteq X$, we have
\[
H=G_A(I_A(H)),
\]
\item
\label{thmitem:formal b2}
for all $X \in \AV_w$ and all $S(A)$-submodules $M \subseteq T_A(X)$ of finite index, we have
\[
M=I_A(G_A(M)).
\]
\end{enumerate}
\item
\label{thmitem:formal c}
$A$ is an injective cogenerator:
\begin{enumerate}[label = (c$_\text{\arabic*}$) \ , ref = (c$_\text{\arabic*}$)] 
\item 
\label{thmitem:formal c1}
for all $X \in \AV_w$, there exist an $n \in \bN$ and an injective homomorphism
\[
X \inj A^n,
\]
\item
\label{thmitem:formal c2}
for any $X,Y \in \AV_w$ and any injective homomorphism $\ph: X \inj Y$, the induced map 
\[
\ph^\ast: T_A(Y)  \surj T_A(X)
\]
is surjective.
\end{enumerate}
\end{enumerate}
We have the following two chains of equivalences
\[
\ref{thmitem:formal a}  \iff  \ref{thmitem:formal b} \iff \ref{thmitem:formal c} ,
\]
and
\[
\textrm{\emph{\ref{thmitem:formal a1}}}  \iff  \textrm{\emph{\ref{thmitem:formal b1}}}   \iff  \textrm{\emph{\ref{thmitem:formal c1}}}.
\]
\end{thm}

\begin{proof}
\ref{thmitem:formal a1} $\Longrightarrow$ \ref{thmitem:formal b1}: 
Let $H \subseteq X$ be a finite subgroup scheme, and set $G=G_A(I_A(H))$. By Remark~\ref{rmk:finite-cofinite} also $G$ is a finite subgroup scheme. Since the inclusion $H \subseteq G$ trivially holds, see Proposition~\ref{prop:IGproperties}, there is a natural isogeny
\[
\ph : X/H \longrightarrow X/G.
\]
By the very definition of $G_A(-)$, it follows that the induced map
\[
\ph^\ast : T_A(X/G) \longrightarrow T_A(X/H) 
\]
is an isomorphism.
Since by assumption the functor $T_A(-)$ is fully faithful, we conclude that $\ph$ is an isomorphism as well and therefore $H = G = G_A(I_A(H))$.

\ref{thmitem:formal b1} $\Longrightarrow$ \ref{thmitem:formal c1}: 
This is just the special case $H = 1$. Let $\{\psi_1, \ldots, \psi_n\}$ be a set of generators of $T_A(X)$ as an $S(A)$-module and consider the homomorphism $\psi : X \to A^n$ whose $i$-th component is $\psi_i$. Since the $\psi_i$'s generate $T_A(X)$, we have
\[
\ker(\psi)=G_A(T_A(X)).
\]
Assuming \ref{thmitem:formal b1}, we have $1=G_A(I_A(1))$; hence we deduce
\[
\ker(\psi) = G_A(T_A(X)) = G_A(I_A(1)) = 1,
\]
which proves \ref{thmitem:formal c1}.

\ref{thmitem:formal c1} $\Longrightarrow$ \ref{thmitem:formal a1}: 
We must show that for all abelian varieties $X,Y \in \AV_w$ the map 
\begin{equation}\label{eq:M_is_fully_faithful}
\Hom_{\bF_q}(X, Y) \longrightarrow \Hom_{S(A)}(T_A(Y), T_A(X))
\end{equation}
is bijective. 

We first show that \eqref{eq:M_is_fully_faithful} is injective. By assumption $(c_1)$ there is an injective map $\psi : Y \inj A^n$, denote by $\psi_i$ its
$i$-th component. If $f: X \to Y$ is a homomorphism with 
\[
0 = f^\ast : T_A(Y) \longrightarrow T_A(X),
\]
then $\psi_i \circ f = 0$ for all $i$, and thus $\psi \circ f = 0$. But since $\psi$ is injective this implies $f =0$.

We next address the surjectivity of  \eqref{eq:M_is_fully_faithful}. Let, as above, $\psi: Y\inj A^n$ be an injection. Since $T_A(Y)$ is a finitely generated $S(A)$-module,
we may also assume that $\psi$ is chosen so that its components $\psi_1, \ldots, \psi_n : Y \to A$ generate $T_A(Y)$. Construct the quotient $\pr: A^n \to Z = A^n/Y$, which is itself an object of $\AV_w$, and
choose an injection $\iota : Z \inj A^m$, which exists again by assumption. Setting $\sigma = \iota \circ \pr:A^n\to A^m$, we obtain a short exact sequence 
(a co-presentation)
\[
0 \to Y \xrightarrow{\psi} A^n \xrightarrow{\sigma} A^m.
\]
We may view $\sigma$ as an $m\times n$ matrix $\sigma = (s_{ij}) \in \rM_{m \times n}(S(A))$ with entries in $S(A)$. Let now
\[
g: T_A(Y) \to T_A(X)
\]
be any map of $S(A)$-modules, and let $\ph:X\to A^n$ be the morphism whose $i$-th component is given by $\ph_i = g(\psi_i) : X \to A$, where $1\leq i\leq n$.
Let ${(\sigma \circ \ph)}_j$ be the $j$-th component of the composition $\sigma\circ\ph$, where $1\leq j\leq m$. From the $S(A)$-linearity of $g$ we deduce for any $j$
\[
{(\sigma \circ \ph)}_j = \sum_{1\leq k\leq n} s_{jk} \ph_k = \sum_{1\leq k\leq n} s_{jk} g(\psi_k) = g(\sum_{1\leq k\leq n} s_{jk} \psi_k) = g((\sigma \circ \psi)_j) = 0.
\]
Therefore $\ph$ factors as $\psi\circ f$, for a unique map $f : X \to Y$. Since 
\[
f^\ast(\psi_i) = \psi_i \circ  f = \ph_i = g(\psi_i),
\]
the maps $f^\ast$ and $g : T_A(Y) \to T_A(X)$ agree on a generating set, hence $g = f^\ast$ and \eqref{eq:M_is_fully_faithful} is indeed surjective.

\smallskip

After having established the equivalence of the assertions \ref{thmitem:formal a1}, \ref{thmitem:formal b1} and \ref{thmitem:formal c1} we show that of 
\ref{thmitem:formal a}, \ref{thmitem:formal b} and \ref{thmitem:formal c}. We recall that $A$ has Weil support equal to $w$ by assumption.

\smallskip

\ref{thmitem:formal a} $\Longrightarrow$ \ref{thmitem:formal b2}: Let $X$ be an object of $\AV_w$ and let $M \subseteq T_A(X)$ be a submodule of finite index. Thanks to the the fact that $T_A(-)$ is essentially surjective by assumption, there is an object $Y$ of $\AV_w$ and an isomorphism $T_A(Y) \simeq M$ of $S(A)$-modules. We deduce an injective $S(A)$-homomorphism $\iota:T_A(Y) \inj T_A(X)$ with finite cokernel and which is induced from a homomorphism $\ph : X \to Y$, because $T_A(-)$ is assumed to be fully faithful.

Since $\iota\otimes\bQ:\Hom_{\bF_q}(Y, A) \otimes\bQ \to \Hom_{\bF_q}(X, A) \otimes\bQ$
is an isomorphism and $A$ has Weil support equal to $w$, the map $\ph$ must be an isogeny. It then follows from the definition that
$M = I_A(\ker(\ph))$. Applying now property \ref{IGprop5} of Proposition~\ref{prop:IGproperties} to the subgroup $\ker(\ph)$ we find
\[
M = I_A(\ker(\ph)) = I_A(G_A(I_A(\ker(\ph))) = I_A(G_A(M)).
\]

\ref{thmitem:formal b} $\Longrightarrow$ \ref{thmitem:formal c2}:  Let $\ph: X \inj Y$ be an injective homomorphism in $\AV_w$. We have to show that an arbitrary map $f: X \to A$ extends to a certain $\tilde{f} : Y \to A$. By composing $\ph$ with an inclusion $Y \inj A^n$, which exists by \ref{thmitem:formal c1}, we may assume that $Y=A^n$. 

Let $\ph_i :  X \to A$ be the $i$-th component of $\ph$ and denote by $M$ the $S(A)$-submodule of $T_A(X)$ generated by all the $\ph_1, \ldots, \ph_n$. 
Poincar\'e's theorem of complete reducibility implies that $X$ is a direct factor of $A^n$ up to isogeny, i.e.\ there is an abelian subvariety $Z \subseteq A^n$ such that $\ph: X \inj A^n$ factors as the inclusion of $X \inj X \times Z$ followed by an isogeny
\[
\psi \colon X \times Z \xrightarrow{} A^n.
\]
Then $f$ extends as the composition $f \circ \pr_1$ with the first projection to a map $X \times Z \to A$, and the nonzero multiple $g=mf$ with $m = \deg(\psi)$ extends to a morphism
$\tilde g:A^n\to A$. This means precisely that $g$ belongs to $M$. In this way we see that $M$ has finite index in $T_A(X)$ and hence, by assumption \ref{thmitem:formal b2}, we have 
\[
M = I_A(G_A(M)) = I_A(\bigcap_i \ker(\ph_i)) = I_A(\ker(\ph)) = I_A(1) = T_A(X).
\]
Thus there are $s_i \in S(A)$ with $f = \sum_i s_i \ph_i$ and therefore the map $s : A^n \to A$ defined as
\[
s=\sum_{i} s_i\circ\pr_i,
\]
where $\pr_i:A^n\to A$ is the $i$-th projection, extends $f$ as desired.

\ref{thmitem:formal c} $\Longrightarrow$ \ref{thmitem:formal a2}:  Let $M$ be an object of $\Mod_{\Ztf}(S(A))$ and choose a finite presentation 
\[
S(A)^{\oplus m} \xrightarrow{f} S(A)^{\oplus n} \to M \to 0.
\]
By $S(A) = T_A(A)$ and additivity of $T_A(-)$ the map $f$ comes from a $\ph: A^n \to A^m$. Let $Y$ be the image of $\ph$ and let $X = \ker(\ph)^0$ be the reduced connected component of the kernel. Both $X$ and $Y$ are objects of $\AV_w$. Let $G = \ker(\ph)/X$ be the finite group scheme of connected components. 
Since the functor $T_A(-)$ is left exact as a functor on group schemes over $\bF_q$, we obtain a diagram of $S(A)$-modules
\[
\xymatrix@M+1ex{
 & & S(A)^{\oplus m} \ar[d]^f \ar@{->>}[dl]& T_A(G) \ar[d] \\
 0 \ar[r] & T_A(Y) \ar[r] & S(A)^{\oplus n} \ar[r]^(0.4)g \ar[d] \ar@{->>}[dr]^{i^\ast} & T_A(\ker(\ph)) \ar[d] \\
 & & M \ar[d] \ar@{.>}[r]^(0.4)h & T_A(X) \\
 & & 0 & }
\]
with an exact middle row and exact columns. The map $S(A)^{\oplus m} \surj T_A(Y)$ comes from $Y \inj A^m$ and is surjective by assumption \ref{thmitem:formal c2}. The same applies to $i^\ast : S(A)^{\oplus n} \surj T_A(X)$ which comes from the inclusion $i: X \inj \ker(\ph) \inj A^n$. It follows that the map denoted $g$ induces an injective map $M \inj T_A(\ker(\ph))$. Because $T_A(G)$ is finite, we deduce that the composite $h$ 
\[
h: M \longrightarrow T_A(X)
\]
has finite kernel, and moreover is surjective due to $i^\ast$ being surjective. But $M$ is torsion free by assumption, hence $h$ is an isomorphism showing that $M$ lies in the essential image of $T_A(-)$.
\end{proof}

\section{The truncated anti-equivalence for finite Weil support}
\label{sec:truncequiv}

Let $w\subseteq W_q$ be any finite subset fixed throughout the section, and let $A_w$ be a $w$-locally projective abelian variety (see Definition~\ref{defi:locallyprojective}) with $\bF_q$-endomorphism ring $\End_{\bF_q}(A_w)$ denoted by $S_w$.
In this section we show that the functor represented by $A_w$
\[
T_w \colon \AV_w\longrightarrow \Mod_{\Ztf}(S_w), \qquad  T_w(X) = \Hom_{\bF_q}(X,A_w)
\]
is an anti-equivalence of categories if $A_w$ has full support $w(A_w) = w$. 

The ring $S_w$ is a finite $R_w$-algebra and free as a $\bZ$-module. If $\lambda$ (resp.\ $\fp$) is a maximal ideal of $R_w \otimes \bZ_\ell$ for $\ell \not= p$ (resp.\ of $R_w \otimes \bZ_p$), we introduce notation 
\[
S_{w,\lambda} = S_w \otimes_{R_w} R_{w,\lambda} \quad \text{ and } \quad S_{w,\fp} = S_w \otimes_{R_w} R_{w,\fp}.
\]
It follows from the local version of Tate's theorem, Proposition~\ref{prop:localTateTheorem}, that $T_\lambda(-)$ and $T_\fp(-)$ induce natural isomorphisms
\begin{align*}
S_{w,\lambda} & = \End_{R_{w,\lambda}}\big(T_\lambda(A_w) \big), \\
S_{w,\fp} & = \End_{R_{w,\fp}}\big(T_\fp(A_w) \big)^\op.
\end{align*}

\subsection{Preliminary lemmata}

We now present four lemmata that are needed when proving that $T_w(-)$ is an anti-equivalence of categories.
The first lemma clarifies an assumption on the support. 

\begin{lem}
\label{lem:support}
Let $A$ be an abelian variety in $\AV_w$ with support $w(A) = w$. 
\begin{enumerate}[label=(\arabic*),align=left,labelindent=0pt,leftmargin=*,widest = (8)]
\item
\label{lemitem:support1}
$T_\lambda(A)$ is nontrivial for all maximal ideals $\lambda$ of $R_w \otimes \bZ_\ell$ for $\ell \not= p$.
\item
\label{lemitem:support2}
$T_\fp(A)$ is nontrivial for all maximal ideals $\fp$ of $R_w \otimes \bZ_p$.
\end{enumerate}
\end{lem}
\begin{proof}
Since $A$ has support $w$, there is an $n \geq 1$ and an isogeny $B \times A' \to A^n$ with $B$ being a $w$-balanced abelian variety. By Definition~\ref{defi:balanced} and Definition~\ref{defi:localTate modules}, we have $T_\lambda(B)  \simeq (R_{w,\lambda})^{\oplus 2\expo}$. The induced map 
\[
(R_{w,\lambda})^{\oplus 2\expo} \simeq T_\lambda(B) \inj T_\lambda(A^n) = T_\lambda(A)^{\oplus n},
\]
is injective, and this shows claim \ref{lemitem:support1}. The argument for \ref{lemitem:support2} is similar but uses an isogeny $A^n \to B \times A'$ and the map
\[
(\cD_{w,\fp})^{\oplus 2} \simeq T_\fp(B) \inj T_\fp(A^n)  = T_\fp(A)^{\oplus n} . \qedhere
\] 
\end{proof}

We continue by introducing some notation.
If $\lambda$ (resp.\ $\fp$) is a maximal ideal of $R_w \otimes \bZ_\ell$ for $\ell \not= p$ (resp.\ of $R_w \otimes \bZ_p$), denote by
\[
\Mod_\Zltf(R_{w,\lambda}) \quad \big(\text{resp.\ }  \Mod_\Zptf(\cD_{w,\fp}) \, \big)
\]
the category of finitely generated $R_{w,\lambda}$-modules that are free over $\bZ_\ell$ (resp.\  finitely generated $\cD_{w,\fp}$-modules that are free over $\bZ_p$). Similarly, let
\[
\Mod_\Zltf(S_{w,\lambda}) \quad \big(\text{resp.\ } \Mod_\Zptf(S_{w,\fp}) \, \big)
\]
be the category of finitely generated left modules over $S_{w,\lambda}$ that are free over $\bZ_\ell$ (resp.\  finitely generated $S_{w,\fp}$-modules that are free over $\bZ_p$). 
If $N$ is any object of $\Mod_\Zltf(R_{w,\lambda})$, the formula
\begin{equation}\label{dualmodules}
N^\star := \Hom_{R_{w,\lambda}}(N, T_\lambda(A_w))
\end{equation}
defines a contravariant functor on $\Mod_\Zltf(R_{w,\lambda})$ with values in 
$\Mod_\Zltf(S_{w,\lambda})$,
thanks to the identification $S_{w,\lambda} =\End_{R_{w,\lambda}}(T_\lambda(A_w))$ that follows from Proposition~\ref{prop:localTateTheorem}~\ref{propitem:localTateTheorem lambda}.

\begin{lem}[Local anti-equivalence at $\lambda$]
\label{lem:upperstarfunctor} 
Let $w\subseteq W_q$ be any finite subset, and let $A_w$ be a $w$-locally projective abelian variety with support $w(A_w) = w$. Let $\lambda$ be a maximal ideal of $R_w \otimes \bZ_\ell$ for $\ell \not= p$.
The functor
\[
{(-)}^\star: \Mod_\Zltf(R_{w,\lambda})  \longrightarrow\Mod_\Zltf(S_{w,\lambda})
\]
is an anti-equivalence of categories.
\end{lem}

\begin{rmk}\label{rmk:reflexivmodules}
In \eqref{eq:RwtensorZell} we recalled that for a prime $\ell \not=p$ the ring $R_w\otimes\bZ_\ell$ is isomorphic to $\bZ_\ell[x]/P_w(x)$. In particular,
$R_w\otimes\bZ_\ell$ is a Gorenstein ring of dimension one, being a complete intersection. The relevant consequence for us is that
any object of $\Mod_\Zltf(R_{w,\lambda})$ is reflexive, that is to say that the dual module functor
\begin{equation}\label{eq:dualmodule}
(-)^\vee :=\Hom_{R_{w,\lambda}}(-, R_{w,\lambda})
\end{equation}
is an anti-equivalence of $\Mod_\Zltf(R_{w,\lambda})$ to itself. For more details see 
\cite[Lemma~13]{CS:part1}.
\end{rmk}

\begin{proof}[Proof of Lemma~\ref{lem:upperstarfunctor}] For any object $N$ of $\Mod_\Zltf(R_{w,\lambda})$ there is a natural isomorphism
\[
\xi_N:N^\star\stackrel{\sim}{\longrightarrow} N^\vee\otimes_{R_{w,\lambda}}T_\lambda(A_w)
\]
which depends functorially on $N$, where $N^\vee$ is the dual of $N$, as defined in \eqref{eq:dualmodule}. This is to say that ${(-)}^\star$ is isomorphic to the composition of functors
\begin{equation}\label{MoritaDual}
\xymatrix@M+1ex{
\Mod_\Zltf(R_{w,\lambda}) \ar[rr]^{(-)^\vee} && \Mod_\Zltf(R_{w,\lambda}) 
\ar[rrr]^{-\otimes_{R_{w,\lambda}} T_\lambda(A_w)} &&& \Mod_\Zltf(S_{w,\lambda}).}
\end{equation}
Since $T_\lambda(A_w)$ is a free $R_{w,\lambda}$-module, and moreover nontrivial by Lemma~\ref{lem:support} due to the assumption on the Weil support, the functor $-\otimes_{R_{w,\lambda}} T_\lambda(A_w)$ is a Morita equivalence, and
thus ${(-)}^\star$ is an anti-equivalence of categories, being a composition of the anti-equivalence \eqref{eq:dualmodule} with an equivalence.
\end{proof}

For what concerns the situation at $p$, for an object $N$ of $\Mod_\Zptf(\cD_{w,\fp})$ the formula
\[
N_\star :=\Hom_{\cD_{w,\fp}}(T_\fp(A_w), N)
\]
defines a covariant functor from $\Mod_\Zptf(\cD_{w,\fp})$ to $\Mod_\Zptf(S_{w,\fp})$. Notice that $S_{w,\fp}$ acts on $N_\star$ from the left thanks to the contravariance of Proposition~\ref{prop:localTateTheorem}~\ref{propitem:localTateTheorem p} and the natural right action of $\End_{\cD_{w,\fp}}(T_\fp(A_w))$ on $N_\star$.

\begin{lem}[Local equivalence at $\fp$]
\label{lem:lowerstarfunctor} 
Let $w\subseteq W_q$ be any finite subset, and let $A_w$ be a $w$-locally projective abelian variety with support $w(A_w) = w$. Let $\fp$ be a maximal ideal of $R_w \otimes \bZ_p$.
The functor
\[
{(-)}_\star: \Mod_\Zptf(\cD_{w,\fp})   \longrightarrow   \Mod_\Zptf(S_{w,\fp}),
\]
is an equivalence of categories.
\end{lem}
\begin{proof}
Since $A_w$ is assumed $w$-locally projective, by Proposition~\ref{prop:locallyfreeandprojective} there is an integer $n \geq 1$ such that $A_w^n$ is $w$-locally free. The functors ${(-)}_\star$ for $A_w$ and the power $A_w^n$ are linked by a Morita equivalence. Therefore we may assume withlout loss of generality that $T_\fp(A_w)$ is a free $\cD_{w,\fp}$-module.

Let $T_p(A_w)$ be free of rank $n$ over $\cD_{w,\fp}$. By Lemma~\ref{lem:support}, due to the assumption on the Weil support, we have $n \geq 1$. A choice of a basis determines an isomorphism 
$T_p(A_w) \simeq \cD_{w,\fp}^{\oplus n}$ and a natural isomorphism
\[
N_\star = \Hom_{\cD_{w,\fp}}(T_\fp(A_w), N) \simeq \Hom_{\cD_{w,\fp}}(\cD_{w,\fp}^{\oplus n}, N)  = N^{\oplus n}
\]
with the $S_{w,\fp}$ action on $N_\star$  corresponding to the 
$\rM_n(\cD_{w,\fp})$-action on $N^{\oplus n}$ under the isomorphism 
\[
S_{w,\fp} = \End_{\cD_{w,\fp}}(T_\fp(A_w))^\op \simeq  \End_{\cD_{w,\fp}}(\cD_{w,\fp}^{\oplus n})^\op = \rM_n(\cD_{w,\fp})
\]
of Proposition~\ref{prop:localTateTheorem}~\ref{propitem:localTateTheorem p} and 
induced by that very same choice of basis.
Hence the functor ${(-)}_\star$ is isomorphic to a Morita equivalence restricted to modules that are free and finitely generated as $\bZ_p$-modules.
\end{proof}

While Lemmata~\ref{lem:upperstarfunctor} and ~\ref{lem:lowerstarfunctor} above enter in the proof of fully faithfulness of $T_w(-)$, the next lemma
is needed to show that $T_w(-)$ is essentially surjective.

\begin{lem} \label{lem:T-on-injective}
Let $A \inj B$ be an injective homomorphism of abelian varieties in $\AV_{\bF_q}$, and let $\ell \not=p$ be a prime number. Then the following holds.
\begin{enumerate}[label=(\arabic*),align=left,labelindent=0pt,leftmargin=*,widest = (8)]
\item 
\label{lemitem:T-on-injective1}
The natural map $T_\ell(A) \to T_\ell(B)$ is injective and cotorsion free, i.e.\ the cokernel is a free $\bZ_\ell$-module.
\item 
\label{lemitem:T-on-injective2}
The natural map $T_p(B) \to T_p(A)$ is surjective.
\end{enumerate}
\end{lem}
\begin{proof}
The quotient $C=B/A$ exists in the category of finite type group schemes over $\Spec(\bF_q)$ and is an abelian variety. For all $n \in \bN$ multiplication by $n$ is an isogeny, hence surjective. It follows from the snake lemma in the abelian category of finite type group schemes over $\Spec(\bF_q)$, see \cite[Exp. VI, \S5.4 Th\'eor\`eme]{sga3},  that we have an exact sequence 
\begin{equation} \label{eq:kernelof[n]}
0 \longrightarrow A[n] \longrightarrow B[n] \longrightarrow C[n] \longrightarrow 0
\end{equation}
of finite flat group schemes. 
By passing to the limit for $n = \ell^m$ one deduces from \eqref{eq:kernelof[n]} the exact sequence
\[
0 \longrightarrow T_\ell(A) \longrightarrow T_\ell(B) \longrightarrow T_\ell(C) \longrightarrow 0.
\]
Now \ref{lemitem:T-on-injective1} follows because the cokernel $T_\ell(C)$ is a free $\bZ_\ell$-module. 

\ref{lemitem:T-on-injective2} Since $T_p(A)$ is the limit of the Dieudonn\'e modules for the system of $A[p^m]$ and the Dieudonn\'e module functor is exact on finite flat group schemes we also have an exact sequence
\[
0 \longrightarrow T_p(C) \longrightarrow T_p(B) \longrightarrow T_p(A) \longrightarrow 0. \qedhere
\]
\end{proof}

\subsection{Existence and characterisation of injective cogenerators}\label{subsec:injectivecogen}

We are now ready to prove the main result of the section.

\begin{thm} 
\label{thm:truncatedfullyfaithful}
Let $w\subseteq W_q$ be a finite subset, and let $A_w$ in $\AV_w$ be an abelian variety. Then the following are equivalent.
\begin{enumerate}[label=(\alph*),align=left,labelindent=0pt,leftmargin=*,widest = (m)]
\item 
\label{thmitem:wlocallyprojective}
$A_w$ is $w$-locally projective with support $w(A_w) = w$. 
\item
\label{thmitem:injectivecogen}
$A_w$ is an injective cogenerator for $\AV_w$.
\item
\label{thmitem:equivalence}
The functor\footnote{To avoid double indices, we use the notation $T_w$ for the functor that was previously denote by $T_{A_w}$.}
\[
T_w \colon  \AV_w  \longrightarrow \Mod_\Ztf(S_w), \qquad T_w(X) = \Hom_{\bF_q}(X,A_w)
\]
is an anti-equivalence of categories. 
\end{enumerate}
\end{thm}

\begin{proof} 
We start by showing \ref{thmitem:wlocallyprojective} $\Longrightarrow$ \ref{thmitem:equivalence}. So  $A_w$ is $w$-locally projective with support $w$. The first thing to show is that $T_w(-)$ is fully faithful. To this purpose let $X, Y$ be abelian varieties in $\AV_w$, and consider the map
\[
\tau : \Hom_{\bF_q}(X,Y) \longrightarrow \Hom_{S_w}(T_w(Y), T_w(X))
\]
induced by $T_w(-)$. Both source and target of $\tau$ are finitely generated $R_w$-modules. Hence it suffices to verify that $\tau \otimes R_{w,\lambda}$  and also $\tau \otimes R_{w,\fp}$ are bijective for all maximal ideals $\lambda$ of $R_w \otimes \bZ_\ell$, for $\ell \not= p$,   and  all maximal ideals $\fp$ of $R_w \otimes \bZ_p$ respectively.
 
\smallskip
 
\textit{The case $\lambda$:}  consider the commutative diagram
 \[
\xymatrix@M+1ex@R-2ex
{\Hom_{\bF_q}(X,Y) \otimes_{R_w} R_{w,\lambda} \ar[r]^(0.4){\tau \otimes R_{w,\lambda}}  \ar[dd]^{\rm Tate} & \Hom_{S_w}\big(T_w(Y), T_w(X)\big) \otimes_{R_w} R_{w,\lambda}  \ar[d] \\
& \Hom_{S_{w,\lambda}}\big(T_w(Y) \otimes_{R_w} R_{w,\lambda}, T_w(X) \otimes_{R_w} R_{w,\lambda} \big) \ar[d]^{\rm Hom(Tate,Tate)} \\
\Hom_{R_{w,\lambda}} \big(T_\lambda(X),T_\lambda(Y) \big) \ar[r] &
\Hom_{S_{w,\lambda}}\big({T_\lambda(Y)}^\star, {T_\lambda(X)}^\star \big)}
\]
whose arrows all are the natural ones. Thanks to the $\lambda$-adic version of Tate's isomorphism as in Proposition~\ref{prop:localTateTheorem}, and to flat localization along
$R_w \to R_{w,\lambda}$, all vertical arrows of the diagram are isomorphisms. Thanks to Lemma~\ref{lem:upperstarfunctor}, the bottom
horizontal arrow is an isomorphism, hence the same is true for $\tau \otimes R_{w,\lambda}$.

\smallskip
 
\textit{The case $\fp$:} Similarly, the vertical and the bottom arrows of the commutative diagram
\[
\xymatrix@M+1ex@R-2ex
{\Hom_{\bF_q}(X,Y) \otimes_{R_w} R_{w,\fp} \ar[r]^(0.4){\tau \otimes R_{w,\fp}}  \ar[dd]^{\rm Tate} & \Hom_{S_w}\big(T_w(Y), T_w(X)\big) \otimes_{R_w} R_{w,\fp}  \ar[d] \\
& \Hom_{S_{w,\fp}}\big(T_w(Y)\otimes_{R_w} R_{w,\fp},T_w(X)\otimes_{R_w} R_{w,\fp} \big) \ar[d]^{\rm Hom(Tate,Tate)} \\
\Hom_{\cD_{w,\fp}}\big(T_\fp(Y),T_\fp(X) \big) \ar[r] &
\Hom_{S_{w,\fp}}\big({T_\fp(Y)}_\star, {T_\fp(X)}_\star \big)}
\]
are isomorphisms, thanks to the $\fp$-adic version of Tate's isomorphism as in Proposition~\ref{prop:localTateTheorem},  flat localization along
$R_w \to R_{w,\fp}$, and Lemma~\ref{lem:lowerstarfunctor}. This implies that  $\tau \otimes R_{w,\fp}$ is an isomorphism, completing the proof that $T_w(-)$ is fully faithful.

\smallskip

We now show that $T_w(-)$ is essentially surjective. Since we already know that $T_w(-)$ is fully faithful, it is enough
to check that $T_w(-)$ satisfies condition \ref{thmitem:formal c2} of Theorem~\ref{thm:formal}, i.e.\ we need to show that for all injections $i: X \inj Y$ in $\AV_w$ the induced map 
\begin{equation}\label{eq:M=IGM}
i^\ast: T_w(Y) \longrightarrow T_w(X) 
\end{equation}
is surjective. We accomplish this by showing that $i^\ast\otimes\bZ_\ell$ is surjective for all primes $\ell$. 

\smallskip
 
\textit{The case $\ell \not= p$:} by Tate's isomorphism \eqref{TateThmell} the scalar extension $i^\ast\otimes\bZ_\ell$ is identified with the induced map 
\[
\Hom_{R_w \otimes \bZ_\ell} (T_\ell(Y), T_\ell(A_w)) \longrightarrow \Hom_{R_w \otimes \bZ_\ell} (T_\ell(X), T_\ell(A_w)), \qquad \ph \mapsto \ph \circ T_\ell(i) 
\]
on $\ell$-adic Tate modules. Define the $R_w\otimes\bZ_\ell$-module $M$ by exactness of the sequence
\[
0 \longrightarrow T_\ell(X) \longrightarrow T_\ell(Y) \longrightarrow M \longrightarrow 0.
\]
Since $i$ is injective, Lemma~\ref{lem:T-on-injective}~\ref{lemitem:T-on-injective1} shows that $M$ is free as $\bZ_\ell$-module and hence reflexive as $R_w \otimes \bZ_\ell$-module, see Remark~\ref{rmk:reflexivmodules}. The $\Ext$-sequence
\[
\Hom_{R_w \otimes \bZ_\ell} (T_\ell(Y), T_\ell(A_w)) \longrightarrow \Hom_{R_w \otimes \bZ_\ell} (T_\ell(X), T_\ell(A_w)) \longrightarrow \Ext^1_{R_w \otimes \bZ_\ell}(M,T_\ell(A_w))
\]
shows that surjectivity of $i^\ast\otimes\bZ_\ell$ follows from the vanishing of $\Ext^1_{R_w \otimes \bZ_\ell}(M,T_\ell(A_w))$.
Since $R_w \otimes \bZ_\ell$ is the product of the local rings $R_{w,\lambda}$, we have
\[
\Ext^1_{R_w \otimes \bZ_\ell}(M,T_\ell(A_w)) = \bigoplus_{\lambda} \Ext^1_{R_{w,\lambda}}(M \otimes R_{w,\lambda},T_\lambda(A_w)).
\]
Since $A_w$ is locally projective, Proposition~\ref{prop:locallyfreeandprojective} shows that $T_\lambda(A_w)$ is a free $R_{w,\lambda}$-module. Now the vanishing 
\[
\Ext^1_{R_{w,\lambda}}(M \otimes R_{w,\lambda},T_\lambda(A_w)) = 0
\]
is a consequence of  \cite[Lemma 17]{CS:part1}  and the fact that $R_{w,\lambda}$ is a Gorenstein ring of dimension $1$ as the localization of $R_w \otimes \bZ_\ell$.

\smallskip
 
\textit{The case $\ell = p$:} using Tate's isomorphism \eqref{TateThmp}, the surjectivity of $i^\ast\otimes\bZ_p$ translates into the surjectivity of 
\[
\Hom_{\cD_w} (T_p(A_w),T_p(Y)) \longrightarrow \Hom_{\cD_w} (T_p(A_w),T_p(X)),
\]
which is an immediate consequence of Lemma~\ref{lem:T-on-injective}~\ref{lemitem:T-on-injective2} and the fact that $T_p(A_w)$ is a projective $\cD_w$-module.

\smallskip

The equivalence of \ref{thmitem:equivalence} and \ref{thmitem:injectivecogen} is proven as part of 
Theorem~\ref{thm:formal}. 
It remains to show that \ref{thmitem:injectivecogen} implies  \ref{thmitem:wlocallyprojective}. For that we pick an auxiliary abelian variety $A'_w$ that is $w$-balanced. 
Such an $A'_w$ exists by Theorem~\ref{thm:integralstructure}, and $A'_w$ is $w$-locally projective with support $w$. By what we have already shown, $A'_w$ is an injective cogenerator.  Therefore there exist an $n$ and an embedding $i: A_w \inj (A'_w)^n$ according to Theorem~\ref{thm:formal} \ref{thmitem:formal c1}. Now we apply Theorem~\ref{thm:formal}~\ref{thmitem:formal c2} to this embedding, but with respect to the injective cogenerator $A_w$. It follows that the map 
\[
\Hom_{\bF_q}((A'_w)^n, A_w) \surj \Hom_{\bF_q}(A_w, A_w), \qquad \ph \mapsto \ph \circ i
\]
is surjective. A preimage of the identity is a retraction $(A'_w)^n \to A_w$ which shows that $A_w$ is a direct factor of $(A'_w)^n$. In particular, $T_\lambda(A_w)$ (resp.\ $T_\fp(A_w)$) is a direct factor of a 
projective $R_{w,\lambda}$ (resp.\ a projective $\cD_{w,\fp}$)-module and hence is itself projective. This shows that $A_w$ is $w$-locally projective. 

For all $\pi \in w$, applying  Theorem~\ref{thm:formal} \ref{thmitem:formal c1} to a simple abelian variety $B_\pi$ in the isogeny class associated to $\pi$ yields an embedding $B_\pi \inj A_w^n$ for some $n$. In particular, $\pi$ is contained in the support of $A_w$. This completes the proof. 
\end{proof}

\begin{rmk}
In an appendix to \cite{lauter:serre} Serre proves a version of Theorem~\ref{thm:truncatedfullyfaithful} for $w = \{\pi\}$ and an ordinary Weil number $\pi$ associated to an ordinary elliptic curve $E$ such that $R_\pi$ is a maximal order in $\bQ(\pi)$. In contrast to Deligne's proof in \cite{De}, Serre's proof uses the covariant functor 
$\Hom_{\bF_q}(E,-)$, 
so no multiplicities are used. 

We reprove in Theorem~\ref{thm:reprovedeligne} Deligne's result for the category of all ordinary abelian varieties. Also in this case the representing object $A_w^{\ord}$ is isogenous to $\prod_{\pi \in w} B_\pi$, so again no multi\-pli\-cities are necessary. However, the general case requires multiplicities. In Section~\ref{sec:multiplicity} we dicuss completely the necessity of multiplicities in the case of $\AV_\pi$ for a Weil number $\pi$ associated to an elliptic curve. The crucial result concerning multiplicities is 
Theorem~\ref{thm:connectedSpec classification of locallyprojective}.
\end{rmk}

\begin{cor}
Let $A$ in $\AV_w$ be an injective cogenerator. Then $A$ is a direct factor of a power of a $w$-balanced abelian variety.
\end{cor}
\begin{proof}
This was proved at the end of the proof of Theorem~\ref{thm:truncatedfullyfaithful}.
\end{proof}

\begin{prop}
\label{prop:rankOfTw}
Let $w\subseteq W_q$ be a finite subset, and let $A_w$ in $\AV_w$ be a $w$-balanced abelian variety.
Then, for any $X$ in $\AV_w$, the $\bZ$-rank of $T_w(X) = \Hom_{\bF_q}(X,A_w)$ equals
\[
\rank_{\bZ} T_w(X) = 4 
\expo  \dim(X).
\]
\end{prop}
\begin{proof}
Since both sides of the equation are isogeny invariant and additive with respect to products of abelian varieties, it suffices to consider $\bF_q$-simple objects $X=B_\pi$ for  $\pi \in w$. Then 
\begin{align*}
\rank_\bZ(T_w(B_\pi)) &
= m_\pi\rank_\bZ\left(\End_{\bF_q}(B_\pi)\right) 
= \dfrac{2\expo}{s_\pi}\cdot s_\pi^2[\bQ(\pi):\bQ] 
= 4 \expo  \dim(B_\pi),
\end{align*}
where the last equality uses Theorem~\ref{thm:TateEndostructure}~\ref{thmitem:dimfoemula}  from Honda-Tate theory.
\end{proof}

\subsection{Realizing \texorpdfstring{$R_w$}{the minimal central order} as the center}
\label{subsec:center}

If $w=\{\pi\}$ consists of a single {\it ordinary} Weil $q$--number $\pi$, then Waterhouse shows in 
\cite[Chapter 7]{Wa} that the minimal endomorphism ring $R_\pi$ arises as the endomorphism ring of a simple, ordinary abelian variety $B_\pi$. We will reprove this later as part of Proposition~\ref{prop:MoritaReduction}. In this section we will show that $R_w$ agrees with the center of $\End_{\bF_q}(A)$ for any injective cogenerator $A$ in $\AV_w$. This verifies the claim made in the introduction that the center $\dR_q$ of $\dS_q$ can be described explicitly in terms of Weil $q$-numbers.

\begin{prop}
\label{prop:centerofinjectivecogenerators}
Let $A \in \AV_w$ be an injective cogenerator. Then the natural map 
\[
R_w \longrightarrow \End_{\bF_q}(A)
\]
identifies $R_w$ with the center of $\End_{\bF_q}(A)$. 
\end{prop}
\begin{proof}
We first prove the result for a $w$-balanced abelian variety $A_w$. In order to show that the map $R_w \to Z(\End_{\bF_q}(A_w))$ is an isomorphism, where $Z(-)$ denotes the center, it suffices to show this after completion at all prime numbers.  We deduced from Tate's theorems \eqref{TateThmell} and \eqref{TateThmp}
that for all primes $\ell\neq p$
\begin{align*}
\End_{\bF_q}(A_w) \otimes\bZ_\ell & = \End_{R_w \otimes \bZ_\ell}(T_\ell(A_w)) \simeq  \End_{R_w \otimes \bZ_\ell}\big((R_w \otimes \bZ_\ell)^{\oplus 2\expo}\big) = \rM_{2\expo}(R_w\otimes\bZ_\ell), \\
\End_{\bF_q}(A_w)  \otimes\bZ_p  & = {\End_{\cD_w}(T_p(A_w))}^\op \simeq  {\End_{\cD_w}(\cD_w^{\oplus 2})}^\op = \rM_2(\cD_w).
\end{align*}
In both cases we determine the center as indeed $R_w \otimes \bZ_\ell$ (resp.\ $R_w \otimes \bZ_p$, see Lemma~\ref{lem:centerDw}).

Now let $A$ be an arbitrary injective cogenerator. Let $S  = \End_{\bF_q}(A)$ and denote the anti-equivalence $\Hom_{\bF_q}(-,A)$ by $T(-)$.  Then the composition of the natural injective maps
\[
R_w \inj Z(S) \inj Z(\End_S(T(A_w))) = Z(\End_{\bF_q}(A_w)^\op)
\]
is an isomorphism due to the discussion of the $w$-balanced case. 
The result follows. 
\end{proof}

\begin{rmk}
The converse to Proposition~\ref{prop:centerofinjectivecogenerators} does not hold for $q = p^4$ and $\pi = p^2$. The simple object in $\AV_\pi$ is an elliptic curve $B_\pi$ with $\End(B_\pi)$ a maximal order in the  quaternion algebra over $\bQ$ ramified in $p\infty$. Then $R_\pi = \bZ$ equals the center of $\End(B_\pi)$. But $\Hom(-,B_\pi)$ is not an anti-equivalence according to \cite[Theorem 1.1]{JKPRSBT}. 
Since $\pi$ is non-ordinary, Theorem~\ref{thm:connectedSpec classification of locallyprojective} says that all injective cogenerators of $\AV_\pi$ are isogenous to a power of a reduced $\pi$-balanced object. Thus injective cogenerators of $\AV_\pi$ have even multiplicity because of $\widebar{m}_\pi = \expo/s_\pi = 2$.
\end{rmk}

\section{Compatible truncated anti-equivalences}
\label{sec:compatibleantiequivalences}

\subsection{Maximal subgroup with partial Weil support}
\label{sec:subgroup}
Recall that for an abelian variety  $A$ over $\bF_q$ we set
\[
S(A) = \End_{\bF_q}(A),
\]
and that $S(A)$ is an  $R_w$-algebra if $A \in \AV_w$.

\begin{prop}
\label{prop:formal_compatiblefunctors}
Let $v \subseteq w \subseteq W_q$ be finite sets of Weil numbers, and let $A_w \in \AV_w$ be an injective cogenerator.
Then the subgroup generated by all images
\[
A_{v,w} : = \left\langle \im(f) \ ; \ f: X \to A_{w}, \ X \in \AV_{v}\right\rangle \subseteq A_{w}
\] 
satisfies the following:
\begin{enumerate}[label=(\arabic*),align=left,labelindent=0pt,leftmargin=*,widest = (8)]
\item
\label{propitem:inAVv} $A_{v,w}$ belongs to $\AV_{v}$ and is an abelian subvariety of $A_{w}$.
\item 
\label{prop:formal_compatiblefunctors2}
$A_{v,w}$ is an injective cogenerator in $\AV_v$. 
\item 
\label{prop:formal_compatiblefunctors3}
Restriction to $A_{v,w}$ induces a natural surjection 
\[
\pr_{v,w} : S(A_{w}) \surj S(A_{v,w})
\]
that is an  $R_w \surj R_v$ algebra map.
\item 
\label{propitem:commute} 
We set $T_w(X) = \Hom_{\bF_q}(X,A_w)$ and $T_{v,w}(X) = \Hom_{\bF_q}(X,A_{v,w})$. 
The following diagram naturally commutes:
\begin{equation} \label{eq:changeofWeilstring2}
\xymatrix@M+1ex{
\AV_{w} \ar[rr]^(0.4){T_{w}} & & \Mod_\Ztf(S(X_w))\\
\AV_{v} \ar@{}[u]|{\displaystyle \bigcup} \ar[rr]^(0.4){T_{v,w}}  && \Mod_\Ztf(S(X_{v,w})) \ar@{}[u]|{\displaystyle \bigcup}.
}
\end{equation}
Here the right inclusion is defined by pulling back the $S(A_{v,w})$-module structure via $\pr_{v,w}$ to a $S(A_w)$-module structure.
\item
\label{propitem:balanced is preserved}
If $A_w$ is $w$-balanced, then $A_{v,w}$ is $v$-balanced.
\end{enumerate}
\end{prop}

\begin{proof}
Assertion \ref{propitem:inAVv}  is trivial. 

\ref{prop:formal_compatiblefunctors2}  Being an injective cogenerator means that property property \ref{thmitem:formal c} of Theorem~\ref{thm:formal} holds, so this holds for $A_w$. 
We are going to show that property \ref{thmitem:formal c}  also holds for $A_{v,w}$.
Now \ref{thmitem:formal c1} holds because for an arbitrary $X \in \AV_v$ there is an injection 
\[
X \inj A_w^n
\]
for a suitable $n$ that factors through $A_{v,w}^n \subseteq X_w^n$ by the definition of $A_{v,w}$. 

To show \ref{thmitem:formal c2}  we start with an injection $X \inj Y$ in $\AV_w$ and a homomorphism $f_0: X \to A_{v,w}$. By composing with $i : A_{v,w} \inj A_w$ we can extend $f = i \circ f_0$ to a $g : Y \to A_w$ since $A_w$ is an  injective object, i.e.\ $A_w$ satisfies property \ref{thmitem:formal c2}. By definition $g$ factors through $g_0 : Y \to A_{v,w}$ and $g_0$ extends $f_0$ because $i$ is injective.

\ref{prop:formal_compatiblefunctors3} Restriction with $i: A_{v,w} \inj A_w$ yields a surjection by Theorem~\ref{thm:formal} \ref{thmitem:formal c2} 
\[
\pr_{v,w}: S(A_w) = \Hom_{\bF_q}(A_w,A_w) \surj \Hom_{\bF_q}(A_{v,w},A_w) = \Hom_{\bF_q}(A_{v,w},A_{v,w}) = S(A_{v,w})
\]
with the next to last equality due to the definition of $A_{v,w}$. Here $\pr_{v,w}(g)$ for a homomorphism $g: A_w \to A_w$ is the unique homomorphism so that the following commutes:
\[
\xymatrix@M+1ex{
A_w \ar[r]^g & A_w  \\
A_{v,w}  \ar@{^(->}[u]^i \ar[r]^{\pr_{v,w}(g)} & A_{v,w}  \ar@{^(->}[u]^i. 
}
\]
The map $\pr_{v,w}$ is indeed a ring homomorphism as for a composition $f \circ g$ of $f,g \in S(A_w)$ the composition $\pr_{v,w}(f) \circ \pr_{v,w}(g)$ completes the square as in the following diagram:
\[
\xymatrix@M+1ex{
A_w \ar[r]^g & A_w \ar[r]^f & A_w \\
A_{v,w}  \ar@{^(->}[u]^i \ar[r]^{\pr_{v,w}(g)} & A_{v,w} \ar[r]^{\pr_{v,w}(f)}  \ar@{^(->}[u]^i  & A_{v,w}  \ar@{^(->}[u]^i . 
}
\]
Moreover, the map $\pr_{v,w}$ is an $R_w \surj R_v$ algebra map, because Frobenius and Verschiebung are natural with resepct to $i: A_{v,w} \inj A_w$. 

Assertion~\ref{propitem:commute} follows from the natural equality for every $X \in \AV_{v}$
\[
\Hom_{\bF_q}(X,A_{v,w}) = \Hom_{\bF_q}(X,A_{w}),
\]
since every morphism $f:X \to A_{w}$ takes values in the subvariety $A_{v,w}\subseteq A_{w}$.

Now we show \ref{propitem:balanced is preserved}. If $A_w$ is $w$-balanced,  then all simple abelian varieties $B_\pi$ for $\pi \in w$ occur in $A_w$ with multiplicity $m_\pi$. Since up to isogeny $A_{v,w}$ consists of all isogeny factors $B_\pi$ of $A_w$ with $\pi \in v$, we deduce that $A_{v,w}$ has the same multiplicities $m_\pi$ for all $\pi \in v$. Thus $A_{v,w}$ is isogenous to a $v$-balanced abelian variety $A_v$. Being an injective cogenerator, $A_{v,w}$  is $v$-locally projective by Theorem~\ref{thm:truncatedfullyfaithful}. Since $A_v$ is also $v$-locally projective, now 
Proposition~\ref{prop:2outof3locallyisom} shows that $A_{v,w}$ and $A_v$ are Tate-locally isomorphic. Again Proposition~\ref{prop:2outof3locallyisom} then shows that $A_{v,w}$ is $v$-balanced, since $A_v$ is $v$-balanced.
\end{proof}

\subsection{The direct system}
\label{sec:proof}
In order to prove Theorem~\ref{MainThm} we construct a direct system 
\[
\mathcal{A}= {\varinjlim_w} A_{w}   
\]
consisting of abelian varieties $A_{w}$ indexed by finite 
subsets $w \subseteq W_q$ of  Weil $q$--numbers such that for all $w$ the functors
\[
T_w \colon \AV_w \longrightarrow  \Mod_\Ztf(S_w), \qquad  T_w(X) = \Hom_{\bF_q}(X,A_w) 
\]
are anti-equivalences, that moreover are naturally compatible among each other. 

\begin{proof}[Proof of Theorem~\ref{MainThm}]
For any finite subset $w\subseteq W_q$  let $Z(w)$ be the set of isomorphism classes $[A]$ of $w$-balanced abelian varieties
The set $Z(w)$ is not empty by Theorem~\ref{thm:integralstructure}.
The elements of $Z(w)$ all belong to the same isogeny class, and so $Z(w)$ is finite, since there are only finitely many isomorphism classes of abelian varieties over a finite field lying in a given isogeny class 
(in fact, Zarhin shows in \cite[Theorem~4.1]{zarhin:endos} that finiteness holds for isomorphism classes of abelian varieties of fixed dimension).  

For any pair $v\subseteq w$ of finite 
subsets of Weil $q$--numbers, we construct a map
\[
\zeta_{v,w}:Z(w)  \longrightarrow Z(v)
\]
by $\zeta_{v,w}([A]) = [B]$ where $B$ is the abelian subvariety of $A$ generated by the image of all $f: C \to A$ with $w(C) \subseteq v$.  
Proposition~\ref{prop:formal_compatiblefunctors}~\ref{propitem:balanced is preserved} states that $\zeta_{v,w}$ indeed takes values in $Z(v)$. 
These maps satisfy the compatibility condition 
\[
\zeta_{u,w}=\zeta_{u,v} \zeta_{v,w},
\]
for all tuples $u \subseteq v\subseteq w$, hence they define a projective system
\[
(Z(w), \zeta_{v,w})
\]
indexed by finite 
subsets $w \subseteq W_q$. Since the sets $Z(w)$ are finite and non-empty, a standard compactness argument shows that the inverse limit is not empty:
\[
Z = \varprojlim_w Z(w) \not= \emptyset.
\]
We choose a compatible system $z = (z_w) \in Z$ of isomorphism classes of abelian varieties. 

\smallskip

Now we would like to choose abelian varieties $A_w$ in each class $z_w$ and inclusions 
\[
\ph_{w,v}: A_v \longrightarrow A_w
\]
for every $v \subseteq w$ that are isomorphic to the inclusions $A_{v,,w} \subseteq A_w$ discussed in Proposition~\ref{prop:formal_compatiblefunctors} in a compatible way: for $u \subseteq v \subseteq w$ we want 
\[
\ph_{w,u}=\ph_{w,v} \ph_{v,u}.
\]

Because the set of Weil numbers is countable, we may choose a cofinal totally ordered subsystem of finite subsets of $W_q$ 
\[
w_1 \subseteq w_2 \subseteq \ldots \subseteq w_i \subseteq \ldots ,
\]
where cofinal implies $\bigcup_i w_i = W_q$. 
Working with this totally ordered subsystem, we can construct  a direct system 
\[
\mathcal{A}_0=(A_{w_i},\ph_{w_j,w_i})
\]
of abelian varieties as desired by induction. If $A_{w_i}$ is already constructed, then we choose $A_{w_{i+1}}$ in $z_{w_{i+1}}$ and deduce from $\zeta_{w_{i},w_{i+1}}(z_{w_{i+1}}) = z_{w_i}$ that there is an injective homomorphism $\ph_{w_{i+1},w_i}: A_{w_i} \to A_{w_{i+1}}$ as desired. By construction, the map 
$\ph_{w_{i+1},w_i}$ is an isomorphism of $A_{w_i}$ with the abelian subvariety $A_{w_i,w_{i+1}}$ of $A_{w_{i+1}}$ as discussed in Proposition~\ref{prop:formal_compatiblefunctors}.

Once $\dA_0$ is constructed, we may identify all transfer maps of the restricted system $\dA_0$ with inclusions. 
Let now $w$ be a general finite subset of $W_q$. Since $\bigcup_i w_i = W_q$ is a union of a totally ordered system of $w_i$, we find $w \subseteq w_i$ for all large enough $i$. 
Then we can define the abelian variety $A_w$
by means of the construction of Proposition~\ref{prop:formal_compatiblefunctors} as the abelian subvariety 
\[
A_w := A_{w,w_i}  \subseteq A_{w_i}.
\]
This choice is well defined, i.e.\ independent of $i \gg 0$.
Furthermore, there are compatible transfer maps $\ph_{v,w} : A_v \to A_w$ for all $v \subseteq w$ that lead to the desired direct system 
\[
\dA = (A_w, \ph_{w,v}).
\]
In the sense of ind-objects we have $\dA_0 \simeq \dA$ and so $\dA_0$ would suffice for Theorem~\ref{MainThm}. 
But we prefer the more aesthetic ind-system $\dA$ indexed by  
all finite subsets $w \subseteq W_q$.

\smallskip

Let $X$ be any element of $\AV_{\bF_q}$, and set
\[
T(X)=\Hom_{\bF_q}(X, \mathcal{A})  = \varinjlim_w \Hom_{\bF_q}(X, A_{w})  = \varinjlim_w T_w(X).
\]
The
groups $\Hom_{\bF_q}(X, A_{w})$ are stable when $w$ is large enough. More precisely, if $w$, $w'$
are finite 
subsets of Weil $q$--numbers with $w(X)\subseteq w\subseteq w'$, then the map 
\[
\varphi_{w',w} \circ -  : \Hom_{\bF_q}(X, A_{w})\longrightarrow\Hom_{\bF_q}(X, A_{w'})
\]
is an isomorphism, cf.\ Proposition~\ref{prop:formal_compatiblefunctors}~\ref{propitem:commute}. 
Moreover, $T(-)$ restricted to
$\AV_{w}$ recovers the functor $T_{w}(-)$ of Theorem~\ref{thm:truncatedfullyfaithful} constructed using the object $A_{w}$ of $\mathcal{A}$. Furthermore, the functor $T_w(-)$ induces an anti-equivalence between $\AV_{w}$ and $\Mod_\Ztf(S_w)$ by Theorem~\ref{thm:truncatedfullyfaithful} because $A_w$ is $w$-balanced.

We next show that $T(-)$ is an $\dR_q$-linear functor. We denote the linear Frobenius isogeny for an abelian variety $X$ over $\bF_q$ by $\pi_X: X \to X$. Observe that, since the Frobenius isogeny is a natural transformation, 
for any finite subset  $w \subseteq W_q$ and for any $f\in\Hom_{\bF_q}(X, A_{w})$ the diagram
\[
\xymatrix@M+1ex{
X \ar[r]^{\pi_X}\ar[d]_f  & X \ar[d]^f\\
A_{w} \ar[r]^{\pi_{A_{w}}}& A_{w}\\
}
\]
is commutative. This implies that, for $w$ sufficiently large so that 
\[
T(X) = \Hom_{\bF_q}(X,A_w), \qquad \text{ $w$ large},
\]
the action $F(f) = \pi_{A_w} \circ f$ on $f \in T(X)$ of the image of $F$ in $R_{w} \subseteq S_w$ 
is given by the morphism induced by the Frobenius isogeny $\pi_X$ via functoriality of $T$
\[
T(\pi_X) : f \mapsto f \circ \pi_X = \pi_{A_w} \circ f .
\]
A similar consideration with the isogeny Verschiebung shows that indeed $T(-)$ is $\dR_q$-linear.

Compatibility in $w$ shows that $T(-)$ induces an anti-equivalence
\[
T = \varinjlim T_w \colon \AV_{q} = \varinjlim_w \AV_w  \xrightarrow{\sim} \varinjlim_w \Mod_\Ztf(S_w) = \Mod_\Ztf( \dS_q).
\]
Complemented by the rank computation of  Proposition~\ref{prop:rankOfTw},
this is precisely the claim of Theorem~\ref{MainThm} and so its proof is complete.
\end{proof}

\begin{rmk}
\label{rmk:equiv_reducedbalanced}
When reduced $w$-balanced abelian varieties exist, i.e.\ if we avoid $\pi = \pm \sqrt{q}$ when $r$ is odd, then the above proof of Theorem~\ref{MainThm} can be applied mutatis mutandis to construct an anti-equivalence of categories represented by an ind-abelian variety $\widebar{\dA} = (\widebar{A}_w,\ph_{w,\otherw})$ consisting of reduced $w$-balanced abelian varieties. This anti-equivalence yields modules $\Hom_{\bF_q}(X,A_w)$ of $\bZ$-rank $2\expo \dim(X)$. 
For $r=1$, this reproves the main result of \cite{CS:part1}. 
\end{rmk}

\subsection{Weil restriction of scalars and base change}
Another form of compatibility can be established with respect to Weil restriction of scalars and with respect to base change. We will not need these compatibilities here and therefore only report the results, all of which have quite formal proofs. 

\begin{prop} \label{prop:base_change_and_weil_restriction}
Let $w$ be a finite set of Weil $q^m$-numbers, and set 
$\sqrt[m]{w} := \{\pi \ ; \ \pi^m \in w\}$.
\begin{enumerate}[label=(\arabic*),align=left,labelindent=0pt,leftmargin=*,widest = (8)]
\item 
Let $A \in \AV_w$ be an injective cogenerator. Then the Weil restriction  $RA := R_{\bF_{q^m}|\bF_q}(A)$ is an injective cogenerator for $\AV_{\sqrt[m]{w}}$, and 
for all $X \in \AV_{\sqrt[m]{w}}$ we have a natural isomorphism of functors in $X$
\[
T_A(X \otimes_{\bF_q} \bF_{q^m}) \simeq  T_{RA}(X).
\]
\item 
Let $A \in \AV_{\sqrt[m]{w}}$ be an injective cogenerator.
The base change $A_m := A \otimes_{\bF_q} \bF_{q^m}$ is an injective cogenerator for $\AV_{w}$, and for all $X \in \AV_{w}$ we have a natural isomorphism of functors in $X$
\[
T_A(R_{\bF_{q^m}|\bF_q}(X)) = T_{A_m}(X).
\]
\item
Let $A_{\sqrt[m]{w}} \in \AV_{\sqrt[m]{w}}$ be $\sqrt[m]{w}$-balanced. Then  the scalar extension
\[
A_w : = A_{\sqrt[m]{w}} \otimes_{\bF_q} \bF_{q^m}
\]
in $\AV_w$ is $w$-balanced.
\end{enumerate}
\end{prop}

One can further show that an ind-representing object $\dA = (A_\omega)_{\omega \subseteq W_q}$ provides a similar ind-representing object $\dA \otimes_{\bF_q}  \bF_{q^m}$ after base changing those with index  $\sqrt[m]{w}$ for subsets $w \subseteq W_{q^m}$. 

\section{The commutative case and optimal rank} 
\label{sec:lowermult}

We discuss when a modified anti-equivalence takes values in modules over a commutative ring. 

\subsection{Recovering Deligne's result: the ordinary case} \label{sec:ordinary}
In this section we will be concerned with lowering multiplicities of injective cogenerators for sets of ordinary Weil $q$-numbers. 

\begin{lem} 
\label{lem:ordinaryDisAzumaya}
For a finite set $w$ of ordinary Weil $q$-numbers, we have $\cD_w \simeq \rM_\expo(R_w \otimes \bZ_p)$. \end{lem}
\begin{proof}
It suffices to determine the structure locally at maximal ideals $\fp$ of $R_w \otimes \fp$, because $R_w\otimes \bZ_p$ is the product of these localizations. As $w$ consists only of ordinary Weil $q$-numbers either $F \notin \fp$ of $V \notin \fp$. The claim  $\cD_{w,\fp} = \rM_\expo(R_{w,\fp})$ now follows from the local structure results \eqref{eq:structureDwfp Finvertible} and \eqref{eq:structureDwfp Vinvertible} obtained in the proof of 
Proposition~\ref{prop:localprojectiveDieudonnestructure}.
\end{proof}

\begin{prop}
\label{prop:balancedEndMatrices}
Let $w$ be a finite set of ordinary Weil $q$-numbers. Then there exist a $w$-balanced abelian variety $A_w$ and an $R_w$-linear isomorphism
\[
\End_{\bF_q}(A_w) \simeq \rM_{2\expo}(R_w).
\]
\end{prop}
\begin{proof} 
Let $A'_w$ be an arbitrary $w$-balanced abelian variety. Then $\End_{\bF_q}(A'_w) \otimes \bQ$ is an Azumaya algebra of degree $2\expo$ over its center $\bQ(w)$. This center is a product of number fields. The local-global principle for the Brauer group, see \cite{BHN}, shows that $\End_{\bF_q}(A'_w) \otimes \bQ$ is split as an Azumaya algebra over $\bQ(w)$, because all local invariants are trivial by Tate's formula as recalled in Theorem~\ref{thm:TateEndostructure}~\ref{thmitem:invariantatp}. Here we use that $w$ only contains ordinary Weil numbers. The spliting translates into a $\bQ(w)$-algebra isomorphism
\[
\End_{\bF_q}(A'_w) \otimes \bQ \simeq \rM_{2\expo}(\bQ(w)).
\]
Moreover, the integral local structure of $S'_w =\End_{\bF_q}(A'_w)$ can be deduced from Tate's theorems \eqref{TateThmell} and \eqref{TateThmp} and Lemma~\ref{lem:ordinaryDisAzumaya} as
\begin{align*}
S'_w\otimes\bZ_\ell & = \End_{R_w \otimes \bZ_\ell}(T_\ell(A'_w)) \simeq  \End_{R_w \otimes \bZ_\ell}\big((R_w \otimes \bZ_\ell)^{\oplus 2\expo}\big) = \rM_{2\expo}(R_w\otimes\bZ_\ell),  \\
S'_w\otimes\bZ_p  & = {\End_{\cD_w}(T_p(A'_w))}^\op \simeq  {\End_{\cD_w}(\cD_w^{\oplus 2})}^\op = \rM_2(\cD_w) =\rM_{2r}(R_w \otimes \bZ_p).
\end{align*}
Proposition~\ref{prop:modifyEnd} applied to the $R_w$-order $\fO = \rM_{2\expo}(R_w)$ shows that there exists an abelian variety $A_w$ that is Tate-locally isomorphic to $A'_w$ and such that $\End_{\bF_q}(A_w) \simeq \rM_{2\expo}(R_w)$. The abelian variety $A_w$ is also $w$-balanced by Proposition~\ref{prop:2outof3locallyisom}.
\end{proof}

\begin{prop}
\label{prop:MoritaReduction}
Let $w$ be a finite set of ordinary Weil $q$-numbers. Then there exists an abelian variety $A^{\ord}_w$ with the following properties.
\begin{enumerate}[label=(\roman*),align=left,labelindent=0pt,leftmargin=*,widest = (iii)]
\item
The Weil support of  $A^{\ord}_w$ is equal to $w$,
\item
the natural inclusion $R_w\subset\End_{\bF_q}(A^{\ord}_w)$ is an isomorphism,
\item
$A^{\ord}_w$ is $w$-locally projective, and
\item
$2\dim(A^{\ord}_w) = [\bQ(w):\bQ]$.
\end{enumerate}

In particular, the corresponding
contravariant functor 
\[
T_{A^{\ord}_w}  \colon \AV_w \longrightarrow \Mod_{\Ztf}(R_w), \qquad T_{A^{\ord}_w}(X) = \Hom_{\bF_q}(X , A^{\ord}_w)
\]
gives an $R_w$--linear anti-equivalence of categories.
\end{prop}

\begin{proof}
We choose a $w$-balanced $A_w$  and fix an isomorphism 
$\End(A_w) \simeq \rM_{2\expo}(R_w)$ as in Proposition~\ref{prop:balancedEndMatrices}. Let $e_i \in  \End(A_w)$ be the idemponent that corresponds in $\End(A_w) \simeq \rM_{2\expo}(R_w)$ 
to the matrix with a single $1$ at the $i$th diagonal entry. These idempotents $e_1, \ldots, e_{2\expo}$ commute and sum to the identity. Let $A_i$ be the image of $e_i : A_w \to A_w$. It follows formally as in any additive category that 
\[
A_w \simeq A_1 \times \ldots \times A_{2\expo}.
\]
Moreover, let $e_{ij} \in  \End(A_w) \simeq \rM_{2\expo}(R_w)$ correspond to the elementary matrix with a single $1$ in the $i$th row and $j$th column. Then $e_{ij}|_{A_j} : A_j \to A_i$ is an isomorphism, so all direct factors of $A_w$ are mutually isomorphic. We define 
\[
A^{\ord}_w = e_1 (A_w) = A_1,
\]
so that $A_w \simeq (A^{\ord}_w)^{2\expo}$. In particular, the support of $A^{\ord}_w$ is the support of $A_w$ and equal to $w$. Moreover, we can compute as $R_w$-algebras
\[
\End(A^{\ord}_w) = e_1 \End(A_w) e_1 = R_w.
\]
Now $A_w^{\ord}$ is $w$-locally projective because a multiple  $A_w = (A_w^{\ord})^{2\expo}$ is $w$-locally projective, and direct summands of projective modules are again projective. As  $A_w^{\ord}$ has full supprt, by Theorem~\ref{thm:truncatedfullyfaithful} we have that $A_w^{\ord}$ indeed represents the desired anti-equivalence of the claim.

In order to compute the dimension of $A_w^{\ord}$, we note that all $\pi \in w$ are ordinary and thus have commutative $E_\pi = \End_{\bF_q}(B_\pi)$. Hence $s_\pi = 1$, and the multiplicity of $B_\pi$ in $A_w$ is $m_\pi = 2\expo$. It follows that $A_w^{\ord}$ is isogenous to $\prod_{\pi \in w} B_\pi$. The dimension formula is a consequence of  Theorem~\ref{thm:TateEndostructure}~\ref{thmitem:dimfoemula} and $\bQ(w) = \prod_{\pi \in w} \bQ(\pi)$.
\end{proof}

We can now reprove Deligne's main result of \cite{De}. We choose to state a contravariant form in accordance with our overall choice. By precomposing with the functor \textit{dual abelian variety} we can pass from contravariant to covariant equivalences.

Recall that $\AV_{\bF_q}^{\ord}$ denotes the full subcategory of $\AV_{\bF_q}$ consisting of ordinary abelian varieties. Similarly, we denote by $\dR_q^{\ord}$ the projective system $(R_w,r_{v,w})$ restricted to finite sets $w$ of ordinary Weil $q$-numbers. 

\begin{thm}[Deligne \cite{De}] 
\label{thm:reprovedeligne}
Let $q=p^\expo$ be a power of a prime number $p$. There exists an ind-abelian variety $\dA^{\ord} = (A^{\ord}_w,\ph_{w,\otherw})$ indexed by finite sets $w \subseteq W^{\ord}_q$ of ordinary Weil $q$-numbers, such that $A^{\ord}_w$ is $w$-locally projective, the transition maps are inclusions, and $\End_{\bF_q}(A_w^{\ord}) = R_w$. The contravariant and  $\dR^{\ord}_q$-linear functor
\[
T^{\ord} \colon \AV_{\bF_q}^{\ord} \longrightarrow \Mod_{\Ztf}(\dR_q^{\ord}), \qquad T^{\ord}(X) = \Hom_{\bF_q}(X,\dA^{\ord})
\]
is an anti-equivalence of categories which preserves the support. Moreover, the $\bZ$-rank of $T^{\ord}(X)$ 
is equal to $2\dim(X)$. 
\end{thm}
\begin{proof}
The construction of Proposition~\ref{prop:formal_compatiblefunctors} applied to an $A^{\ord}_w$ produced by 
Proposition~\ref{prop:MoritaReduction} yields an injective cogenerator $A^{\ord}_{v,w}$ such that $\End_{\bF_q}(A^{\ord}_{v,w})$ is an $R_v$-algebra that is a quotient of $R_w = \End_{\bF_q}(A^{\ord}_{w})$. It follows that also 
$\End_{\bF_q}(A^{\ord}_{v,w}) = R_v$. 

Now the construction of the ind-abelian variety $\dA^{\ord}$ works as in the proof of Theorem~\ref{MainThm} by replacing $w$-balanced abelian varieties by varieties $A_w^{\ord}$ that are $w$-locally projective with $\End_{\bF_q}(A_w^{\ord}) = R_w$. This proves the claim.
\end{proof}

The reduction of multiplicity achieved in Proposition~\ref{prop:MoritaReduction} allows to complete the structure theory of $w$-locally projective abelian varieties begun in Section~\S\ref{sec:truely local}. Recall that there is a tensor product construction between abelian varieties with multiplication by a ring $R$ and certain $R$-modules, see Serre's appendix to \cite{lauter:serre} and \cite[\S4.1]{JKPRSBT}.

\begin{thm}
\label{thm:classification wlocallyprojectiveordinary}
Let $w$ be a finite set of ordinary Weil $q$-numbers, and let $A_w^{\ord}$ be a $w$-locally projective abelian variety with $R_w = \End_{\bF_q}(A_w^{\ord})$ as in Proposition~\ref{prop:MoritaReduction}. 
Then any $w$-locally projective abelian variety $A$ is of the form 
\[
A \simeq A_w^{\ord} \otimes_{R_w} P
\]
where $P = \Hom_{\bF_q}(A_w^{\ord},A)$ is a finitely generated projective $R_w$-module.
\end{thm}
\begin{proof}
Proposition~\ref{prop:VB} \ref{propitem:projectiveHom} shows that $\Hom_{\bF_q}(A_w^{\ord},A)$ is indeed projective. The general properties of the tensor product yield an evaluation map
\[
A_w^{\ord} \otimes_{R_w} \Hom_{\bF_q}(A_w^{\ord},A) \longrightarrow A.
\]
In order to show that this map is an isomorphism, it sufffices to show this locally on $\ell$-adic and $p$-adic Tate modules. There it follows because $A$ and $A_w^{\ord}$ are $w$-locally projective and $T_\ell(A_w^{\ord})$ (resp.\ $T_p(A_w^{\ord})$) is locally free of rank $1$ (resp.\ the unique indecomposable projective module, see Proposition~\ref{prop:localprojectiveDieudonnestructure}).
\end{proof}

\begin{rmk}
Using Theorem~\ref{thm:connectedSpec classification of locallyprojective} and 
Theorem~\ref{thm:classification wlocallyprojectiveordinary} one obtains a complete description of the isogeny classes of $\AV_w$ containing an injective cogenerator for $\AV_w$. See also Remark~\ref{rmk:survey of injective cogenerators}.
\end{rmk}

\subsection{Injective cogenerators with commutative endomorphism ring}

Let $W \subseteq W_q$ be a (possibly infinite\footnote{We choose the notation $W$ over $w$ to distinguish the case of a distinctive finite set $w$ of Weil numbers from the general case $W$.}) set of Weil numbers. The proof of Theorem~\ref{MainThm} adapts immediately to provide an  ind-abelian variety $\dA_W$ that  ind-represents an
anti-equivalence  of $\AV_W$ with $ \Mod_{\Ztf}(\dS_W)$, where $S_W$ equals the pro-ring $\End_{\bF_q}(\dA_W)$. 

We would like to determine all sets $W$ for which there is such an $\dA_W$ as above with $S_W$ commutative.

Recall that we denote by $\expo$ the degree of $\bF_q$ over $\bF_p$.

\begin{prop}
\label{prop:Weilnumbers commutative injective cogenerator}
Let $\pi$ be a Weil $q$-number, and let $A_\pi$ be an injective cogenerator for $\AV_\pi$ that has a commutative ring of endomorphisms $S_\pi$. Then 
\begin{enumerate}[label=(\arabic*),align=left,labelindent=0pt,leftmargin=*,widest = (8)]
\item $\pi$ is ordinary, or 
\item $r=1$ and $\pi$ is not the real conjugacy class of Weil $p$-numbers $\{\pm \sqrt{p}\}$.
\end{enumerate}
\end{prop}
\begin{proof}
In order to have a commutative ring of endomorphisms, $A_\pi$ must be $\bF_q$-simple and $s_\pi = 1$.  The last condition also imples that $\bQ(\pi)$ has no real places and thus that a reduced $\pi$-balanced object exists. 
If $\pi$ is not ordinary, then by Theorem~\ref{thm:connectedSpec classification of locallyprojective} (note that $\Spec(R_\pi)$ is connected), the abelian variety $A_\pi$ must be isogenous to a power of a reduced $\pi$-balanced abelian variety. The variety $A_\pi$ being $\bF_q$-simple, this in particular  implies $\widebar m_\pi = 1$, and thus $\expo =  s_\pi \cdot \widebar m_\pi = 1$.
\end{proof}

We conclude that the examples of commutative cases as presented in Section~\S\ref{subsection:subcategories} of  the introduction cover all commutative cases.

\begin{prop}
\label{prop:commutative cases}
The only sets of Weil $q$-numbers $W$  with a commutative $\dS_W$ and an ind-representable $\dR_W$-linear anti-equivalence $\AV_W \to \Mod_{\Ztf}(\dS_W)$ are contained in 
\begin{enumerate}[label=(\roman*),align=left,labelindent=0pt,leftmargin=*,widest = (iii)]
\item 
the set $W^{\ord}\subseteq W_q$ of ordinary Weil $q$-numbers, or 
\item 
the set $W_p^{\com}=W_p\setminus \{\pm\sqrt p\}$ of non-real Weil $p$-numbers. 
\end{enumerate}
\end{prop}
\begin{proof}
Let $\dA_W$ be an inductive system of abelian varieties from $\AV_W$ that ind-represents an anti-equivalence of categories as in the proposition. By a diagonal process based on the technique of Proposition~\ref{prop:formal_compatiblefunctors}, we may assume that $\dA_W = (A_w)_{w \subseteq W}$ is indexed by finite subsets of Weil $q$-numbers in $W$, and such that $A_w$ represents the restriction of the anti-equivalence to $\AV_w$. It follows from Theorem~\ref{thm:truncatedfullyfaithful} that,  for all $\pi \in W$, the abelian variety $A_\pi$ must be an injective cogenerator for $\AV_\pi$ with a commutative ring of endomorphisms.
Now the claim follows from Proposition~\ref{prop:Weilnumbers commutative injective cogenerator}.
\end{proof}

\subsection{Lattices of optimal rank}
A noticeable feature of the main result Theorem~\ref{MainThm} is the necessity to work with lattices whose rank is a multiple of the first Betti number. 

\begin{prop}\label{prop:constantC} 
Let $\Lambda:\AV_{\bF_q} \to \Mod_{\Ztf}(\bZ)$ be an additive, contravariant functor that attaches to any abelian variety over $\bF_q$
a free $\bZ$-module of finite rank. Assume that there is a constant $\gamma>0$ such that
\[
\rank_{\bZ}(\Lambda(X))=\gamma \cdot \dim(X)
\]
for any $X$ in $\AV_{\bF_q}$. Then $\gamma$ is divisible by $2\lcm(\expo,2)$.
\end{prop}
\begin{proof}
Let $\pi$ be a Weil $q$-number, and $B_\pi$ a simple object of $\AV_{\bF_q}$ associated to $\pi$.
The $\bQ$-vector space $\Lambda(B_\pi)\otimes\bQ$ defines, via functoriality of $\Lambda$, a right representation of $E_\pi=\End_{\bF_q}(B_\pi)\otimes\bQ$. 

Since $E_\pi$ is a division algebra, $E_\pi$ is the unique simple object  in the category of right representations of $E_\pi$. It follows that $\Lambda(B_\pi)\otimes\bQ$ is isomorphic to a multiple of $E_\pi$. Therefore
\[
\gamma \cdot \dim(B_\pi) = \dim_{\bQ} \Lambda(B_\pi) \otimes \bQ 
\]
is a multiple of (using Theorem~\ref{thm:TateEndostructure}~\ref{thmitem:dimfoemula}) 
\[
\dim_{\bQ} E_\pi = s_\pi^2 [\bQ(\pi):\bQ] = 2 s_\pi \dim(B_\pi).
\]
Hence $\gamma$ is divisible by the least common multiple of all $2s_\pi$ for all Weil $q$-numbers $\pi$. It was discussed in Section~\S\ref{sec:multiplicities} that $s_\pi$ always divides $\lcm(\expo,2)$. 

To complete the proof it is enough to show that 
\begin{itemize}
\item for any $\expo > 2$ there exists a Weil $q$-number $\pi'$ such that $s_{\pi'}=\expo$, and
\item for any $\expo \geq 1$ there exists a Weil $q$-number $\pi''$ with $s_{\pi''}=2$. 
\end{itemize}
The first statement is the content of Lemma~\ref{lem:attaine} below. The second one
is proved after checking that any root $\pi''$ of $x^2-q$ satisfies the required properties.
\end{proof}

\begin{lem}\label{lem:attaine} 
Let $\expo>2$, and let $\pi$ be a root of the polynomial $f(x)=x^2-px+q$, which is irreducible in $\bQ[x]$.
Then $\pi$ defines a Weil $q$-number such that the index $s_\pi$ of $E_\pi$ is $\expo$.
\end{lem}
\begin{proof} 
The discriminant of $f(x)$ is $p^2-4q<0$. Thus $f(x)$ is irreducible with two complex conjugate roots and hence $\pi$ is a Weil $q$-number. 

In order to compute $s_\pi$ we must compute the order of the local invariants at $p$-adic places as in Theorem~\ref{thm:TateEndostructure}~\ref{thmitem:invariantatp}. The assumption $\expo > 2$ ensures that the Newton polygon of $x^2 - px + q$ with respect to the $p$-adic valuation $v_p$ has two different (negative) slopes $1$ and $\expo-1$. Therefore $f(x)$ splits in $\bQ_p[x]$ into a product of two distinct linear factors corresponding to two distinct primes of $\bQ(\pi)$ above $p$. Moreover, by comparing with the slopes, we can choose a prime $\fp \mid p$  such that  $v_\fp(\pi) = 1$. 
By Theorem~\ref{thm:TateEndostructure}~\ref{thmitem:invariantatp}, the local invariant of the division ring
$E_\pi$ at $\fp$ is $1/\expo \pmod{1}$, which suffices to ensure $s_\pi=\expo$.
\end{proof}

\section{Abelian varieties isogenous to a power of an elliptic curve} 
\label{sec:multiplicity}

Let $E$ be an elliptic curve over $\bF_q$ and $\pi=\pi_E$ the associated Weil $q$-number.
The category $\AV_\pi$ defined in Section~\S\ref{subsection:notation}
is the full subcategory of $\AV_q$ whose objects are the abelian varieties over $\bF_q$ isogenous to a power of $E$.
This category is the object of study of \cite{JKPRSBT}.
In this stimulating paper the authors give a characterization of those elliptic curves $E$ for which the functor
\[
T_{E} \colon  \AV_\pi  \longrightarrow \Mod_\Ztf(\End_{\bF_q}(E)), \qquad T_{E}(X) = \Hom_{\bF_q}(X,E)
\]
is an anti-equivalence of categories. Their main result \cite[Theorem~1.1]{JKPRSBT} says that $T_E$ is an anti-equivalence precisely in the following cases (recall $\expo = [\bF_q: \bF_p]$):
\begin{itemize}
\item $\pi$ is ordinary and $\End(E) = R_\pi$;
\item $\pi$ is supersingular, $\expo = 1$, and $\End(E) = R_\pi$, or
\item  $\pi$ is supersingular, $\expo = 2$, and $\End(E)$ is  of $\bZ$-rank $4$.
\end{itemize}

\begin{rmk}
\label{rmk:supersingularEhasreducedbalanced}
This result provides the complete list of Weil numbers $\pi$ for which an injective cogenerator of dimension $1$ for 
$\AV_\pi$ exists. 
\end{rmk}

For those supersingular Weil numbers $\pi$ associated to an elliptic curve that are left out from the treatment of \cite{JKPRSBT}, it is natural to ask what is the minimal dimension of an injective cogenerator of $\AV_\pi$.
To this purpose we draw the following consequence of Theorem~\ref{thm:connectedSpec classification of locallyprojective}.

\begin{cor}
Let $\pi$ be a  Weil number corresponding to an isogeny class of supersingular elliptic curves over $\bF_q$. Then 
\[
\widebar{m}_\pi = \begin{cases}
\expo & \text{ if  $\pi \notin \bQ$}, \\
\expo/2 & \text{ if $\pi \in \bQ$}
\end{cases}
\]
is an integer and equals   the minimal dimension of an object $A \in \AV_\pi$
such that the functor
\[
T_A \colon  \AV_\pi  \longrightarrow \Mod_\Ztf(\End_{\bF_q}(A)), \qquad T_A(X) = \Hom_{\bF_q}(X,A),
\]
is an anti-equivalence of categories.
\end{cor}
\begin{proof}
The formula for $\widebar{m}_\pi $ follows from Theorem~\ref{thm:TateEndostructure}~\ref{thmitem:dimfoemula}.
If $\pi$ is rational, then $\expo$ is even, hence  $\widebar{m}_\pi$ is a natural number. 

Let $\widebar{A}_\pi$ be a reduced $\pi$-balanced abelian variety, see Theorem~\ref{thm:integralstructure} for the construction. Since $\Spec(R_\pi)$ is connected and $\pi$ is supersingular, Theorem~\ref{thm:connectedSpec classification of locallyprojective} and Theorem~\ref{thm:truncatedfullyfaithful} show that any injective cogenerator for $\AV_\pi$ is isogenous to a power of $\widebar{A}_\pi$. Because $\widebar{A}_\pi$ is $\pi$-locally projective, it is an injective cogenerator by Theorem~\ref{thm:truncatedfullyfaithful}.  In particular, the minimal dimension of an injective cogenerator is equal to $\dim(\widebar{A}_\pi)$, which equals $\widebar{m}_\pi\cdot\dim(B_\pi)$. Since $B_\pi$ is $1$-dimensional the corollary follows. 
\end{proof}

We conclude the section and the paper with the detailed analysis of an example of a Weil number $\pi$ corresponding to a supersingular elliptic curve where  the reduced
$\pi$-balanced object $\widebar{A}_\pi$ is $2$-dimensional. In this example, all local invariants of $E_\pi=\End_{\bF_q}(B_\pi)\otimes\bQ$ are trivial and the
endomorphism ring of $\widebar{A}_\pi$ is a non-maximal order of $\rM_2(\bQ(\pi))$ that will be computed explicitly.

\smallskip

Let $p$ be a prime number such that $p \equiv 3 \pmod 4$ and let $r=2$, so that $q=p^2$. Let 
$\pi=ip$ be the Weil $q$-number
whose minimal polynomial is $x^2+p^2$. Let $B_\pi$ be any $\bF_q$-simple abelian variety over $\bF_q$ associated to $\pi$ via Honda-Tate theory. Since  $p$ is inert in the quadratic field $\bQ(i)= \bQ(\pi)$, we deduce from Theorem~\ref{thm:TateEndostructure}~\ref{thmitem:invariantatp} that all local invariants of $E_\pi$ are trivial. Hence  
\[
E_\pi = \End_{\bF_q}(B_\pi)\otimes\bQ=\bQ(\pi)=\bQ(i)
\]
is commutative and $B_\pi$ is an elliptic curve by Theorem~\ref{thm:TateEndostructure}~\ref{thmitem:dimfoemula}.
Since the trace $\pi + q/\pi$ 
of $\pi$ is zero, the elliptic curve $B_\pi$ is supersingular, however, not all of its geometric endomorphisms are defined over $\bF_q$. The minimal central order associated to $\{\pi\}$ is 
\[
R_\pi = \bZ[\pi,q/\pi] = \bZ[ip],
\]
which has index $p$ in the maximal order $\bZ[i]$. In particular, $R_\pi$ is not maximal at $p$. On the other hand, any such supersingular elliptic curve $B_\pi$ has an endomorphism ring which is maximal at $p$, cf.\ \cite[Theorem~4.2(3)]{Wa}. Thus  the  inclusion of $R_\pi$ in $\End_{\bF_q}(B_\pi)$  is proper and more precisely
\[
\bZ[ip] = R_\pi\subsetneq \End_{\bF_q}(B_\pi) = \bZ[i].
\]
This shows that the functor 
\[
T_{B_\pi} \colon \AV_\pi \longrightarrow \Mod_{\Ztf}(\End_{\bF_q}(B_\pi)), \qquad  T_{B_\pi}(X)  = \Hom_{\bF_q}(X,B_\pi)
\]
is not an anti-equivalence, for otherwise $\End_{\bF_q}(B_\pi)$ would have to agree with $R_\pi$ (see Proposition~\ref{prop:centerofinjectivecogenerators}). More directly, we can use that $\End_{\bF_q}(B_\pi) = \bZ[i]$ is a principal ideal domain and therefore $\Mod_{\Ztf}(\End_{\bF_q}(B_\pi))$ has a unique simple object. On the other hand, it follows from \cite[Theorem~5.1]{Wa} that $\AV_\pi$ has two non-isomorphic objects of dimension $1$, which therefore are both simple objects (see also Remark~\ref{rmk:whyBpiNotInjectiveCogenerator}). That also shows that $T_{B_\pi}(-)$ is not an anti-equivalence. 

\smallskip

The functor $\Hom_{\bF_q}(-,B_\pi)$ has only a problem locally at $p$ that prevents it from being fully faithful. We will calculate Dieudonn\'e modules for $B_\pi$, for its Forbenius twist 
\[
B_\pi^{(p)} = B_\pi \times_{\bF_q, \Frob} \bF_q
\]
and  for a choice of a reduced balanced $\bar A_\pi$  in order to show what the obstacle is and how $\bar A_\pi$ circumvents the local problem at $p$. 
Let $K=\bQ_p(\pi)$ be the completion of $\bQ(\pi)$ at the unique prime above $p$, fix an embedding $W(\bF_q) \inj K$ and denote by $\fo_K$ the image of $W(\bF_q)$.
The non-trivial Galois automorphism of $K/\bQ_p$ will be denoted by $a \mapsto \bar a$.

\begin{lem}
\label{lem:explicitdieudonne pi equal to pi}
There is an isomorphism of $K$-algebras
$\psi :  \cD_{\pi}^0 \xrightarrow{\sim} \rM_2(K)$,
\[
\psi(a) = 
\begin{pmatrix}
a&0\\
0&\bar a
\end{pmatrix}
\text{ for all $a \in W(\bF_q)$,  } \quad \psi(\cF) =  
\begin{pmatrix}
0&1\\
ip&0\\
\end{pmatrix}
\text{ and } \quad  \psi(\cV) = 
\begin{pmatrix}
0&-i\\
p&0\\
\end{pmatrix}.
\]
The integral Dieudonn\'e ring $\cD_\pi$ is identified by $\psi$ with the $\fo$-subalgebra 
\begin{equation}
\label{eq:psiDpi}
\psi(\cD_\pi) = \left\{\begin{pmatrix}
a&b\\
c&d\\
\end{pmatrix} \in \rM_2(\fo_K) \ ; \ p|c \text{ and } \ a  \equiv  \bar d \pmod p \right\}
\end{equation}
of index $p^4$ in the maximal order $\rM_2(\fo_K)$.
\end{lem}
\begin{proof}
The minimal polynomial of $\pi = ip$ is $P_\pi(x) = x^2 + p^2$. As defined in Section~\S\ref{sec:mincenord}, it follows that $h_\pi(F,V) = F+V$, and \eqref{eq:definitionDw} implies 
the description 
\[
\cD_\pi^0 = W(\bF_{q})\{\cF,\cV\}/\cF\cV - p, \cF^2+\cV^2) \otimes_{\bZ_p} \bQ_p
\]
in terms of generators and relations. A computation shows that $\psi$ is well defined. Since both $\cD_\pi^0$ and $\rM_2(K)$ are central simple algebras over $K$ of degree $2$, the map $\psi$ must be an isomorphism. 

It remains to compute the image $\psi(\cD_\pi)$. The elements $\cV, 1, \cF$ and $\cF^2$ span $\cD_\pi$ as a left $W(\bF_{q})$-module. Hence $\psi(\cD_\pi)$ is the set of 
\[
\psi(u\cV+ x \cdot 1 + v \cF + y \cF^2) = \begin{pmatrix}
x+ yip & v - i u \\
p (-\bar u + i \bar v)  & \bar x + \bar y ip 
\end{pmatrix}
\]
for all $x,u,y,v \in W(\bF_{q})$. Now the claim follows from
\[
\left\{ \binom{x+yip}{\bar x + \bar y ip} \ ; \ x,y \in W(\bF_{q})\right\} = \left\{ \binom{a}{d} \ ; \ a,d \in W(\bF_{q}), \ a \equiv \bar d \pmod p\right\},
\]
\[
\left\{ \binom{v - i u}{-\bar u + i \bar v} \ ; \ u,v \in W(\bF_{q})\right\} = \left\{ \binom{b}{c'} \ ; \ b,c' \in W(\bF_{q})\right\}. \qedhere
\]
\end{proof}

Via the embedding $\cD_\pi \subseteq \rM_2(K)$ 
induced by $\psi$ of Lemma~\ref{lem:explicitdieudonne pi equal to pi} 
we obtain two $\cD_\pi$-modules 
\[
 \Lambda_1 = \{\binom{a}{c} \in \fo_K^{\oplus 2} \ ; \ p \mid c\}, \qquad \Lambda_2 = \fo_K^{\oplus 2}, 
\]
as the first column, respectively the second column, of the matrix description. This results in the following embedding of $\cD_\pi$-modules.
\begin{equation}
\label{eq:dieudonnemoduleApi}
\cD_\pi \subseteq \Lambda_1 \oplus \Lambda_2 \subseteq \rM_2(\fo_K) \simeq \Lambda_2 \oplus \Lambda_2.
\end{equation}

\begin{lem}
Up to $K^\times$-homothety, there are only two $\cD_\pi$-lattices in $K^{\oplus 2}$, namely $\Lambda_1$ and $\Lambda_2$.
\end{lem}
\begin{proof}
Since we have 
\[
p\rM_2(\fo_K) \subseteq \cD_\pi \subseteq \rM_2(\fo_K),
\]
any $\cD_\pi$-lattice $\Lambda$ can by homothety be brought to a position 
\[
p \Lambda_2 \subseteq \Lambda \subseteq \Lambda_2.
\]
The choice of $\Lambda$ corresponds to a $\cD_\pi$-stable subgroup of $\Lambda_2/p\Lambda_2 = \fo_K^{\oplus 2}/p\fo_K^{\oplus 2} = \bF_q^{\oplus 2}$ on which $\cD_\pi$ acts through the reduction mod $p$, i.e.\ the image of 
\[
\cD_\pi \xrightarrow{\psi} \rM_2(\fo_K) \surj \rM_2(\fo_K/p\fo_K) \simeq \rM_2(\bF_q).
\]
This image consists of 
\[
\left\{\begin{pmatrix}
a&b\\
0&d\\
\end{pmatrix} \in \rM_2(\bF_q) \ ; \  a = \bar d \right\}.
\]
Therefore the only nontrivial invariant subspace in $\Lambda_2/p\Lambda_2$ is the image of $\Lambda_1$.
\end{proof}

Let $\Frob: B_\pi\to B_\pi^{(p)}$ be the $\bF_q$-isogeny given by the relative ($p$-)Frobenius morphism.  
We set $T_p(B_\pi) = \Lambda$ and identify $\Frob$ via Dieudonn\'e theory with the inclusion of $\cD_\pi$-modules
\[
T_p(B_\pi^{(p)}) = \cF\Lambda \subseteq \Lambda = T_p(B_\pi),
\]
which has a cokernel 
of $W(\bF_q)$-length $1$ corresponding to the kernel of Frobenius 
under the Dieudonn\'e functor for finite group schemes. 
Up to exchanging $B_{\pi}$ with its Frobenius twist, we may and will assume that $T_p(B_\pi) = \Lambda_2$ and $T_p(B_\pi^{(p)}) = \Lambda_1 \subseteq \Lambda_2$.
With these identifications in place, the $\cD_\pi$-embedding $\cD_\pi \inj \Lambda_1 \oplus \Lambda_2$ from \eqref{eq:dieudonnemoduleApi} has cokernel 
of $W(\bF_q)$-length $1$ and determines an isogeny
\[
B_\pi^{(p)} \times B_\pi \longrightarrow \bar A_\pi
\]
of degree $p$ with $T_p(\bar A_\pi) = \cD_\pi$. 

\begin{prop}
The abelian variety $\bar A_\pi$ constructed above is a reduced $\pi$-balanced object in $\AV_\pi$ and hence an injective cogenerator.
\end{prop}
\begin{proof}
By construction, $T_p(\bar A_\pi) = \cD_\pi$ is a free $\cD_\pi$-module of rank $1$.
For $\ell \not= p$, the $R_\pi \otimes \bZ_\ell$-module $T_\ell(\bar A_\pi)$ is free of rank $2$ because  $R_\pi = \bZ[ip]$ and the localization $R_\pi[\frac{1}{p}] = \bZ[i,\frac{1}{p}]$ is a Dedekind ring, hence $R_\pi$ is regular away from $p$. 
\end{proof}

\begin{rmk}
\label{rmk:whyBpiNotInjectiveCogenerator}
We may formulate the failure for $B_\pi$ to be an injective cogenerator as follows. There are failures in two steps:
\begin{enumerate}[label=(\arabic*),align=left,labelindent=0pt,leftmargin=*,widest = (8)]
\item
There are two non-isomorphic $\cD_\pi$ lattices $\Lambda_1$ and $\Lambda_2$ that represent non-isomorphic simple objects in $\AV_\pi$, while the target category of $\End_{\bF_q}(B_\pi)$-modules has only one isomorphism type of simple objects that those can be mapped to. 
\item 
The reduced $\pi$-balanced abelian variety $\bar A_\pi$ does a better job because first of all it combines both $\cD_\pi$-lattices. But $B_\pi^{(p)} \times B_\pi$ also does this, and still $\Hom_{\bF_q}(-,B_\pi^{(p)} \times B_\pi)$ is not an anti-equivalence, because otherwise $B_\pi^{(p)} \times B_\pi$ was $\pi$-locally projective and then the same would apply to the direct factor $B_\pi$, contradiction.
\item
The decisive improvement of $\bar A_\pi$ over the product $B_\pi^{(p)} \times B_\pi$ comes from choosing an appropriate sublattice $\cD_\pi$ in the product $\Lambda_1 \oplus \Lambda_2$. Now, the abelian variety $\bar A_\pi$ has $T_p(\bar A_\pi) = \cD_\pi$ and so is also locally projective at $p$. The construction of the sublattice can be described in terms of a congruence as follows.  Although $\Lambda_1$ is not isomorphic to $\Lambda_2$ as $\cD_\pi$-modules, the Frobenius on $\bF_q$ induces an isomorphism of finite $\cD_\pi$-modules
\begin{align*}
\Lambda_1/\cF\Lambda_1  \xrightarrow{\sim} \bF_q & \xrightarrow{x \mapsto \bar x} \bF_q \xleftarrow{\sim}   \Lambda_2/\cF\Lambda_2 
\\
\binom{a}{c} + \cF \Lambda_1  \mapsto a +  p\fo_K & \  \longmapsto \  d \equiv \bar{a}  + p\fo_K \mapsfrom \binom{b}{d} + \cF \Lambda_2 .
\end{align*}
This congruence yields a description of $\cD_\pi$ as a fiber product of $\cD_\pi$-modules
\[
\xymatrix@M+1ex{
\cD_\pi \ar[r] \ar[d] & \Lambda_2 \ar[d]^{\binom{b}{d} \mapsto \bar d} \\
\Lambda_1 \ar[r]^{\binom{a}{c} \mapsto a} & \bF_q ,
}
\]
and in some sense it is this coincidence of a congruence between $\Lambda_1$ and $\Lambda_2$ as $\cD_\pi$-lattices which endows $T_p(\bar A_\pi)$ and a posteriori $\bar A_\pi$ with its remarkable properties.
\end{enumerate}
\end{rmk}

We conclude the example by computing $S_\pi = \End_{\bF_q}(\bar A_\pi)$.

\begin{prop}
\label{prop:endoApiexample}
There is an isomorphism 
\[
S_\pi  = \End_{\bF_q}(\bar A_\pi) \simeq
\left\{\begin{pmatrix}
a&b\\
c&d\\
\end{pmatrix} \in \rM_2(\bZ[i]) \ ; \ p|c \text{ and } \ a  \equiv \bar d \pmod p \right\}.
\]
\end{prop}

\begin{proof}
This follows by the proof of Theorem~\ref{thm:Sw} applied to the isogeny of $p$-power degree
\[
\ph: B_\pi \times B_\pi \xrightarrow{\Frob \times \id} B_\pi^{(p)} \times B_\pi \to \bar A_\pi. \qedhere
\]
\end{proof}

As a consequence of Proposition~\ref{prop:endoApiexample} the category $\AV_\pi$ is anti-equivalent to the category
of modules over the non maximal order $S_\pi$ of $\rM_2(\bQ(i))$ which are  finite and free as $\bZ$-modules.


\end{document}